\PassOptionsToPackage{lowtilde}{url}		
\documentclass[11pt]{amsart}

\usepackage[english]{babel}
\usepackage{amsfonts}
\usepackage{amsmath} 
\usepackage{amssymb}
\usepackage{amsthm}
\usepackage{abraces}
\usepackage{mathtools}                          
\usepackage{mathrsfs}				    
\usepackage{verbatim}				    
\usepackage[all]{xy}    	          
\SelectTips{cm}{11}                  
\usepackage[T1]{fontenc}        
\usepackage{graphicx}	
\usepackage[normalem]{ulem}          
\usepackage{microtype}               
\usepackage{tikz-cd}
\usepackage{graphicx, subcaption}	

\usepackage{enumitem}
\setenumerate[1]{label=(\roman*),ref=\thesubsection.{\roman*}}    

\usepackage{fullpage}

\usepackage{hyperref}
\hypersetup{
  colorlinks   = true,          
  urlcolor     = blue,          
  linkcolor    = blue,          
  citecolor   = red             
}

\usepackage{pdfsync}

\newcommand{\A}{{\mathbb{A}}}
\newcommand{\N}{{\mathbb{N}}}
\newcommand{\Z}{{\mathbb{Z}}}
\newcommand{\Q}{{\mathbb{Q}}}
\newcommand{\R}{\mathbb{R}}			
\newcommand{\C}{\mathbb{C}}
\newcommand{\F}{{\mathbb F}}
\newcommand{\G}{{\mathbb G}}

\newcommand{\Hx}{{\mathcal H(x)}}

\renewcommand{\AA}{{\mathbb A}}
\newcommand{\PP}{{\mathbb P}}
\newcommand{\aank}{{\mathbb{A}^{1,\mathrm{an}}_k}}	
\newcommand{\pank}{{\mathbb{P}^{1,\mathrm{an}}_k}}	
\newcommand{\pana}[1]{{\mathbb{P}^{1,\mathrm{an}}_{#1}}}	

\newcommand{\calA}{{\mathcal A}}
\newcommand{\calB}{{\mathcal B}}
\newcommand{\calC}{{\mathcal C}}

\newcommand{\calG}{{\mathcal G}}
\newcommand{\calH}{{\mathcal H}}
\newcommand{\calI}{{\mathcal I}}

\newcommand{\calL}{{\mathcal L}}
\newcommand{\calM}{{\mathcal M}}

\newcommand{\calO}{{\mathcal O}}

\newcommand{\calS}{{\mathcal S}}
\newcommand{\calT}{{\mathcal T}}

\newcommand{\frakF}{{\mathfrak F}}
\newcommand{\frakB}{{\mathfrak B}}

\newcommand{\id}{\mathrm{id}}
	
\newcommand{\an}{\mathrm{an}}
\newcommand\wc{{\mkern 2mu\cdot\mkern 2mu}}
\newcommand\va{|\wc|}
\newcommand\nm{\|\wc\|}
\newcommand\eps{\varepsilon}
\renewcommand\le{\leqslant}
\renewcommand\ge{\geqslant}
\newcommand\Mumf{\mathrm{Mumf}}
\newcommand\MCG{\mathrm{MCG}}

\newcommand{\too}{\longrightarrow}
\newcommand{\mapstoo}{\longmapsto}

\DeclarePairedDelimiter\abs{\lvert}{\rvert}

\DeclareMathOperator{\Hom}{Hom}

\DeclareMathOperator{\Aut}{Aut}

\DeclareMathOperator{\Out}{Out}
\DeclareMathOperator{\Gal}{Gal}
\DeclareMathOperator{\GL}{GL}
\DeclareMathOperator{\PGL}{PGL}

\DeclareMathOperator{\Frac}{Frac}	

\DeclareMathOperator{\pr}{pr}
\DeclareMathOperator{\tr}{tr}
\DeclareMathOperator{\Stab}{Stab}

\newtheoremstyle{plain2}    
  {}            
  {}            
  {\itshape}    
  {}            
  {\bfseries}   
  {.}           
  {5pt plus 1pt minus 1pt}  
  {{\thmnumber{(#2)} \thmname{#1}{\thmnote{ (#3)}}}}          
\newtheorem{theorem}{Theorem}[subsection]
\newtheorem{theorem*}{Theorem}
\newtheorem{corollary}[theorem]{Corollary}
\newtheorem{lemma}[theorem]{Lemma}
\newtheorem{proposition}[theorem]{Proposition}

\newtheoremstyle{definition2}    
  {}   
  {}   
  {\normalfont}  
  {}       
  {\bfseries} 
  {.}        
  {5pt plus 1pt minus 1pt} 
  {{(\thmnumber{#2}) \thmname{#1}{\thmnote{#3}}}}          

\theoremstyle{definition}
\newtheorem{definition}[theorem]{Definition}
\newtheorem{notation}[theorem]{Notation}
\newtheorem{example}[theorem]{Example}

\theoremstyle{remark}
\newtheorem{remark}[theorem]{Remark}

\newtheoremstyle{stepstyle}
  {}     {}   
  {\normalfont}  
  {\parindent}       
  {\itshape} 
  {}         
  {5pt plus 1pt minus 1pt} 
  {{\thmname{#1} \thmnumber{#2}:{\thmnote{#3}}}}          
\theoremstyle{stepstyle}

\newtheoremstyle{point}
  {}     {}   
  {\normalfont}  
  {}       
  {\bfseries} 
  {}         
  {5pt plus 1pt minus 1pt} 
  {{\thmname{#1}\thmnumber{#2}.\thmnote{ #3.}}}          
\theoremstyle{point}
\newtheorem{point}[subsection]{}

\newtheoremstyle{point*}
  {}     {}   
  {\normalfont}  
  {}       
  {\bfseries} 
  {}         
  {5pt plus 1pt minus 1pt} 
  {{\thmname{#1}\thmnote{ #3.}}}          
\theoremstyle{point*}
\newtheorem{point*}[subsubsection]{}

\numberwithin{equation}{subsection}

\newtheoremstyle{subpoint}
  {}     {}            
  {\normalfont}  
  {}                   
  {\normalfont} 
  {}         
  {5pt plus 1pt minus 1pt} 
  {{\thmname{#1}(\thmnumber{#2})\thmnote{ #3.}}}          
\theoremstyle{subpoint}
\newtheorem{subpoint}[equation]{}

\title{Schottky spaces and universal Mumford curves over $\Z$}

\author{J\'er\^ome Poineau}
\address{Universit\'e de Caen - Normandie}
\email{\href{mailto:jerome.poineau@unicaen.fr}{jerome.poineau@unicaen.fr}}
\urladdr{\url{https://poineau.users.lmno.cnrs.fr/}}

\author{Daniele Turchetti}
\address{University of Warwick}
\email{\href{mailto:daniele.turchetti.math@gmail.com}{daniele.turchetti.math@gmail.com}}
\urladdr{\url{https://www.mathstat.dal.ca/~dturchetti/}}

\usepackage{etoolbox}
\makeatletter
\patchcmd{\@maketitle}
  {\ifx\@empty\@dedicatory}
  {\ifx\@empty\@date \else {\vskip3ex \centering\footnotesize\@date\par\vskip1ex}\fi
   \ifx\@empty\@dedicatory}
  {}{}
\patchcmd{\@adminfootnotes}
  {\ifx\@empty\@date\else \@footnotetext{\@setdate}\fi}
  {}{}{}
\makeatother

\begin{document}

\begin{abstract}
For every integer $g \geq 1$ we define a universal Mumford curve of genus $g$ in the framework of Berkovich spaces over $\Z$.
This is achieved in two steps: first, we build an analytic space $\calS_g$ that parametrizes marked Schottky groups over all valued fields.
We show that $\calS_g$ is an open, connected analytic space over $\Z$.
Then, we prove that the Schottky uniformization of a given curve behaves well with respect to the topology of $\calS_g$, both locally and globally.
As a result, we can define the universal Mumford curve $\calC_g$ as a relative curve over $\calS_g$ such that every Schottky uniformized curve can be described as a fiber of a point in $\calS_g$.
We prove that the curve $\calC_g$ is itself uniformized by a universal Schottky group acting on the relative projective line $\mathbb{P}^1_{\calS_g}$.
Finally, we study the action of the group $\Out(F_g)$ of outer automorphisms of the free group with $g$ generators on $\calS_g$, describing the quotient $\Out(F_g) \backslash \calS_g$ in the archimedean and non-archimedean cases.
We apply this result to compare the non-archimedean Schottky space with constructions arising from geometric group theory and the theory of moduli spaces of tropical curves.
\end{abstract}

\maketitle

\setcounter{tocdepth}{1}
\tableofcontents

\section{Introduction}
The uniformization of Riemann surfaces is one of the most central results in the theory of analytic curves. Independently proven in 1907 by P. Koebe and H. Poincar\'e, this theorem is the culmination of almost a century of work in complex geometry\footnote{A complete account of the results and the mathematicians that made this breakthrough possible is given with detailed proofs in the impressive collective work \cite{SaintGervaisBook}.}, and states that the universal cover of a connected compact complex analytic curve of genus $g$ is analytically isomorphic to:
\begin{itemize}
\item[-] the projective line if $g=0$;
\item[-] the affine line if $g=1$;
\item[-] the open unit disc if $g\geq 2$.
\end{itemize}
This fact bears important consequences in numerous fields of mathematics, such as differential equations, number theory (special functions, modular forms, elliptic curves), Kleinian and Fuchsian group representations, and the theory of algebraic curves and their fundamental groups.
After proving the aforementioned theorem, Koebe went on to show several related results that clarified many aspects of uniformization theory.
The most notorious one is the \emph{retrosection theorem}, stating that every connected compact complex analytic curve is the quotient of an open dense subset $O \subset \C$ by the action of a free, finitely generated subgroup $\Gamma$ of $\PGL_2(\C)$ for which $O$ is the region of discontinuity. Any group $\Gamma$ arising in this way is called a \emph{Schottky group}, and the resulting theory is usually referred to as \emph{Schottky uniformization} of Riemann surfaces.

The quest for a non-archimedean analogue of uniformization theory is at the heart of the establishment of non-archimedean analytic geometry.
In fact, J. Tate developed his theory of rigid analytic geometry to show that elliptic curves with split multiplicative reduction over a non-archimedean field $k$ are always of the form $\G_m / q^{\Z}$ for some element $q\in k$ such that $0<|q|<1$.
In the early '70s, D. Mumford had the intuition that Tate's uniformization could be extended to higher genus curves by building a non-archimedean theory of Schottky uniformization.
In his celebrated paper \cite{Mumford72}, he defined a notion of non-archimedean Schottky group and, using tools from formal geometry and Bruhat-Tits theory, showed that any such group acts on a suitable open subset of the projective line, and that the quotient identifies to (the analytification of) a projective curve. Finally, he characterized the curves that admit Schottky uniformization according to their reduction type. A further development of this theory in the context of rigid geometry was subsequently carried out by several authors, most notably by L.~Gerritzen and M.~van der Put in their book \cite{GerritzenPut80}.

The non-archimedean theory is remarkable in many aspects, and particularly for the ways it addresses the issue of the lack of a nice topological structure: non-archimedean fields are totally disconnected, hence it is highly nontrivial even to define what an open dense subset of the projective line in this context should be.
The solution of Tate's rigid geometry, in the late '50s, is to consider non-archimedean spaces with a Grothendieck topology instead of a classical topology.
A more modern approach, developed by V.~Berkovich in the late '80s, consists in defining non-archimedean analytic spaces as spaces of absolute values. Since they contain many points beyond the classical ones, they are not easily described explicitly, but they may be endowed with a structure of topological space in the classical sense.
Given a non-archimedean field $k$, the group $\PGL_2(k)$ acts naturally on the Berkovich projective line $\pank$ and one can describe Schottky uniformization of Mumford curves in complete analogy with the complex case. 
A study of classical results on Schottky groups and Mumford curves in this new framework was initiated in Berkovich's first monograph (see \cite[\S 4.4]{Berkovich90}) and has been expanded in recent work by the authors \cite{VIASMII}, whose appendix contains a description of striking applications of these results.

\medbreak

In this paper, we are interested in the interplay between the archimedean and non-archimedean theories of Schottky uniformization.
In order to construct a rigorous and coherent common framework for these two theories, we adopted the viewpoint of ``Berkovich spaces over $\Z$''. 
Let us introduce it in a few words.

Although the theory of Berkovich analytic geometry was originally geared towards non-archimedean spaces, that is to say spaces over non-archimedean valued fields, it is worth noting that Berkovich's original definitions apply under less restrictive assumptions, and allow to make sense of analytic spaces over arbitrary Banach rings. For example, one may consider the base ring $(\C,\va_{\infty})$, where $\va_{\infty}$ denotes the usual absolute value. The Berkovich analytic spaces obtained in this way are nothing but the familiar complex analytic spaces.

Note that the field~$\C$ may be endowed with other absolute values. By considering a power~$\va_{\infty}^\eps$ of the usual absolute value, with $\eps\in (0,1]$, one still obtains a theory completely parallel to the complex analytic theory, but with a different normalization. Passing to the limit when $\eps$ tends to~$0$, one is led to consider the field~$\C$ endowed with the trivial absolute~$\va_{0}$ (defined by $\abs{a}_{0} = 0$ if $a = 0$ and $\abs{a}_{0} = 1$ otherwise), which is a non-archimedean field, hence belongs to the realm of the usual theory of Berkovich spaces. 

By using the theory of Berkovich spaces over Banach rings, one may fit all the preceding spaces into a common one. 
To be more precise, let us endow the field~$\C$ endowed with the (non-multiplicative) hybrid norm $\nm_{\textrm{hyb}} := \max(\va_{0},\va_{\infty})$. Its spectrum in the sense of Berkovich is $\calM_{\textrm{hyb}} := 
\{\va_{0}\} \cup \{\va_{\infty}^\eps, 0< \eps\le 1 \}$, and every Berkovich space over $(\C,\nm_{\textrm{hyb}})$ admits a natural morphism~$\pr_{\textrm{hyb}}$ to~$\calM_{\textrm{hyb}}$. As one may expect,  the fibers $\pr_{\textrm{hyb}}^{-1}(\va_{\infty}^\eps)$ are complex analytic spaces, whereas the fiber $\pr_{\textrm{hyb}}^{-1}(\va_{0})$ is a non-archimedean Berkovich space. As a result, such a hybrid space witnesses complex analytic spaces converging towards a non-archimedean space. This rough idea actually leads to concrete results and has allowed to investigate precisely various properties of degeneration of families of complex spaces such as mixed Hodge structures \cite{BerkovichW0}, volume forms \cite{BoucksomJonsson}, equilibrium measures of endomorphisms \cite{FavreEndomorphisms}, etc. It also found a striking arithmetic application to uniform Manin-Mumford bounds for a family of genus~2 curves in~\cite{DKY}. From this same point of view, in this paper, we will observe complex curves with a Schottky uniformization converging to non-archimedean curves with a Mumford uniformization.

It is possible to push further this line of thought and consider not only families of complex spaces and their non-archimedean limits (lying over a trivially valued field), but even $p$-adic spaces and, more generally, spaces over arbitrary valued fields. To this end, one starts with the base ring $(\Z,\va_{\infty})$ and consider the space~$\calM(\Z)$ of all absolute values (and more generally multiplicative seminorms) on~$\Z$. Thanks to Ostrowski's theorem, up to a power, the latter are known to be exactly the usual one~$\va_{\infty}$, the $p$-adic one~$\va_{p}$, where~$p$ is a prime number, the trivial one~$\va_{0}$ and the multiplicative seminorm~$\va_{p,0}$ induced by the trivial absolute value on~$\F_{p}$, where $p$ is a prime number and~$\F_{p}$ is the finite field with $p$~elements. By the same reasoning as in the hybrid case, Berkovich analytic spaces over~$\Z$ are global objects that naturally admit a morphism~$\pr_{\Z}$ to~$\calM(\Z)$. The fibers or~$\pr_{\Z}$ may either be complex analytic spaces, $p$-adic analytic spaces, or spaces over trivially valued fields. This is the framework we will use to carry out our study of analytic uniformization of curves in a uniform way. It will allow us to build a space parametrizing all uniformizable curves (or, equivalently, Schottky groups) over all possible valued fields: archimedean or not, of arbitrary characteristic and arbitrary residue characteristic.


The foundations of the theory of analytic spaces over Banach rings
were laid by Berkovich at the beginning of his manuscript~\cite{Berkovich90}, but he soon switched more specifically to non-archimedean spaces. Over base rings of a special type (including, among others, $\Z$, rings of integers of numbers fields, and the hybrid~$\C$), the theory was then further developed by the first-named author in~\cite{A1Z} (case of the affine line), \cite{Poineau13} (local algebraic properties such as Noetherianity of the stalks or coherence of the structure sheaf) and~\cite{LemanissierPoineau20} with T.~Lemanissier (definition of the category of analytic spaces, local path-connectedness of the spaces, cohomological vanishing on disks). 
As an example of application, let us mention that Berkovich spaces over~$\Z$ were used in~\cite{Raccord} to give a geometric proof of a result of D.~Harbater \cite{HarbaterGaloisCovers} solving the inverse Galois problem over a ring of convergent arithmetic power series (the subring of $\Z[\![t]\!]$ consisting of power series that converge on the complex open unit disc).

\medbreak

Let us go back to the construction of a common framework for Schottky uniformization. 
Our first main new contribution in this direction is the definition of a moduli space $\calS_g$ of Schottky groups of rank $g$ as a Berkovich space over $\Z$.
This comes as a generalization of the \emph{Schottky spaces} that have been extensively investigated both over the complex numbers and over non-archimedean fields.
In order to simplify the notation, let us assume that $g\geq 2$ (the case $g=1$ is analogous in many aspects, though much simpler).
The complex Schottky space is defined by L.~Bers in~\cite{Bers75} as a complex submanifold of $\C^{3g-3}$ and parametrizes complex Schottky groups together with the choice of a basis up to conjugation by elements of $\PGL_2(\C)$.
Its properties, generalizations, and various applications have been studied ever since, for example in \cite{JorgensenMardenEtAl79}, \cite{GerritzenHerrlich88}, and \cite{Herrlich91}.
In analogy with the complex case, Gerritzen \cite{Gerritzen81}, \cite{Gerritzen82} gave a definition of Schottky space over a non-archimedean field $k$ as a rigid analytic subspace of the affine space of dimension $3g-3$ over~$k$, \textit{i.e.} $\bar{k}^{3g-3}/\Gal(\bar k/k)$, where $\bar{k}$ is an algebraic closure of~$k$. Both parametrizations are based on the notion of \textit{Koebe coordinates}: a hyperbolic element $\gamma$ of $\PGL_2(\C)$ or $\PGL_2(k)$ is uniquely determined by the datum of its attracting fixed point $\alpha$, its repelling fixed point $\alpha'$, and its multiplier $\beta$, and, conversely, any ordered triplet $(\alpha, \alpha', \beta)$ of elements of~$\C$ or~$k$ satisfying 
\[\begin{cases} \alpha \neq \alpha' \\ 0<|\beta|<1 \end{cases}\]
gives rise to a hyperbolic element of $\PGL_2(\C)$ or $\PGL_2(k)$. A Schottky group of rank~$g$ admits a basis consisting of $g$ elements, which gives rise to $3g$ Koebe coordinates. Using a M\"obius transformation to send the first three fixed points to~0, 1 and~$\infty$, which amounts to some normalization, reduces the number to $3g-3$ coordinates.

The space~$\calS_g$ defined in this paper 
is a subset of an affine Berkovich space over~$\Z$, namely~$\AA^{3g-3,\an}_{\Z}$. 
Recall that each point~$x$ in~$\AA^{3g-3,\an}_\Z$ naturally determines a complete valued field~$\calH(x)$, its \textit{complete residue field}, as well as $3g-3$ elements in it (obtained by evaluating the coordinate functions at~$x$).
The space~$\calS_{g}$ then consists in the points~$x$ in~$\AA^{3g-3,\an}_\Z$ whose associated $(3g-3)$-tuple of elements in~$\calH(x)$ corresponds, up to normalization, to the Koebe coordinates of a basis of a Schottky group of rank~$g$ in~$\PGL_{2}(\calH(x))$ (see Definition \ref{def:Schottkyspace} for more precision).
This definition allows us to retrieve both the complex Schottky space and the Berkovich analogue of Gerritzen's space, as the complete residue fields~$\calH(x)$ may be complex, $p$-adic, etc.
We first establish some topological properties of the Schottky space over $\Z$.

\begin{theorem*}[\protect{Theorems~\ref{thm:open} and~\ref{thm:pathconn}}]\label{introthm:open}
The space $\calS_g$ is open in $\AA^{3g-3,\an}_{\Z}$ and path connected.
\end{theorem*}

As a result of its openness, the space $\calS_g$ inherits a structure of analytic space. In other words, the Mumford and Schottky uniformizations naturally fit together into a well-behaved family. Under this new perspective, Mumford's construction appears to be more than a mere analogue of the complex construction: both are particular instances of a more general theory.

The properties of the space~$\calS_{g}$ allow to move continuously between the archimedean and the non-archimedean constructions, as in the case of hybrid spaces mentioned before.
Similarly, there is a phenomenon of continuous degeneration from $p$-adic Schottky spaces to Schottky spaces in characteristic $p$. 
We believe that these features could be exploited to build interesting partial compactifications of Schottky spaces and to establish analogies between the theory of Mumford curves in mixed characteristic and in positive characteristic.

\medbreak

We then turn to the question of the connection between the space $\calS_g$ and the moduli space of curves.
Let $\Out(F_g)$ be the group of outer automorphisms of the free group with $g$ generators.
There is a natural action of~$\Out(F_g)$ on~$\calS_g$, whose orbits consist of unmarked Schottky groups (\textit{i.e.} with no chosen basis).
Since a Mumford curve over a non-archimedean field $k$ determines a unique conjugacy class of Schottky groups, it is natural to consider the quotient by this action of the non-archimedean part of the Schottky space as a global space of Mumford curves.
In order to do so, we let~$\calS_g^{\mathrm{na}}$ be the set of non-archimedean points of~$\calS_g$. We prove the following result.

\begin{theorem*}[\protect{Corollaries~\ref{cor:actionOutFgna} and~\ref{cor:quotientOutFgfiber}}]\label{introthm:outer}
The action of~$\Out(F_g)$ on~$\calS_{g}^{\mathrm{na}}$ is proper and has finite stabilizers. For each $a\in \calM(\Z)$, the quotient space $\Mumf_{g,a}:=\Out(F_{g}) \backslash(\calS_{g}\cap\pr_{\Z}^{-1}(a))$ inherits a structure of $\calH(a)$-analytic space.
\end{theorem*}

The action of~$\Out(F_g)$ is proper also over the space~$\calS_g^{\mathrm{a}}$ of archimedean points of~$\calS_g$. 
The proof of this fact boils down to a globalized version of the analogue result over the complex numbers, which is already known and can be proven using Teichm\"uller theory, once one provides the connection between the Schottky space and Teichm\"uller space (see \cite{HerrlichSchmithuesen07} for a detailed discussion of this connection).
In the non-archimedean framework, our proof is inspired by the work of Gerritzen \cite{Gerritzen82}, and consists in applying Serre's theory of free groups acting on trees \cite[\S 3]{Serre77} to the case of Schottky groups acting on the Berkovich projective line.
This strategy is not easily adapted to the archimedean case, and we are not able to prove the properness of the action of~$\Out(F_g)$ on the entire space~$\calS_g$.
This would be a first step to give a structure of analytic space over~$\Z$ to the global quotient $\Out(F_{g}) \backslash \calS_{g}$, a result that we believe to hold true.

Even without a $\Z$-analytic structure on the global space of uniformizable curves, the theory of Schottky spaces over $\Z$ allows us to construct a universal Mumford curve, in the sense of a relative curve over $\calS_g$ that encodes all possible archimedean and non-archimedean uniformizations at once.

\begin{theorem*}\label{introthm:univunif}
There is a smooth proper morphism of analytic spaces over $\Z$
\[\calC_{g}\longrightarrow \calS_{g}\]
that is of relative dimension 1 and satisfies the following:
given a point $x\in \calS_{g}$, its preimage in $\calC_{g}$ is a curve over $\calH(x)$ uniformized by the Schottky group $\Gamma_x$.
\end{theorem*}

It is important to note that the uniformization property which lies at the basis of the theories of Schottky and Mumford carries over to our global setting. Indeed, the universal Mumford curve~$\calC_{g}$ may be uniformized by an open subset of the relative projective line over~$\calS_{g}$. The precise statement is as follows.

\begin{theorem*}[\protect{Corollary~\ref{cor:universalunif}}]\label{introthm:Omegag}
There exists an open subset~$\Omega_{g}$ of~$\PP^{1}_{\calS_{g}}$ and a morphism $\Omega_{g} \to \calC_{g}$ over~$\calS_{g}$ that is a local isomorphism. Over each point~$x$ in~$\calS_{g}$, the restriction of the morphism gives back the uniformization by the Schottky group~$\Gamma_{x}$. 
\end{theorem*}

To contextualize our result, let us mention a different construction of families of Mumford curves over $\Z$, due to T. Ichikawa. In \cite{Ichikawa00}, he introduces objects called \emph{generalized Tate curves}: stable curves defined over a $\Z$-algebra mixing polynomials and power series.
A typical example is the curve~$C_{0}$, which, by specialization on the base, gives back Mumford curves whose reduction graph is a rose with $g$ petals. It is defined over the ring $R_{0} := \Z[\underline{x},\prod_{i\neq j} \frac{1}{x_i-x_j}][[\underline{y}]]$, where the multivariables $\underline{x}$ and $\underline{y}$ are deformation parameters related to Schottky uniformization. A similar curve~$C_{\Delta}$ exists for each stable graph~$\Delta$ of genus~$g$ and is defined over a ring~$R_{\Delta}$ analogous to~$R_{0}$. Ichikawa's construction makes no use of Schottky spaces and is rather a generalization of Mumford's strategy to the case of a nonlocal base ring. More recently, in the preprint~\cite{Ichikawa20}, which appeared online after the first version of the present work, Ichikawa has shown that the generalized Tate curves may be glued into a universal Mumford curve, defined over a formal scheme over~$\Z$.
Ichikawa's approach has allowed him to study objects of arithmetic interest, such as Teichm\"uller modular forms \cite{Ichikawa18}, periods of Mumford curves \cite{Ichikawa20a}, and $p$-adic solutions of certain systems of partial differential equations, such as the Korteweg-deVries hierarchy \cite{Ichikawa01}.

There are several advantages in our construction of universal Schottky uniformization over $\Z$. First of all, the curve $\calC_g$ admits a nice description as a global object: the Schottky space $\calS_g$ is a subspace of the affine space and therefore has canonical global coordinates.
Moreover, the presence of the analytic topology provides a finer description of degenerations of families of Schottky groups, as well as a concrete way of studying natural group actions on the universal Mumford curve. 
Finally, Berkovich geometry brings out the connections between our construction and objects in neighboring theories such as tropical geometry and geometric group theory.

A convenient way to study these connections is the notion of \emph{skeleton} of a non-archimedean analytic space, a combinatorial object that plays an important role in Berkovich's theory and that can be interpreted as a ``tropical shadow'' of such a space.
In the one-dimensional case, skeletons are finite graphs that capture fundamental properties of the curve they represent.\footnote{For more details and a discussion of higher dimensional skeletons, we refer the reader to the excellent survey by A. Werner \cite{Werner16}.}
They appear in this work as invariants that define strata in the Schottky space, bringing out connections with tropical moduli spaces as follows.
Let us fix a non-archimedean point $a\in \calM(\Z)$, and recall that $\Mumf_{g,a}$ is the space of Mumford curves over extensions of $\calH(a)$, as in Theorem \ref{introthm:outer}.
Then, by assigning to each Mumford curve its skeleton, we can build a map of topological spaces
\[ \psi: \Mumf_{g,a} \to M_g^{\mathrm{trop}},\] 
where $M_g^{\mathrm{trop}}$ denotes the moduli space of tropical curves, parametrizing stable graphs of genus $g$ with weights assigned at vertices.
The image of the map~$\psi$ consists of unweighted metric graphs, and can be characterized as the quotient of the Culler-Vogtmann outer space $CV_g'$ by the action of $\Out(F_g)$.
The outer space is a fundamental object of geometric group theory, and has been used to show many group theoretic properties of $\Out(F_g)$ that were previously unknown.
The reader can find more information in the original paper \cite{CullerVogtmann86}, as well as in \cite{Vogtmann14}, which discusses more recent advances.
The space $CV_g'$ parametrizes graphs with an extra structure called a marking, and the action of the group $\Out(F_g)$ is nontrivial only at the level of the marking, in a way that is reminescent of the action of $\Out(F_g)$ on $\calS_g$.
It is therefore a very natural problem to ask for a comparison between the two actions, which we are able to provide with the following result.

\begin{theorem*}[Theorem~\protect{\ref{thm:Schottky-outer-space}}]\label{introthm:Schottkyouterspace}
Let $a \in \calM(\Z)$ be a non-archimedean point and let $\calS_{g,a}=\calS_g \cap \pr^{-1}_\Z(a)$ be the fiber of the Schottky space over $a$. Then, there is a continuous surjective function
\[\phi:\calS_{g,a}\longrightarrow CV_g' \times_{M_g^{\mathrm{trop}}} \Mumf_{g,a},\]
which is not injective for $g\geq 2$.
\end{theorem*}

The interplay between non-archimedean Schottky spaces, tropical geometry, and the outer space is also the object of very recent work of M. Ulirsch. In \cite{Ulirsch21}, which appeared online at the same time as the present work, he defines a non-archimedean analytic Deligne-Mumford stack $\overline{\calT_g}$, over a fixed algebraically closed non-archimedean complete valued field~$K$, which provides an analogue of Teichm\"uller space in non-archimedean analytic geometry. The construction proceeds in two steps. He first considers the logarithmic algebraic stack $\calT_{g}^{\mathrm{log}} = \calM_{g}^{\mathrm{log}} \times_{M_{g}^{\mathrm{trop}}} \calT_{g}^{\mathrm{trop}}$, where $\calM_{g}^{\mathrm{log}}$ is the logarithmic algebraic stack of stable curves and $\calT_{g}^{\mathrm{trop}}$ is the tropical Teichm\"uller space defined in~\cite{ChanMeloEtAl13}. The desired space $\overline{\calT_g}$ is then obtained by taking the algebraic stack underlying~$\calT_{g}^{\mathrm{log}}$, base changing it to the value group of~$K$, and applying Raynaud's generic fiber functor.

The space $\overline{\calT_g}$ parametrizes Berkovich stable curves over valued extensions of $K$, together with a choice of a basis of their topological fundamental group.
It contains a natural locus of marked Mumford curves, which is related to the Schottky space $\calS_g \times_\Z K$. Even though the techniques of the present paper significantly differ from those of \cite{Ulirsch21}, we can describe the extent of this relation by considering the connections of the two constructions with the spaces $M_g^{\mathrm{trop}}$ and $CV_g'$.
This is done in Remark \ref{rmk:Martin}, as well as in Remark 5 in the introduction of \cite{Ulirsch21}.


\medbreak

Finally, let us note that other moduli spaces of arithmetic significance have been extensively studied, over both archimedean and non-archimedean valued fields.
It seems natural to try to define a version over $\Z$ of these objects, the closest to the present work being the moduli space $\Mumf_{g,n}$ of pointed Mumford curves and the moduli space $\calA_g$ of principally polarized abelian varieties.
While the case of $\Mumf_{g,n}$ descends easily from the results already established by the authors, the space~$\calA_g$ and a version of the Torelli map over $\Z$ would require a bigger effort.
We plan to address these issues and related applications in future work.


\medbreak

The paper is structured as follows. Section 2 is a self-contained introduction to the affine and projective Berkovich spaces over the Banach ring $(\Z,\va_{\infty})$, providing all the definitions and basic material on the subject needed for the rest of the paper. In Section 3, we formulate classical definitions and results of the theory of Schottky groups in such a way to be applied to any valued field, archimedean and non-archimedean alike.
We adapt classical notions in the non-archimedean theory to the framework of Berkovich spaces, relying on previous work \cite{VIASMII} by the authors.
In Section 4, we define the space $\calS_g$ as a Berkovich space over $\Z$ and we prove that it is open in $\AA^{3g-3,\an}_{\Z}$. In Section 5, we study the natural action of $\Out(F_g)$ on $\calS_g$: we  establish some results that allow us to complete the proof of Theorem \ref{introthm:open} and determine the properties of the quotient as stated in Theorem \ref{introthm:outer}.
Finally, in Section 6, we prove the universal uniformization theorem (Theorems \ref{introthm:univunif} and \ref{introthm:Omegag}) and a result clarifying the connections with the outer space and the moduli space of tropical curves (Theorem \ref{introthm:Schottkyouterspace}). 

All the results proven in this paper remain valid if we replace $(\Z,\va_{\infty})$ with the ring of integers of a number field.

\subsection*{Acknowledgements}
We warmly thank Melody Chan, Frank Herrlich, Fabian Ruoff, Thibaud Lemanissier, and Marco Maculan for stimulating mathematical exchanges that improved a preliminary version of this paper.
We are especially grateful to Sam Payne who raised to us the question answered by Theorem \ref{thm:Schottky-outer-space} and to Martin Ulirsch for sharing with us a preliminary version of his work \cite{Ulirsch21}.

\section{Berkovich spaces over $\Z$}

\subsection{Analytic spaces over Banach rings}\label{sec:AnA}

Let $(A,\nm)$ be a Banach ring. In this section, we recall Berkovich's definition of analytic spaces over~$A$ (see \cite[Section~1.5]{Berkovich90}).

We start with the affine analytic space of dimension~$n$ over~$A$, denoted by $\A^{n,\an}_{A}$. It is a locally ringed space and we define it in three steps: underlying set, topology and structure sheaf. 

The set underlying $\A^{n,\an}_{A}$ is the set of bounded multiplicative seminorms on $A[T_{1},\dotsc,T_{n}]$ that are bounded on~$A$, \emph{i.e.} the set of maps 
\[ \va \colon A[T_{1},\dotsc,T_{n}] \too \R_{\ge 0} \]
that satisfy the following properties:
\begin{enumerate}
\item $|0|=0$ and $|1|=1$;
\item $\forall P,Q \in A[T_{1},\dotsc,T_{n}$, $|P+Q| \le |P| + |Q|$;
\item $\forall P,Q \in A[T_{1},\dotsc,T_{n}$, $|P+Q| = |P|\, |Q|$;
\item $\forall a \in A$, $|a| \le \|a\|$.
\end{enumerate}
We set $\calM(A) := \A^{0,\an}_{A}$ and call it the \emph{spectrum} of~$A$. Note that we have a projection map $\pr_{A} \colon \A^{n,\an}_{A} \to \calM(A)$ induced by the morphism $A \to A[T_{1},\dotsc,T_{n}]$.

Let~$x$ be a point of~$\A^{n,\an}_{A}$. Denote by~$\va_{x}$ the multiplicative seminorm associated to it. The ring $A[T_{1},\dotsc,T_{n}]/\ker(\va_{x})$ is a domain and we can consider its field of fractions. The seminorm~$\va_{x}$ induces an absolute value on the later it and we can consider its completion, which we denote by~$\calH(x)$. We simply denote by~$\va$ the absolute value on~$\calH(x)$ induced by~$\va_{x}$ since no confusion may result.

We have a natural morphism $\chi_{x}\colon A[T_{1},\dotsc,T_{n}] \to \calH(x)$. For each $P \in A[T_{1},\dotsc,T_{n}]$, we set $P(x) := \chi_{x}(P)$. Note that, by definition, we have $|P(x)| = |P|_{x}$.

\medbreak

The set $\A^{n,\an}_{A}$ is endowed with the coarsest topology such that, for each $P \in A[T_{1},\dotsc,T_{n}]$, the map
\[x \in \A^{n,\an}_{A} \mapstoo |P(x)| \in \R_{\ge0}\] 
is continuous. The resulting topological space is Hausdorff and locally compact. The spectrum~$\calM(A)$ is even compact. The projection map~$\pr_{A}$ is continuous.

\medbreak

For each open subset~$V$ of~$\A^{n,\an}_{A}$, we denote by~$S_{V}$ the set of element of $A[T_{1},\dotsc,T_{n}]$ that do not vanish on~$V$ and set $K(V) := S_{V}^{-1} A[T_{1},\dotsc,T_{n}]$. 

Let~$U$ be an open subset of~$\A^{n,\an}_{A}$. We define~$\calO(U)$ to be the set of maps
\[f \colon U \too \bigsqcup_{x\in U} \calH(x)\]
such that
\begin{enumerate}
\item for each $x\in U$, $f(x)\in \calH(x)$;
\item each $x\in U$ has an open neighborhood~$V$ on which~$f$ is a uniform limit of elements of~$K(V)$. 
\end{enumerate}

\medbreak

One may now define arbitrary analytic spaces over~$A$ as locally ringed spaces that are locally isomorphic to some $(V(\calI),\calO_{U}/\calI)$, where~$U$ is an open subset of~$\A^{n,\an}_{A}$ and $\calI$~is a sheaf of ideals of~$\calO_{U}$.

\medbreak

A point~$x$ of an analytic space~$X$ over~$A$ is said to be archimedean or non-archimedean if the associated absolute valued on~$\Hx$ is. We denote by~$X^\mathrm{a}$ (resp. $X^\mathrm{na}$) the set of archimedean (resp. non-archimedean) points of~$x$ and call it the archimedean (resp. non-archimedean) part of~$X$. It is well-known that an absolute value on a field is archimedean if, and only if, its restriction to the prime field is. It follows that we have 
\[X^{\text{a}} = \{x\in X : |2(x)|>1\} \textrm{ and } X^{\text{na}} = \{x\in X : |2(x)|\le1\}\]
(and~2 could be replaced by any integer bigger than~1). In particular, the archimedean and non-archimedean parts of~$X$ are respectively open and closed subsets of~$X$.

\medbreak

To go further, one should define the category of analytic spaces over~$A$. When $A$~is a complete non-archimedean valued field, this has been achieved by V.~Berkovich in~\cite{Berkovich90, Berkovich94} (with a more general notion of analytic space). In~\cite[\S 2.1]{LemanissierPoineau20}, T.~Lemanissier and the first-named author gave a definition over an arbitrary Banach ring.  However, the category is shown to enjoy nice properties only under additional assumptions, for instance when $A$~is a discrete valuation ring (with some mild extra hypotheses) or the ring of integers of a number field (see Section~\ref{sec:BerkovichZ} for some definitions related to this setting). For future use, we note that, in those cases, fiber products exist (see \cite[Th\'eor\`eme~4.3.8]{LemanissierPoineau20}).


\subsection{Relative projective line}\label{sec:P1S}

In the rest of the text, we will not only need affine spaces, but also projective spaces and, more precisely, relative projective lines over affine spaces. We explain here how to construct them in a down-to-earth way. 
Let $(A,\nm)$ be a Banach ring. Let $n\in\N$ and denote by~$S$ the analytic space $\AA^{n,\an}_{A}$ with coordinates $T_{1},\dotsc,T_{n}$. 

Let~$U$ (resp.~$V$) be the affine space $\AA^{n+1,\an}_{A}$ with coordinates $T_{1},\dotsc,T_{n},Z$ (resp. $T_{1},\dotsc,T_{n},Z'$) and denote by $U_{0}$ (resp.~$V_{0}$) the open subset defined by the inequality $Z\ne 0$ (resp. $Z'\ne 0$). The morphism 
\[\begin{array}{ccc}
A[T_{1},\dotsc,T_{n},Z,Z^{-1}] & \too & A[T_{1},\dotsc,T_{n},Z',Z'^{-1}]\\
T_{i} & \mapstoo & T_{i}\\
Z & \mapstoo & Z'^{-1}
\end{array}\]
induces an isomorphism $U_{0} \xrightarrow[]{\sim} V_{0}$.

We denote by~$\PP^1_{S}$ the analytic space obtained by glueing~$U$ and~$V$ along~$U_{0}$ and~$V_{0}$ \textit{via} the previous isomorphism. It comes with a natural projection morphism $\pi \colon \PP^1_{S} \to S$. For any open subset~$S'$ of~$S$, we denote by~$\PP^1_{S'}$ the analytic space $\pi^{-1}(S')$. 

When $n=0$, we will denote~$\PP^{1}_{\calM(A)}$ by $\PP^{1,\an}_{A}$. Note that, for each $s\in S$, the fiber $\pi^{-1}(s)$ identifies to~$\PP^{1,\an}_{\calH(s)}$.

Let $M := \begin{pmatrix} a&b\\c&d \end{pmatrix} \in \GL_{2}(\calO(S))$. We may associate to it an endomorphism of~$\PP^1_{S}$ by the usual expression in coordinates
\[Z \mapstoo \frac{aZ+b}{cZ+d}.\]
This way we get an action of~$\GL_{2}(\calO(S))$ on~$\PP^1_{S}$. It factors through~$\PGL_{2}(\calO(S))$. The image of~$M$ in~$\PGL_{2}(\calO(S))$ will be denoted by
$[M] = \begin{bmatrix} a&b\\c&d \end{bmatrix}$. Note that the action restricts to an action on each fiber of~$\pi$, hence also on~$\PP^1_{S'}$ for any open subset~$S'$ of~$S$.

\subsection{Berkovich spaces over~$\Z$}\label{sec:BerkovichZ}

In this section, we consider the special case where $(A,\nm) = (\Z,\va_{\infty})$, where~$\va_{\infty}$ denotes the usual absolute value. We refer to~\cite{A1Z}, and especially Section~3.1 there, for more details.

The spectrum~$\calM(\Z)$ is easily described using Ostrowski's theorem. It contains the following points:
\begin{itemize}
\item a point $a_{0}$, associated to the trivial absolute value~$\va_{0}$, with residue field~$\Q$;
\item for each $\eps \in (0,1]$, a point~$a_{\infty}^\eps$ associated to the absolute value~$\va_{\infty}^\eps$, with residue field~$\R$;
\item for each prime number~$p$ and each $\eps \in (0,+\infty)$, a point~$a_{p}^\eps$ associated to the absolute value~$\va_{p}^\eps$, with residue field~$\Q_{p}$;
\item for each prime number~$p$, a point~$a_{p,0}$ associated to the seminorm on~$\Z$ induced by the trivial absolute value on~$\Z/p\Z$, with residue field~$\Z/p\Z$.
\end{itemize}
We will sometimes drop the exponent~1 and write $a_{p}$ and~$a_{\infty}$ instead of $a_{p}^1$ and~$a_{\infty}^1$. The archimedean part of~$\calM(\Z)$ is the open subset $\calM(\Z)^\mathrm{a} = \{a_{\infty}^\eps \mid \eps\in(0,1]\}$. 

The topology of~$\calM(\Z)$ is quite simple. First, the branches are all homeomorphic to segments: for each prime number~$p$, the map
\[b_{p} \colon \eta \in [0,1] \mapstoo 
\begin{cases}
a_{p,0} &\text{ if } \eta = 0;\\
a_{p}^{-\log(\eta)} &\text{ if } \eta \in (0,1);\\
a_{0} &\text{ if } \eta=1
\end{cases}\]
is a homeomorphism and the map
\[\beta_{\infty} \colon \eps \in [0,1] \mapstoo 
\begin{cases}
a_{0} &\text{ if } \eps = 0;\\
a_{\infty}^{\eps} &\text{ if } \eps \in (0,1]
\end{cases}\]
is a homeomorphism too. Moreover, a subset~$U$ of~$\calM(\Z)$ containg~$a_{0}$ is open if, and only if, the intersection of~$U$ with each $b_{p}([0,1])$ and $\beta_{\infty}([0,1])$ is open and only finitely many of those sets are not contained entirely in~$U$. In other words, $\calM(\Z)$ is homeomorphic to the Alexandroff one-point compactification of the disjoint union of the $b_{p}([0,1))$'s and $\beta_{\infty}((0,1])$, the point at infinity being~$a_{0}$.

\begin{figure}[!h]
\centering
\begin{tikzpicture}
\foreach \x [count=\xi] in {-2,-1,...,17}
\draw (0,0) -- ({10*cos(\x*pi/10 r)/\xi},{10*sin(\x*pi/10 r)/\xi}) ;
\foreach \x [count=\xi] in {-2,-1,...,17}
\fill ({10*cos(\x*pi/10 r)/\xi},{10*sin(\x*pi/10 r)/\xi}) circle ({0.07/(sqrt(\xi)}) ;

\draw ({10.5*cos(-pi/5 r)},{10.5*sin(-pi/5 r)}) node{$a_{\infty}$} ;
\fill ({5.5*cos(-pi/5 r)},{5.5*sin(-pi/5 r)}) circle (0.07) ;
\draw ({5.5*cos(-pi/5 r)},{5.5*sin(-pi/5 r)-.1}) node[below]{$a_{\infty}^\eps$} ;

\draw ({11*cos(-pi/10 r)/2+.1},{11*sin(-pi/10 r)/2}) node{$a_{2,0}$} ;
\fill ({2.75*cos(-pi/10 r)},{2.75*sin(-pi/10 r)}) circle ({0.07/(sqrt(2)}) ;
\draw ({2.75*cos(-pi/10 r)},{2.75*sin(-pi/10 r)-.05}) node[below]{$a_2^\eps$} ;

\draw ({3.9},{-.05}) node{$a_{3,0}$} ;
\fill ({2},{0)}) circle ({0.07/(sqrt(2)}) ;
\draw ({2.1},{-.05}) node[below]{$a_3^\eps$} ;

\draw ({12*cos(pi/5 r)/5+.1},{12*sin(pi/5 r)/5+.1}) node{$a_{p,0}$} ;

\draw (-1.4,-.9) node{$a_{0}$} ;
\draw[thick,dotted,- stealth] (-1.3,-.7) to [out=60,in=190] (-.3,-.05) ;
\end{tikzpicture}
\caption{The analytic spectrum $\calM(\Z)$.}\label{fig:MZ}
\end{figure}
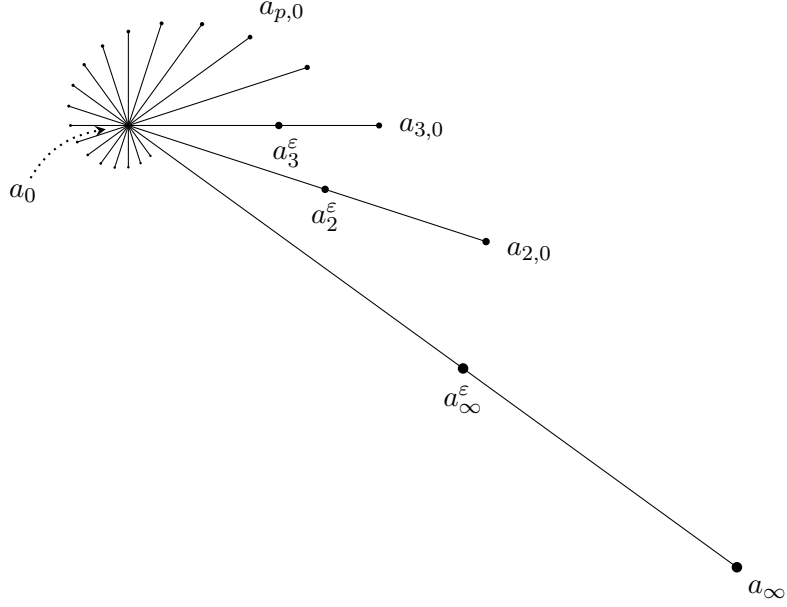

\medbreak

We will often think about an analytic space over~$\Z$ as a family of analytic spaces over the different valued fields associated to the points of~$\calM(\Z)$. The spaces over~$\Q_{p}$, $\Q$, $\Z/p\Z$ (the last two being endowed with the trivial absolute value) are then usual Berkovich spaces. Recall that the analytic spaces over~$\R$ in the sense of Berkovich are the quotients of the corresponding usual analytic spaces over~$\C$ by the complex conjugation.

Let us be more precise in the case of an affine space $\AA^{n,\an}_{\R,\va_{\infty}^\eps}$, for some $\eps\in(0,1]$. The complex conjugation induces an automorphism of~$\C^n$ given by
\[c\colon z = (z_{1},\dotsc,z_{n}) \in \C^n \mapstoo (\bar z_{1},\dotsc,\bar z_{n}) \in \C^n\]
and we have a homeomorphism
\[ \rho_{\eps} \colon z \in \C^n/\langle c \rangle \mapstoo v_{z,\eps} \in \AA^{n,\an}_{\R,\va_{\infty}^\eps},\]
where $v_{z,\eps} \colon P(T) \in \R[T] \mapsto |P(z)|_{\infty}^\eps$. It follows that all the archimedean fibers are the same. More precisely, the map
\[\Phi\colon (v,\eps) \in \AA^{n,\an}_{\R,\va_{\infty}} \times(0,1] \mapstoo \rho_{\eps} \circ \rho_{1}^{-1}(v) \in \big(\AA^{n,\an}_{\Z}\big)^\text{a}\]
is a homeomorphism. Note that $\Phi(v,\eps)$ may also be defined explicitly as the seminorm defined by
\[\Phi(v,\eps) \colon P(T) \in \R[T] \mapstoo |P(v)|^\eps.\]
In particular, the seminorms~$v$ and~$\Phi(v,\eps)$ are equivalent.

\medbreak

As regards topology, analytic spaces over~$\Z$ are known to be locally path-connected thanks to~\cite[Th\'eor\`eme~7.2.17]{LemanissierPoineau20}. As one can expect, surprising phenomena occur when passing from archimedean to the non-archimedean part. We illustrate this by giving two examples of continuous sections of the projection $\pr_{\Z} \colon \AA^{1,\an}_{\Z} \to \calM(\Z)$.

\begin{example}\label{ex:sectionGauss}
Let $\alpha$ be an element of~$\C$ that is transcendental over~$\Q$. For each $a \in \calM(\Z)^\text{na}$, denote by~$\eta_{a,1}$ the Shilov boundary of the disc of center~0 and radius~1, \emph{i.e.} the Gau\ss{} point, in the fiber~$\pr_{\Z}^{-1}(a)$. The map 
\[\sigma \colon a \in \calM(\Z) \mapstoo 
\begin{cases}
\eta_{a,1} & \text{ if } a \text{ is non-archimedean;}\\
\rho_{\eps}(\alpha) & \text{if } a=a_{\infty}^{\eps} \text{ with } \eps \in (0,1].
\end{cases} 
\]
is a continuous section of $\pr_{\Z} \colon \AA^{1,\an}_{\Z} \to \calM(\Z)$. 
For this it is enough to show that $\rho_{\eps}(\alpha)$ tends to~$\eta_{a_{0},1}$ when $\eps$ goes to~0. Remark that the point~$\eta_{a_{0},1}$ corresponds to the trivial absolute value on~$\Z[T]$ and that, for each $P \in \Z[T] - \{0\}$, we have 
\[|P(\rho_{\eps}(\alpha))| = |P(\alpha)|_{\infty}^\eps \xrightarrow[\eps \to 0]{} 1\]
since $\alpha$ is transcendental over~$\Q$. The result follows.
\end{example}

\begin{example}\label{ex:sectionetar}
Let $r \in (0,1)$. For each $a \in \calM(\Z)^\text{na}$, denote by~$\eta_{a,r}$ the Shilov boundary of the disc of center~0 and radius~r in the fiber~$\pr_{\Z}^{-1}(a)$. The map 
\[\tau_{r} \colon a \in \calM(\Z) \mapstoo 
\begin{cases}
\eta_{a,r} & \text{ if } a \text{ is non-archimedean;}\\
\rho_{\eps}(r^{1/\eps}) & \text{if } a=a_{\infty}^{\eps} \text{ with } \eps \in (0,1].
\end{cases} 
\]
is a continuous section of $\pr_{\Z} \colon \AA^{1,\an}_{\Z} \to \calM(\Z)$. It is enough to show that $\rho_{\eps}(r^{1/\eps})$ tends to~$\eta_{a_{0},r}$ when $\eps$ goes to~0. This is clear since, for each $\eps \in (0,1]$, we have 
\[|T(\rho_{\eps}(r^{1/\eps})| = |r^{1/\eps}|_{\infty}^\eps = r\] 
and $\eta_{a_{0},r}$ is the only point of the fiber $\pr_{\Z}^{-1}(a_{0})$ where~$T$ has absolute value~$r$.
\end{example}

\medbreak

One can build a similar theory replacing~$\Z$ by the ring of integers~$\calO_{K}$ of a number field~$K$. To be more precise, let us denote by~$\Sigma_{K}$ the set of complex embeddings of~$K$ up to conjugation and endow~$\calO_{K}$ with the norm 
\[\nm_{K} := \max(|\sigma(\wc)|_{\infty}, \ \sigma\in \Sigma_{K}).\] 
Then, the spectrum~$\calM(\calO_{K})$ looks very similar to~$\calM(\Z)$: it is a tree with one point associated to the trivial absolute value and, for each place of~$K$, one branch emanating from it.

Remark that the restriction of seminorms induces a map $\calM(\calO_{K})\to\calM(\Z)$, and more generally a map $\A^{n,\an}_{\calO_{K}} \to \A^{n,\an}_{\Z}$. Those maps are continuous and open. 

Note also that~$\calM(\calO_{K})$ is an analytic space over~$\Z$ in the sense of Section~\ref{sec:AnA}. In particular, by~\cite[Th\'eor\`eme~4.3.8]{LemanissierPoineau20}, it makes sense to consider the fiber product of an analytic space over~$\Z$ by~$\calM(\calO_{K})$ over~$\calM(\Z)$. We obtain canonical identifications $\A^{n,\an}_{\calO_{K}} = \A^{n,\an}_{\Z} \times_{\calM(\Z)} \calM(\calO_{K})$, $\PP^{1,\an}_{\calO_{K}} = \PP^{1,\an}_{\Z} \times_{\calM(\Z)} \calM(\calO_{K})$, etc.

\subsection{Some useful inequalities}

In this section, we fix a complete valued field $(k,\va)$, archimedean or not. We state a few results that will be useful later.


\begin{lemma}\label{lem:almostnonarchimedean}
Let $a,b \in k$. We have 
\[|a+b| \le \max(|2|,1)\, \max(|a|,|b|).\]
If $|a|> \max(|2|,1)\,|b|$, then we have
\[|a+b| \ge \frac{|a|}{\max(|2|,1)}.\]
\end{lemma}
\begin{proof}
If~$(k,\va)$ is non-archimedean, then $\max(|2|,1) = 1$, and those inequalities are well-known.

Assume that~$(k,\va)$ is archimedean. Then $(k,\va)$ embeds isometrically into $(\C,\va_{\infty}^{\eps})$ for some $\eps \in (0,1]$ and it is enough to prove the result for the latter. In this case, we have $\max(|2|,1) = 2^\eps$. For any $a,b\in\C$, we have $|a+b|_{\infty} \le 2\, \max(|a|_{\infty},|b|_{\infty})$ and the first result follows by raising the inequality to the power~$\eps$.

The inequality applied to $a+b$ and $-b$ gives $|a| \le |2| \max(|a+b|,|b|)$. As a consequence, if we have $|a| > |2|\, |b|$, we must have $|a| \le |2|\, |a+b|$. 
\end{proof}

It will be useful to introduce a notation for discs. We consider here the Berkovich affine line~$\aank$ over~$k$ with coordinate~$T$.

\begin{notation}
For $a\in k$ and $r\in \R_{>0}$, we set
\begin{align*}
D^+(a,r) &:=  \{x\in \aank \mid |T(x)-a| \le r\},\\
D^-(a,r) &:=  \{x\in \aank \mid |T(x)-a| < r\}.
\end{align*}
\end{notation}




\begin{lemma}\label{lem:2rho}
Let $a,b \in k$ and $\rho_{a},\rho_b \in \R_{>0}$. If $|a-b| > \max(|2|,1)\, \max(\rho_{a},\rho_{b})$, then the closed discs $D^+(a,\rho_a)$ and $D^+(b, \rho_b)$ are disjoint.

If~$\va$ is non-archimedean, then the closed discs $D^+(a,\rho_a)$ and $D^+(b, \rho_b)$ are disjoint if, and only if, $|a-b| >\max(\rho_{a},\rho_{b})$.
\end{lemma}
\begin{proof}
If there exists a point~$x$ in $D^+(a,\rho_a) \cap D^+(b, \rho_b)$, then we have $|T(x)-a| \le \rho_{a}$ and $|T(x)-b| \le \rho_{b}$ in $\Hx$, hence 
\[|a-b| = |(a - T(x)) + (T(x) - b)| \le \max(|2|,1)\, \max(\rho_{a},\rho_{b})\]
by Lemma~\ref{lem:almostnonarchimedean}. The first part of the result follows.

The converse implication in the non-archimedean setting is well-known.
\end{proof}

\begin{lemma}\label{lem:distinctabsval}
Let $a,b \in k$. If $|a+b|^2 > \max(|4|,1)\, |ab|$, then $|a|\ne|b|$.

If~$\va$ is non-archimedean, then we have $|a|\ne|b|$ if, and only if, $|ab| < |a+b|^2$.
\end{lemma}
\begin{proof}
If $|a| = |b|$, then, by Lemma~\ref{lem:almostnonarchimedean}, we have $|a+b| \le \max(|2|,1)\, |a|$, hence
\[|a+b|^{2} \le \max(|2|^2,1)\, |a|^2 = \max(|4|,1)\, |ab|.\]
The first part of the result follows.

Let us now assume that~$\va$ is non-archimedean. Assume that $|a|\ne |b|$. Then, we have $|a+b| = \max(|a|,|b|) > \min(|a|,|b|)$, hence 
\[|a+b|^2 =  \max(|a|,|b|)^2 > \max(|a|,|b|)\,\min(|a|,|b|) = |a|\,|b|.\]
The converse implication follows directly from the first part of the statement.
\end{proof}

\subsection{Metric structure}

In this section, we fix a complete non-archimedean valued field $(k,\va)$. In the following, we will often encounter the projective line~$\pank$ and we gather here a few metric properties.

First recall that $\pank$ has the structure of a real tree (see \cite[(3.4.20)]{DucrosRSS}). In particular, for any two distinct points $x,y \in \pank$, there exists a unique segment $[x y]$ joining~$x$ to~$y$. Recall also that each segment consisting of points of type~2 or~3 carries a multiplicative length (or modulus) that is invariant under isomorphisms of~$\pank$, \textit{i.e.} under M\"obius transformations  (see \cite[(3.6.23)]{DucrosRSS}). To define the length of such a segment~$I$, one may proceed as follows. 

\begin{notation}
For $a\in k$ and $r\in \R_{>0}$, we denote by~$\eta_{a,r}$ the unique point is the Shilov boundary of the closed disc~$D^+(a,r)$. 
\end{notation}

There exist a finite extension~$k'$ of~$k$, a coordinate~$T$ on~$\pank$, $a\in k'$ and $r\le s\in \R_{>0}$ such that $I$ is the image of the segment $[\eta_{a,r}, \eta_{a,s}]$ by the projection map $\pana{k'} \to \pank$. We then set 
\[\ell(I) := \frac s r \in [1,+\infty).\] 
It is independent of the choices made. It will convenient to set $\ell(\emptyset) := 1$.

\begin{lemma}\label{lem:CrossRatio}
Let $a,b,c,d$ be distinct points of~$\PP^1(k)$ and denote their cross-ratio by~$[a, b;c, d]$. 

Set $I := [ab] \cap [cd]$. It is either a segment consisting of points of type~2 or~3 or the empty set. If $I$ is a non-trivial segment and if going from~$a$ to~$b$ and from~$c$ to~$d$ induces the same orientation on~$I$, then we set $\eps:=-1$. In all other cases, we set $\eps:=1$.

%

Then, we have
\[|[a,b;c,d]| = \ell(I)^\eps.\]
\end{lemma}
\begin{proof}
Since the cross-ratio is invariant under M\"obius transformations, we may assume that $b=1$, $c=0$ and $d=\infty$. Assume that $|a| <1$. Then $[ab] \cap [cd] = [\eta_{0,|a|},\eta_{0,1}]$ and going from~$a$ to~$b$ and from~$c$ to~$d$ induces the same orientation on it, hence $\eps=-1$. We have 
\[\ell([ab] \cap [cd])^{-1} = |a| = |[a,b;c,d]|\]
as desired. The other cases are dealt with similarly.

\end{proof}

\section{Schottky groups}

The notion of Schottky group is classical over~$\C$ (see \cite{IndrasPearls}) and even over a complete non-archimedean valued field (see \cite{GerritzenPut80}). The definitions, results and proofs that appear in this section are adaptations of the standard ones to a relative setting. 

We would also like to point the reader to the recent text \cite{VIASMII} by the authors, which provides a detailed exposition of the theory non-archimedean Schottky groups and Schottky uniformization in the setting of Berkovich spaces (and corrects some mistakes in the existing literature). Our treatment here follows this reference quite closely.

\subsection{Geometric situation}\label{sec:geometricSchottky}

Let~$S$ be an analytic space over a Banach ring (archimedean or not). As in Section~\ref{sec:P1S}, consider the analytic space~$\PP^1_{S}$ and the projection morphism $\pi\colon \PP^1_{S} \to S$. In this section, we describe geometric properties of the action of some groups of automorphisms of $\PP^1_{S}$. 
It follows the strategy of \cite[I, 4.1]{GerritzenPut80}, see also \cite[\S 6.4.1]{VIASMII}.

\begin{definition}\label{def:Schottkyfigure}
Let $(\gamma_{1},\dotsc,\gamma_{g}) \in \PGL_{2}(\calO(S))^g$. Let $\calB = \big(B^+(\gamma_{i}^\eps), 1\le i\le g, \eps=\pm1\big)$ be a family of closed subsets of~$\PP^1_{S}$ that are disjoint. For each $i\in\{1,\dotsc,g\}$ and $\eps \in \{-1,1\}$, set 
\[B^-(\gamma_{i}^\eps) := \gamma_{i}^\eps (\PP^1_{S} - B^+(\gamma_{i}^{-\eps})).\]
For each $s\in S$, $i\in\{1,\dotsc,g\}$, $\eps \in \{-1,1\}$ and $\sigma\in\{-,+\}$, set $B_{s}^\sigma(\gamma_{i}^\eps) := B^\sigma(\gamma_{i}^\eps) \cap \pi^{-1}(s)$. 

We say that~$\calB$ is a \emph{Schottky figure} adapted to $(\gamma_{1},\dotsc,\gamma_{g})$ if, for each $s\in S$, $i\in\{1,\dotsc,g\}$ and $\eps \in \{-1,1\}$, $B_{s}^+(\gamma_{i}^\eps)$ is a closed disc in $\pi^{-1}(s) \simeq \PP^1_{\calH(s)}$ and $B_{s}^-(\gamma_{i}^\eps)$ is a maximal open disc inside it.
\end{definition}

\begin{remark}\label{rem:conjugatedSchottkyfigure}
Let $\varphi \in \PGL_{2}(\calO(S))$. If $\calB = \big(B^+(\gamma_{i}^\eps), 1\le i\le g, \eps=\pm1\big)$ is a Schottky figure adapted to $(\gamma_{1},\dotsc,\gamma_{g})$, then $\big(\varphi(B^+(\gamma_{i}^\eps)), 1\le i\le g, \eps=\pm1\big)$ is a Schottky figure adapted to $(\varphi\gamma_{1}\varphi^{-1},\dotsc,\varphi\gamma_{g}\varphi^{-1})$.
\end{remark}

In this section, we assume that we are in the situation of Definition~\ref{def:Schottkyfigure}. For $\sigma\in \{-,+\}$, we set 
\[F^\sigma := \PP^1_{S} - \bigcup_{\substack{1\le i\le g \\ \eps=\pm1}} B^{-\sigma}(\gamma_{i}^\eps).\]

Note that, for $\gamma_{0}\in\{\gamma_{1}^{\pm1},\dotsc,\gamma_{g}^{\pm1}\}$ and $\sigma \in \{-,+\}$, $B^\sigma(\gamma_{0})$ is the unique disc among the $B^\sigma(\gamma)$'s  containing~$\gamma_{0} F^\sigma$.

\medbreak

Set $\Delta := \{\gamma_{1},\dotsc,\gamma_{g}\}$. Denote by~$F_{g}$ the free group over the alphabet~$\Delta$ and by~$\Gamma$ the subgroup of $\PGL_{2}(\calO(B))$ generated by~$\Delta$. We have a natural morphism $\varphi\colon F_{g} \to \Gamma$ sending each $\gamma$ in~$\Delta$ to~$\gamma$. It induces an action of~$F_{g}$ on~$\PP^1_{S}$.

We now define subsets of~$\PP^1_{S}$ associated to the elements of~$F_g$. As usual, we will identify those elements with the words over the alphabet $\Delta^{\pm} := \{\gamma^{\pm1}_{1},\dotsc,\gamma_{g}^{\pm1}\}$.

\begin{notation}
For a non-empty reduced word $w=w'\gamma$ over~$\Delta^\pm$ and $\sigma \in \{-,+\}$, we set
\[B^\sigma(w) := w' \,B^\sigma(\gamma).\]
\end{notation}


%

%
%
%
%
%


The following result is stated and proved in \cite[Lemma 6.4.6]{VIASMII}.

\begin{lemma}\label{lem:wordF}
Let $u$ be a non-empty reduced word over~$\Delta^\pm$. Then we have $u F^+ \subseteq B^+(u)$.

Let $v$ be a non-empty reduced word over~$\Delta^\pm$. If there exists a word~$w$ over~$\Delta^\pm$ such that $u = vw$, then we have $u F^+ \subseteq B^+(u) \subseteq B^+(v)$. If, moreover, $u\ne v$, then we have $B^+(u) \subseteq B^-(v)$.

Conversely, if we have $B^+(u) \subseteq B^+(v)$, then there exists a word~$w$ over~$\Delta^\pm$ such that $u = vw$.
\qed
\end{lemma}






\begin{corollary}\label{cor:free}
The morphism~$\varphi$ is an isomorphism and the group~$\Gamma$ is free on the generators $\gamma_{1},\dotsc,\gamma_{g}$.
\qed
\end{corollary}

As a consequence, we will now identify~$\Gamma$ with~$F_g$ and express the elements of~$\Gamma$ as words over the alphabet~$\Delta^\pm$. In particular, we will allow us to speak of the length of an element~$\gamma$ of~$\Gamma$. We will denote it by~$|\gamma|$.

Set
\[O_{n} :=  \bigcup_{|\gamma|\le n} \gamma F^+.\]
Since the complement of~$F^+$ is the disjoint union of the open disks $B^-(\gamma)$ with $\gamma \in \Delta^\pm$, it follows from the description of the action that, for each $n\ge 0$, we have 
\[\PP^1_{S} - O_{n} = \bigsqcup_{|\gamma|= n+1} B^-(\gamma).\]
It follows from Lemma~\ref{lem:wordF} that, for each $n\ge 0$, $O_{n}$ is contained in the interior of~$O_{n+1}$. We set 
\[O := \bigcup_{n\ge 0} O_{n} = \bigcup_{\gamma\in\Gamma} \gamma F^+.\]

\subsection{Over a valued field.}

Let $(k,\va)$ be a complete valued field. In this section, we will focus on the particular case $S = \calM(k)$. In this setting, the material of this section is classical: see \cite[Project~4.5]{IndrasPearls} and \cite[I, 4.1.3]{GerritzenPut80} (or \cite[\S 6.4.1]{VIASMII}) in the archimedean and non-archimedean case respectively.

We still assume that we are in the situation of Definition~\ref{def:Schottkyfigure}. Set $\iota := \begin{bmatrix} 0&1\\1&0 \end{bmatrix} \in \PGL_{2}(k)$. It corresponds to the map $z \mapsto 1/z$ on~$\pank$. The first result follows from an explicit computation.

\begin{lemma}\label{lem:inversiondisc}
Let $\alpha\in k^\ast$ and $\rho \in [0,|\alpha|)$. Then, we have
\[\iota D^+(\alpha,\rho) = 
\begin{dcases}
D^+\left(\frac{\bar\alpha}{|\alpha|^2-\rho^2}, \frac{\rho}{|\alpha|^2-\rho^2}\right) & \textrm{if } k \text{ is archimedean};\\
D^+\left(\frac{1}{\alpha}, \frac{\rho}{|\alpha|^2}\right) \ \text{otherwise}.
\end{dcases}\]
\qed
\end{lemma}

\begin{lemma}\label{lem:radiusdisc}
Let $r>0$ and let $\gamma =
\begin{bmatrix} 
a & b \\
c & d
\end{bmatrix}$ in $\PGL_{2}(k)$ such that $\gamma D^+(0,r) \subseteq \aank$. Then, we have $|d|>r |c|$ and 
\[\gamma D^+(0,r) = 
\begin{dcases}
D^+\left(\frac{b \bar d - a \bar c r^2}{|d|^2 - |c|^2\, r^2} , \frac{|ad-bc|\, r}{|d|^2 - |c|^2\, r^2} \right) & \textrm{if } k \text{ is archimedean};\\
D^+\left(\frac{b}{d}, \frac{|ad-bc|\, r}{|d|^2}\right) \ \text{otherwise}.
\end{dcases}\]

\end{lemma}
\begin{proof}
Let us first assume that $c=0$. Then, we have $d\ne 0$, so the inequality $|d|>r |c|$ holds, and~$\gamma$ is affine with ratio~$a/d$. The result follows. 

Let us now assume that $c\ne 0$. In this case, we have $\gamma^{-1}(\infty) = -\frac{d}{c}$, which does not belong to~$D(0,r)$ if, and only if, $|d| >r |c|$. Note that we have the following equality in~$k(T)$:
\[\frac{aT+b}{cT+d} = \frac{a}{c} - \frac{ad-bc}{c^2}\, \frac{1}{T+\frac{d}{c}}.\]
By Lemma~\ref{lem:inversiondisc}, there exist $\beta \in k$ and $\sigma>0$ such that $\iota D^+(\frac{d}{c},r) = D^+(\beta,\sigma)$. Then, we have $\gamma D^+(0,r) = D^+(\frac{a}{c} - \frac{ad-bc}{c^2}\,\beta, \big|\frac{ad-bc}{c^2}\big|\,\sigma)$ and the result follows from an explicit computation.
\end{proof}

\begin{lemma}\label{lem:quotient}
Let $D' \subseteq D$ be closed concentric discs in~$\aank$. Let $\gamma \in \PGL_{2}(k)$ such that $\gamma D' \subseteq \gamma D \subseteq \A^{1,\an}_{k}$. Then, we have
\[\frac{\textrm{radius of } \gamma(D')}{\textrm{radius of } \gamma(D)} \le \frac{\textrm{radius of } D'}{\textrm{radius of } D},\]
with an equality if~$k$ is non-archimedean.
\end{lemma}
\begin{proof}

Let~$p$ be the center of~$D$ and~$D'$ and let~$\tau$ be the translation sending~$p$ to~0. Up to changing~$D$ into~$\tau D$, $D'$ into~$\tau D'$, $\gamma$ into~$\gamma \tau^{-1}$ and $\gamma'$ into~$\gamma' \tau^{-1}$, we may assume that~$D$ and~$D'$ are centered at~0. The result then follows from Lemma~\ref{lem:radiusdisc}.
\end{proof}

%
%


\begin{proposition}\label{prop:radiusn}
Assume that $\infty \in F^-$. Then, there exist $R>0$ and $c \in (0,1)$ such that, for each $\gamma \in \Gamma-\{\id\}$, $B^+(\gamma)$ is a closed disc of radius at most $R \,c^{|\gamma|}$. 
\end{proposition}
\begin{proof}

Let $\delta,\delta' \in \Delta^\pm$. By Lemma~\ref{lem:wordF}, we have an inclusion of discs $B^+(\delta' \delta) \subseteq B^+(\delta')$. There exists $f_{\delta,\delta'} \in \PGL_{2}(k)$ that sends those discs to concentric disks inside~$\aank$. In the non-archimedean case, the discs are already concentric, so one may take $f_{\delta,\delta'} = \id$, while, in the archimedean case, there is some work to be done, for which we refer to \cite[Project 3.4]{IndrasPearls}. Set 
\[c_{\delta,\delta'} := \frac{\text{radius of } f_{\delta,\delta'}(B^+(\delta' \delta))}{\text{radius of } f_{\delta,\delta'}(B^+(\delta'))} \in (0,1).\]

For each $\gamma\in \Gamma$ such that $\gamma\delta'$ is a reduced word, by Lemma~\ref{lem:quotient}, we have
\[ \frac{\text{radius of } B^+(\gamma\delta'\delta)}{\text{radius of } B^+(\gamma\delta')} = \frac{\text{radius of } \gamma f_{\delta,\delta'}^{-1}f_{\delta,\delta'}(B^+(\delta'\delta))}{\text{radius of } \gamma f_{\delta,\delta'}^{-1} f_{\delta,\delta'}(B^+(\delta'))}  \le c_{\delta,\delta'}.\]

Set 
\[R := \max (\{\text{radius of } B^+(\gamma) \mid \gamma \in \Delta^\pm\})\]
and 
\[c := \max (\{c_{\gamma,\gamma'} \mid \gamma,\gamma' \in \Delta^\pm, \gamma' \ne \gamma^{-1}\}).\]
By induction, for each $\gamma \in \Gamma-\{\id\}$, we have 
\[\text{radius of } B^+(\gamma) \le R\, c^{|\gamma|}.\]
\end{proof}

Let us mention an easy consequence of that result (see \cite[Corollary~6.4.12]{VIASMII} for details).

\begin{corollary}\label{cor:loxodromic}
Every element of $\Gamma - \{\id\}$ is loxodromic.
\qed
\end{corollary}

We now investigate the set~$O$.

\begin{corollary}\label{cor:intersectiondiscs}
Let $w = (w_{n})_{\ne 0}$ be a sequence of reduced words over~$\Delta^\pm$ such that the associated sequence of discs $(B^+(w_{n}))_{n\ge 0}$ is strictly decreasing. Then, the intersection $\bigcap_{n\ge 0} B^+(w_{n})$ is a single $k$-rational point~$p_{w}$. Moreover, the discs $B^+(w_{n})$ form a basis of neighborhoods of~$p_{w}$ in~$\pank$.
\end{corollary}
\begin{proof} 
We include here the idea of the proof and refer the interested reader to \cite[Corollary~6.4.13]{VIASMII} for the complete details (in the non-archimedean case, but the same arguments apply when $k$ is archimedean).\\
One checks that is enough to prove the result after extending the scalars to a finite extension~$k_{0}$ of~$k$. As a result, we may assume that $F^- \cap \pank(k) \ne \emptyset$, and then, up to changing coordinates, that $\infty \in F^-$. By Proposition~\ref{prop:radiusn}, the radius of~$B^+(w_{n})$ tends to~0 when~$n$ goes to~$\infty$, and the result follows.
\end{proof}

\begin{corollary}\label{cor:limittype1}
The set~$O$ is dense in~$\pank$ and its complement is contained in~$\pank(k)$. \qed
\end{corollary}

\subsection{Schottky uniformization}\label{sec:limitsets}

We return to the general case of Definition~\ref{def:Schottkyfigure} with an arbitrary analytic space~$S$. Here again, we follow \cite[\S 6.4.1]{VIASMII}.

\begin{definition}\label{def:limit}
We say that a point $x\in \PP^1_{S}$ is a \emph{limit point} if there exist $x_{0} \in \PP^1_{S}$ and a sequence $(\gamma_{n})_{n\ge 0}$ of distinct elements of~$\Gamma$ such that $\lim_{n\to\infty} \gamma_{n}(x_{0}) = x$.

The \emph{limit set}~$L$ of~$\Gamma$ is the set of limit points of~$\Gamma$.

\end{definition}

Following \cite[III, \S4, D\'efinition~1]{BourbakiTG14}, we say that the action of~$\Gamma$ on a subset~$E$ of~$\PP^1_{S}$ is \emph{proper} if the map
\[\begin{array}{ccc}
\Gamma \times E &\too & E \times E\\
(\gamma,x) &\mapstoo & (x,\gamma\cdot x)
\end{array}\] 
is proper, where~$\Gamma$ is endowed with the discrete topology. By \cite[III, \S4, Proposition~7]{BourbakiTG14}, it is equivalent to requiring that, for every $x, y \in E$, there exist neighborhoods~$U_{x}$ and~$U_{y}$ of~$x$ and~$y$ respectively such that the set $\{\gamma\in \Gamma : \gamma U_{x} \cap U_{y} \ne\emptyset\}$ is finite. By \cite[III, \S4, Proposition~3]{BourbakiTG14}, in this case, the quotient $\Gamma\backslash E$ is Hausdorff.

\medbreak

We denote by~$C$ the set of points $x\in \PP^1_{S}$ that admit a neighborhood~$U_{x}$ satisfying $\{\gamma \in \Gamma : \gamma U_{x} \cap U_{x} \ne\emptyset\} =\{\id\}$. 
Then $C$ is an open subset of~$\PP^1_{S}$ and the quotient map $(\PP^1_{S} - C) \to \Gamma\backslash(\PP^1_{S} - C)$ is a local homeomorphism. In particular, the topological space $\Gamma\backslash(\PP^1_{S} - C)$ is naturally endowed with a structure of analytic space \textit{via} this map.

\begin{proposition}\label{prop:actionproper}
We have $O = C = \PP^1_{S} - L$. Moreover, the action of~$\Gamma$ on~$O$ is free and proper and the quotient map $\Gamma\backslash O \to S$ is proper.
\end{proposition}
\begin{proof}
The proof is the same as that of \cite[Theorem 6.4.18]{VIASMII}. We include it for the convenience of the reader.

Let $x \in L$. By definition, there exists $x_{0} \in \PP^1_{S}$ and a sequence $(\gamma_{n})_{n\ge 0}$ of distinct elements of~$\Gamma$ such that $\lim_{n\to\infty} \gamma_{n}(x_{0}) = x$. Assume that $x\in F^+$. Since $F^+$ is contained in the interior of~$O_{1}$, there exists $N\ge 0$ such that $\gamma_{N}(x_{0}) \in O_{1}$, hence we may assume that $x_{0} \in O_{1}$. Lemma~\ref{lem:wordF} then leads to a contradiction. It follows that~$L$ does not meet~$F^+$, hence, by $\Gamma$-invariance, $L$ is contained in $\PP^1_{S} - O$.

Let $y\in \PP^1_{S} - O$. By definition, there exists a sequence $(w_{n})_{n\ge 0}$ of reduced words over~$\Delta^\pm$ such that, for each $n\ge 0$, $|w_{n}|\ge n$ and $y \in B^-(w_{n})$. Let~$y_{0} \in F^-$. By Lemma~\ref{lem:wordF}, for each $n\ge 0$, we have $w_{n}(y_{0}) \in B^-(w_{n})$ and the sequence of discs $(B^+(w_{n}))_{n\ge 0}$ is strictly decreasing. By Corollary~\ref{cor:intersectiondiscs}, $(w_{n}(y_{0}))_{n\ge 0}$ tends to~$y$, hence $y\in L$. It follows that $\PP^1_{S} - O = L$.

Set 
\[U := F^+ \cup \bigcup_{\gamma \in \Delta^\pm} \gamma F^- = \PP^1_{S} - \bigsqcup_{|\gamma|=2} B^+(\gamma).\] 
It is an open subset of~$\PP^1_{S}$ and it follows from the properties of the action (see Lemma~\ref{lem:wordF}) that we have $\{\gamma\in \Gamma \mid \gamma U \cap U \ne \emptyset\} = \{\id\} \cup \Delta^\pm$. Using the fact that the stabilizers of the points of~$U$ are trivial, we deduce that $U\subseteq C$. Letting~$\Gamma$ act, it follows that $O \subseteq C$. Since no limit point may belong to~$C$, we deduce that this is actually an equality.

\medbreak

We have already seen that the action is free on~$O$. Let us prove that it is proper. Let $x, y \in O$. There exists $n\ge 0$ such that~$x$ and~$y$ belong to the interior of~$O_{n}$. By Lemma~\ref{lem:wordF}, the set $\{\gamma\in\Gamma \mid \gamma O_{n} \cap O_{n} \ne \emptyset\}$ is made of element of length at most~$2n+1$. In particular, it is finite. We deduce that the action of~$\Gamma$ on~$O$ is proper.

Let~$K$ be a compact subset of~$S$. Since $\pi \colon \PP^1_{S} \to S$ is proper, $\pi^{-1}(K)$ is compact, hence its closed subset $F^+ \cap \pi^{-1}(K)$ is compact too. Since $F^+ \cap \pi^{-1}(K)$ contains a point of every orbit of every element of $\pi^{-1}(K)$, we deduce that $\Gamma\backslash (O \cap \pi^{-1}(K))$ is compact.
\end{proof}

\subsection{Koebe coordinates}\label{sec:Koebe}

Let $(k,\va)$ be a complete valued field and and let $\gamma$ be a loxodromic element of~$\PGL_{2}(k)$. The eigenvalues of~$\gamma$ belong to a quadratic extension of~$k$ and have distinct absolute values. If~$k$ is archimedean, it follows immediately that they both belong to~$k$, hence $\gamma$ admits exactly two fixed points $\alpha, \alpha' \in \PP^{1,\an}(k)$. If~$k$ is non-archimedean, then the result still holds by the same argument in characteristic different from~2 and, in general, as a consequence of Hensel's lemma (see \cite[I, 1.4]{GerritzenPut80}). 

We can choose~$\alpha$ so that the associated eigenvalue has minimal absolute value. In this case, $\alpha$ and~$\alpha'$ will be respectively the attracting and repelling fixed points of the M\"obius transformation associated to~$\gamma$. Denote by~$\beta$ the multiplier of~$\gamma$, \textit{i.e.} the ratio of the eigenvalues such that~$|\beta|<1$. For $\varepsilon \in \mathrm{PGL}_2(k)$ such that $\varepsilon(0) = \alpha$ and $\varepsilon(\infty) = \alpha'$, we have $\varepsilon^{-1}\gamma \varepsilon (z) = \beta z$. It follows that the parameters $\alpha, \alpha'$ and~$\beta$ determine uniquely the transformation $\gamma$. They are called the \emph{Koebe coordinates} of~$\gamma$.

Conversely, given $(\alpha = [u\colon v], \alpha' = [u'\colon v'], \beta) \in (\pank)^3$ with $\alpha \ne \alpha'$ (\emph{i.e.} $uv' \ne u'v$) and $0 < |\beta| < 1$, it is not difficult to determine explicitly the element of~$\PGL_{2}(k)$ that has those Koebe coordinates. It is given by 
\[M(\alpha,\alpha',\beta) =   
\begin{bmatrix}
uv' - \beta u' v & (\beta-1) u u'\\
(1-\beta) v v' & \beta u v' - u' v
\end{bmatrix} \in \PGL_{2}(k).
\]
In the rest of the paper, we will sometimes abuse notation and allow ourselves to use $M(\alpha,\alpha',\beta)$ in different contexts, for example when $\alpha,\alpha',\beta$ belong to a ring (provided the conditions $\alpha \ne \alpha'$ and $0 < |\beta| < 1$ are satisified at each point of its spectrum). This should not cause any trouble.

\medbreak

The following interpretation of~$|\beta|$ will be useful later. 

\begin{lemma}\label{lem:betamodulus}
Assume that~$k$ is non-archimedean. Let~$D$ be an open disc of~$\pank$ containing~$\alpha'$ and not~$\alpha$. Then $\gamma(D)$ is an open disc containing~$\alpha'$ and the segment joining the boundary point of~$D$ to that of~$\gamma(D)$ consists of points of type~2 or~3 and has length equal to~$|\beta|^{-1}$.
\end{lemma}
\begin{proof}
M\"obius transformations preserve open discs, their boundary points, and the lenght of segments. Since $\eps^{-1}\gamma\eps(\eps^{-1}(D)) = \eps^{-1}(\gamma(D))$, it is enough to prove the result for $\varepsilon^{-1}\gamma \varepsilon$ and $\eps^{-1}(D)$. In this case, it is clear.

\end{proof}

\medbreak

We now check that the Koebe coordinates depend analytically on the entries of the corresponding matrix. In fact, this is true not only over a valued field, but even over~$\Z$. To prove this, let us introduce some notation. Set
\[K_{\Z} := \{(\alpha,\alpha',\beta) \in (\PP^{1,\an}_{\Z})^3 : \alpha \ne \alpha', 0 < |\beta| < 1\}.\]
It is an open subset of $(\PP^{1,\an}_{\Z})^3$. We also consider $\PP^{3,\an}_{\Z}$ and write its elements in coordinates in the form $\begin{bmatrix} a & b\\ c & d \end{bmatrix}$ instead of the usual $[a\colon b\colon c \colon d]$. Denote by~$L_{\Z}$ the set of elements $x\in \PP^{3,\an}_{\Z}$ such that the matrix 
\[ \begin{bmatrix} a(x) & b(x)\\ c(x) & d(x) \end{bmatrix} \in \PGL_{2}(\calH(x)) \]
is loxodromic. 

\begin{lemma}
The subset~$L_{\Z}$ is open in~$\PP^{3,\an}_{\Z}$.
\end{lemma}
\begin{proof}
Let us first consider the archimedean part~$L_{\Z}^\text{a}$ of~$L_{\Z}$. By Section~\ref{sec:BerkovichZ}, it is enough to prove that its intersection with the fiber over the point~$a_{\infty}^1$, corresponding to the usual absolute value, is open. This allows to translate the statement into a statement about $\PP^3(\C)$ (since the set is clearly stable by complex conjugation), where it is a consequence of the continuity of the roots of a (degree~2) polynomial. 

Let us now handle the non-archimedean part~$L_{\Z}^\text{na}$. By Lemma~\ref{lem:distinctabsval}, we have 
\[L_{Z}^\text{na} = \left\{ \begin{bmatrix} a & b\\ c & d \end{bmatrix} \in \big(\PP^{3,\an}_{\Z}\big)^\text{na} : |ad-bc| < |a+d|^2 \right\}.\]
Let $x \in L_{Z}^\text{na}$. There exists $r >1$ such that $r\, |ad-bc| < |a+d|^2$. The open subset of $\PP^{3,\an}_{\Z}$ defined by the inequality 
\[r\, \max(|4|,1)\,  |ad-bc| < |a+d|^2\]
contains~$x$ and sits inside~$L_{\Z}$, by Lemma~\ref{lem:distinctabsval} again. The result follows.
\end{proof}

\begin{proposition}\label{prop:Koebe}
The morphism
\[ M \colon (\alpha,\alpha',\beta) \in K_{\Z} \mapstoo M(\alpha,\alpha',\beta) \in L_{\Z}\]
is an isomorphism of analytic spaces over~$\Z$. Its inverse is the map that associates to a loxodromic matrix its Koebe coordinates.
\end{proposition}
\begin{proof}
The map~$M$ is clearly analytic and it follows from the discussion above that it is a bijection. Let us prove that its inverse is also analytic.

Let $m\in L_{\Z}$. We may work in an affine chart of~$\PP^{3,\an}_{\Z}$ containing~$m$ and identify it to~$\AA^{3,\an}_{\Z}$. As a result, we may assume that the coefficients $a(m),b(m),c(m),d(m)$ of~$m$ are well-defined. Denote by~$\lambda(m)$ and  $\lambda'(m)$ the eigenvalues of the matrix associated to~$m$, chosen so that $|\lambda(m)| < |\lambda'(m)|$. Remark that the inequality on the absolute values implies that $(a+d)(m)\ne 0$. The elements $\lambda(m)$ and~$\lambda'(m)$ are then the two roots of the characteristic polynomial of the matrix:
\[X^2 - (a+d)(m) X + (ad-bc)(m) = (a+d)^2(m)\, \Big(Y^2 - Y + \frac{ad-bc}{(a+d)^2}(m)\Big),\]
where $Y = \frac{1}{(a+d)(m)}\, X$. 

Note that the polynomial $P(Y) := Y^2-Y + (ad-bc)/(a+d)^2$ (which is actually well-defined on the whole~$L_{\Z}$) has analytic coefficients. We claim that~$\lambda$ and~$\lambda'$ are analytic functions of~$m$. If~$m$ is archimedean and the discriminant~$\Delta(m)$ of~$P(m)(Y)$ is not real, this follows from the fact that there exists a determination of the square-root that is analytic in the neighborhood of~$\Delta(m)$. If~$m$ is archimedean and $\Delta(m)$ is real, then $\Delta(m) > 0$, since otherwise $\lambda(m)$ and $\lambda'(m)$ would be complex conjugates hence would have the same absolute value. The result then follows from the fact that there exists a determination of the square-root that is analytic in the neighborhood of~$\Delta(m)$ and commutes with the complex conjugation.  

Assume that~$m$ is non-archimedean. The stalk~$\calO_{m}$ of the structure sheaf is a local ring (whose maximal ideal is the set of elements that vanish at~$m$). Denote its residue field by~$\kappa(m)$ and set 
\[\kappa(m)^\circ := \{f\in \kappa(m) : |f(m)| \le 1\}\]
and 
\[\kappa(m)^{\circ\circ} := \{f\in \kappa(m) : |f(m)| < 1\}.\]
The set~$\kappa(m)^\circ$ is a local ring with maximal ideal~$\kappa(m)^{\circ\circ}$. We denote its residue field by~$\tilde{\kappa}(m)$. By Lemma~\ref{lem:distinctabsval}, the image of $P(Y)$ in~$\kappa(m)[Y]$ has coefficients in~$\kappa(m)^\circ$ and its reduction is $Y^2-Y$. The roots~$\lambda(m)$ and~$\lambda'(m)$ of~$P(m)(Y)$ reduce respectively to the roots~0 and~1 of $Y^2-Y$. By \cite[Corollaire~5.3]{Poineau13} and \cite[Corollaire~2.5.2]{A1Z}, $\kappa(m)^\circ$ and~$\calO_{m}$ are henselian, and it follows that~$\lambda$ and~$\lambda'$ are analytic in the neighborhood of~$m$. 

It is now clear that $\beta = \lambda/\lambda'$ is analytic in the neighborhood of~$m$. Note that we can also recover~$\alpha$ and~$\alpha'$ from~$\lambda$ and~$\lambda'$ since they correspond to the associated eigenline. More precisely, we have $\alpha(m) = [b(m) \colon (\lambda - a)(m)]$ if $\lambda(m) \ne a(m)$ and $\alpha(m) = [(\lambda-d)(m) \colon c(m)]$ otherwise, and similarly for~$\alpha'$. It follows that~$\alpha$ and~$\alpha'$ are analytic in the neighborhood of~$m$.
\end{proof}

\subsection{Group theory}\label{sec:groupSchottky}
Let $(k,\va)$ be a complete valued field. In this section, we give the general definition of Schottky group over~$k$ and explain how it relates to the geometric situation considered in Section~\ref{sec:geometricSchottky}. Here, we borrow from \cite[\S 6.4.2 and \S 6.4.3]{VIASMII}.

\begin{definition}
A discrete subgroup~$\Gamma$ of $\PGL_2(k)$ is said to be a \emph{Schottky group over~$k$} if
\begin{enumerate}
\item it is free and finitely generated;
\item all its non-trivial elements are loxodromic;
\item there exists a non-empty $\Gamma$-invariant connected open subset of~$\pank$ on which the action of~$\Gamma$ is free and proper (cf. Section \ref{sec:limitsets}).
\end{enumerate}
\end{definition}



\begin{remark}\label{rem:Schottkyva}
The notion of Schottky group over~$k$ is left unchanged if one replaces the absolute value on~$k$ by an equivalent one.
\end{remark}

\begin{remark}\label{rem:SchottkyKL}
Let $(k',\va')$ be an extension of~$(k,\va)$. A subgroup~$\Gamma$ of $\mathrm{PGL}_2(k)$ is a Schottky group over~$k$ if, and only if, it is a Schottky group over~$k'$.
\end{remark}

Schottky groups arise naturally from Schottky figures, as the following proposition shows.

\begin{proposition}\label{prop:geometrygroup}
Let~$\Gamma$ be a discrete subgroup of $\PGL_{2}(k)$ generated by finitely many elements $\gamma_{1},\dotsc,\gamma_{g}$. 
If there exists a Schottky figure adapted to $(\gamma_{1},\dotsc,\gamma_{g})$, then~$\Gamma$ is a Schottky group.
\end{proposition}
\begin{proof}

The group $\Gamma$ satisfies (i) by Corollary~\ref{cor:free}, (ii) by Corollary~\ref{cor:loxodromic} and (iii) by Corollary~\ref{cor:limittype1}.

\end{proof}

\begin{definition}
Let $\gamma = \begin{bmatrix} 
a & b \\
c & d
\end{bmatrix} \in \PGL_2(k)$, with $c\ne 0$, be a loxodromic matrix
and let $\lambda \in \R_{>0}$ be a positive real number. We call open and closed \emph{twisted Ford discs} associated to $(\gamma,\lambda)$ the sets 
\[D_{(\gamma, \lambda)}^- := \Big\{z\in k \ \Big|\ \lambda |\gamma'(z)|=\lambda \frac{|ad-bc|}{|cz+d|^2} > 1\Big\}\]
and
\[D_{(\gamma, \lambda)}^+ := \Big\{z\in k \ \Big|\  \lambda |\gamma'(z)|=\lambda \frac{|ad-bc|}{|cz+d|^2} \ge 1\Big\}.\]

\end{definition}


\begin{lemma}\label{lem:explicitisometricdiscs}
Let $\alpha,\alpha',\beta \in k$ with $|\beta|<1$ and let $\lambda \in \R_{>0}$. Set $\gamma= M(\alpha, \alpha', \beta) = \begin{bmatrix} a&b\\c&d \end{bmatrix}$. The twisted Ford discs $D_{(\gamma, \lambda)}^-$ and $D_{(\gamma, \lambda)}^+$ have center 
\[\frac{\alpha' - \beta \alpha}{1-\beta} = -\frac d c\]
and radius 
\[\rho=\frac{(\lambda|\beta|)^{1/2}|\alpha-\alpha'|}{|1-\beta|} = \frac{(\lambda \,|ad-bc|)^{1/2}}{|c|}.\]
The twisted Ford discs $D_{(\gamma^{-1}, \lambda^{-1})}^-$ and $D_{(\gamma^{-1}, \lambda^{-1})}^+$ have center 
\[\frac{\alpha - \beta \alpha'}{1-\beta} = \frac a c\]
and radius $\rho' =\rho/\lambda$. 
\qed
\end{lemma}

\begin{lemma}
For every loxodromic $\gamma\in \mathrm{PGL}_2(k)$ that does not fix~$\infty$ and every $\lambda\in \R_{>0}$, we have $\gamma(D_{(\gamma, \lambda)}^+)= \pank - {D}_{(\gamma^{-1}, \lambda^{-1})}^-$.\qed
\end{lemma}


\begin{definition}\label{des:Schottkybasis}
Let~$\Gamma$ be a Schottky group of rank~$g$. We say that a basis $B=(\gamma_{1},\dotsc,\gamma_{g})$ of~$\Gamma$ is a \emph{Schottky basis} if there exists a Schottky figure that is adapted to it.
The datum $(\Gamma, B)$ of a Schottky group and a Schottky basis $B$ of~$\Gamma$ is called a \emph{marked Schottky group}.
\end{definition}

If $k$ is archimedean, it is a classical result of Marden \cite{Marden74} that there exist Schottky groups with no Schottky bases.
On the contrary, in the non-archimedean case, a theorem of Gerritzen ensures that Schottky bases always exist (see \cite[\S 2, Satz 1]{Gerritzen74}).
We rephrase it here using the notation of the present paper (see \cite[Corollary 6.4.32]{VIASMII} for a proof using Berkovich geometry).

\begin{theorem}[Gerritzen]\label{Gerritzen}
Assume that~$k$ is non-archimedean. Let $\Gamma$ be a Schottky group over~$k$ whose limit set does not contain~$\infty$. Then, there exist a basis $(\delta_1, \dots, \delta_g)$ of~$\Gamma$ and positive real numbers $\lambda_1, \dots, \lambda_g \in \R_{>0}$ such that the family of twisted Ford discs $\big (D_{\delta_1, \lambda_1}^+, \dots, D_{\delta_g, \lambda_g}^+, D_{\delta^{-1}_1, \lambda^{-1}_1}^+, \dots, D_{\delta^{-1}_g, \lambda^{-1}_g}^+\big )$ is a Schottky figure adapted to $(\delta_1, \dots, \delta_g)$.\qed
\end{theorem}

\section{The Schottky space over $\Z$}
In this section, we define a parameter space for marked Schottky groups of a given rank, where the marking is given by the choice of a basis.
Already for Schottky groups of rank one, one gets an interesting construction, but most uniformization phenomena that are at the center of our interest become apparent only when the rank is at least two.

\subsection{The space $\calS_1$}\label{sec:S1} \text{ }\\
Let $\Gamma=\langle \gamma \rangle$ be a Schottky group of rank one over a complete valued field~$(k, |\cdot|)$. 
Then, $\gamma$ is conjugated in $PGL_2(k)$ to a unique matrix of the form $M(0,\infty, \beta)$ with $0 < |\beta| < 1$ (which corresponds to the multiplication by~$\beta$ as an endomorphism of~$\pank$; see Section~\ref{sec:Koebe} for the notation).

Consider the affine line $\A^{1,\an}_{\Z}$ with coordinate~$Y$ and set 
\[\calS_{1} := \{x\in \A^{1,\an}_{\Z} : 0<|Y(x)|<1 \}.\]
With each point $x\in \calS_1$, one can canonically associate a Schottky group of rank one 
\[\Gamma_{x} := \langle M(0,\infty, Y(x)) \rangle \subset \PGL_2(\calH(x)).\]
The condition imposed on $\calS_1$ ensures that $M(0,\infty, Y(x))$ is a loxodromic transformation of $\mathbb{P}^{1,\mathrm{an}}_{\calH(x)}$ having $0$ as attracting point and $\infty$ as repelling point.

Given a Schottky group of rank one over $\Q_p$, we can retrieve it as $\langle M(0,\infty, Y(x))\rangle$ for a unique $x \in \calS_1$ with $\calH(x) = \Q_p$.
For a general valued field $(k, |\cdot|)$ and an element $\beta\in k$ such that $0<|\beta|<1$, the group generated by $M(0,\infty, \beta)$ can be retrieved as above from a point of $\calS_1 \times_\Z k$. 
This can be seen as a consequence of Lemma \ref{lem:kpoints} below in the special case where $g=1$.

\subsection{Construction of $\calS_{g}$ and equivalent definitions}\label{subsec:Sspace} \text{ }\\
Fix $g\ge 2$. Consider the space $\AA^{3g-3,\an}_{\Z}$ and denote its coordinates by $X_{3},\dots,X_{g},X'_{2},\dots ,X'_{g},Y_{1},\dots ,Y_{g}$. 
For notational convenience, set $X_{1}:=0$, $X_{2}:=1$ and $X'_{1}:=\infty$ (seen as morphisms from $\AA^{3g-3,\an}_{\Z}$ to $\PP^{1,\an}_{\Z}$). 
Denote by $\pr_{\Z} \colon \AA^{3g-3,\an}_{\Z}\to\calM(\Z)$ the projection morphism. 
 Let~$U_{g}$ be the open subset of~$\AA^{3g-3, \an}_\Z$ defined by the inequalities
\[\begin{cases}
0<|Y_{i}| < 1 \text{ for } 1\le i\le g;\\
X^{\sigma_{i}}_i \neq X^{\sigma_{j}}_j \text{ for } i,j \in\{1,\dotsc,g\} \text{ and } \sigma_{i},\sigma_{j} \in \{\emptyset,'\}.
\end{cases}\]
For $i \in \{1,\dotsc,g\}$, consider the transformations
\[ M_{i} := M(X_{i},X'_{i},Y_{i})   \in \PGL_{2}(\calO(U_{g})).\]

\begin{definition}\label{def:Schottkyspace}
The \emph{Schottky space of rank~$g$} over $\Z$, denoted by $\calS_g$, is the set of points $x\in U_{g}$ such that the subgroup $\Gamma_x$ of $\PGL_{2}(\calH(x))$ defined by
\[\langle M_{1}(x), M_{2}(x), \dots, M_g(x) \rangle\]
is a Schottky group of rank~$g$.

\begin{notation}\label{not:GammaxCx}
Recall that a Schottky group over a complete valued field $(k, |\cdot|)$ gives rise to a unique $k$-analytic curve by means of the uniformization described in Section \ref{sec:limitsets}. 
Given $x\in \calS_g$, we denote by $\Gamma_x$ the marked Schottky group of ordered basis $(M_{1}(x), M_{2}(x), \dots, M_g(x))$, and by $\calC_x$ the $\calH(x)$-analytic curve obtained via Schottky uniformization by $\Gamma_x$.

In the non-archimedean case, the curve $\calC_x$ has semi-stable reduction, and results of Berkovich (\cite[Section 4.3]{Berkovich90}) then assert that the dual graph of the stable model of $\calC_x$ can be canonically realized as a subset of $\calC_x$.
Such a subset, denoted by $\Sigma_x$, is a graph of Betti number $g$, called the \emph{skeleton} of~$\calC_x$.
If $\calB=\big({D_{i,\varepsilon}^{+}},\ 1\leq i \leq g, \varepsilon=\pm 1\big)$ is a Schottky figure adapted to a Schottky basis of $\Gamma_x$ and $F^{+}$ is the associated ``closure of a fundamental domain'' (see Definition \ref{def:Schottkyfigure} and following paragraph), then there is an isomorphism $\calC_x \cong F^{+}  / \Gamma_x$ and the skeleton~$\Sigma_x$ of the Mumford curve~$\calC_x$ is obtained through pairwise identification, for every $i=1,\dots,g$, of the points of the Shilov boundaries of $D_{i,1}^{+}, D_{i,-1}^{+}$ in the tree corresponding to the skeleton of $F^{+}$ (see Figure \ref{fig:fundom} for an example with $g=2$). We refer the reader to~\cite[Theorem~6.4.18]{VIASMII} for a proof of this fact.

If $k=\C$, there might not exist a Schottky basis for $\Gamma_x$, but an analogue of the set $F^+$ can be built by replacing the Schottky figure $\calB$ with a $2g$-uple of Jordan curves $\big({C_{i,\eps}},\ 1\leq i \leq g, \varepsilon=\pm 1\big)$ in $\mathbb{P}^1_\C$ with disjoint interiors and such that $\gamma_{i} (C_{i,1})=C_{i, -1}$ for every $i=1,\dots,g$.
In this context, the closed set $F^+$ is the complement of the interiors of these curves.
The quotient map $F^+\to \calC_x$ sends the Jordan curves $C_{1,1},\dots, C_{g,1}$ into smooth simple non-intersecting curves $\alpha_1,\dots, \alpha_g$ inside~$\calC_x$.
The complex description allows to easily treat the case $k=\R$.
\end{notation}

\begin{figure}[ht]
\centering
    \begin{subfigure}[b]{0.44\textwidth}
\includegraphics[scale=.3]{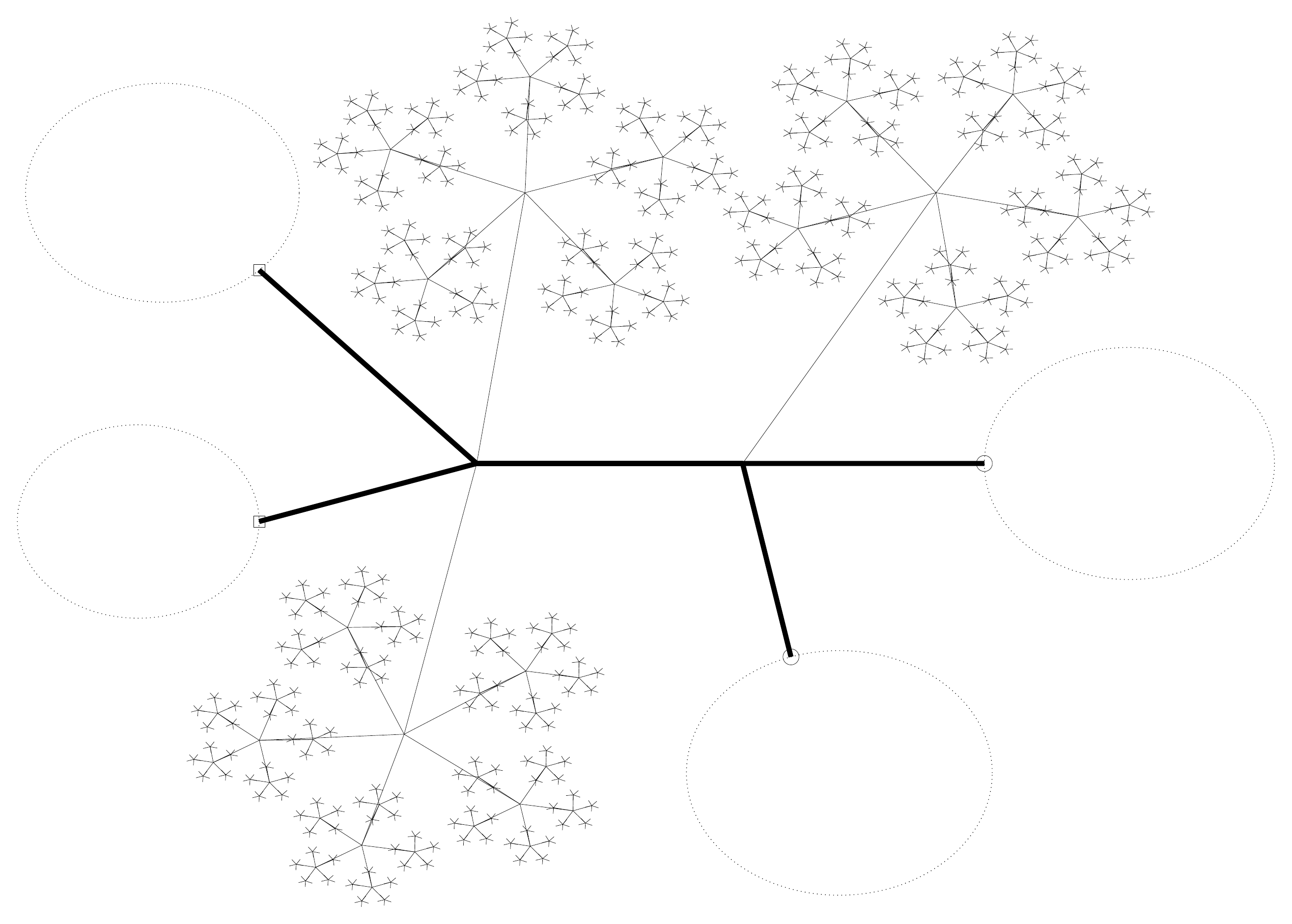}
    \end{subfigure}
    \hspace{1.5cm}
    \begin{subfigure}[b]{0.44\textwidth}
\includegraphics[scale=.3]{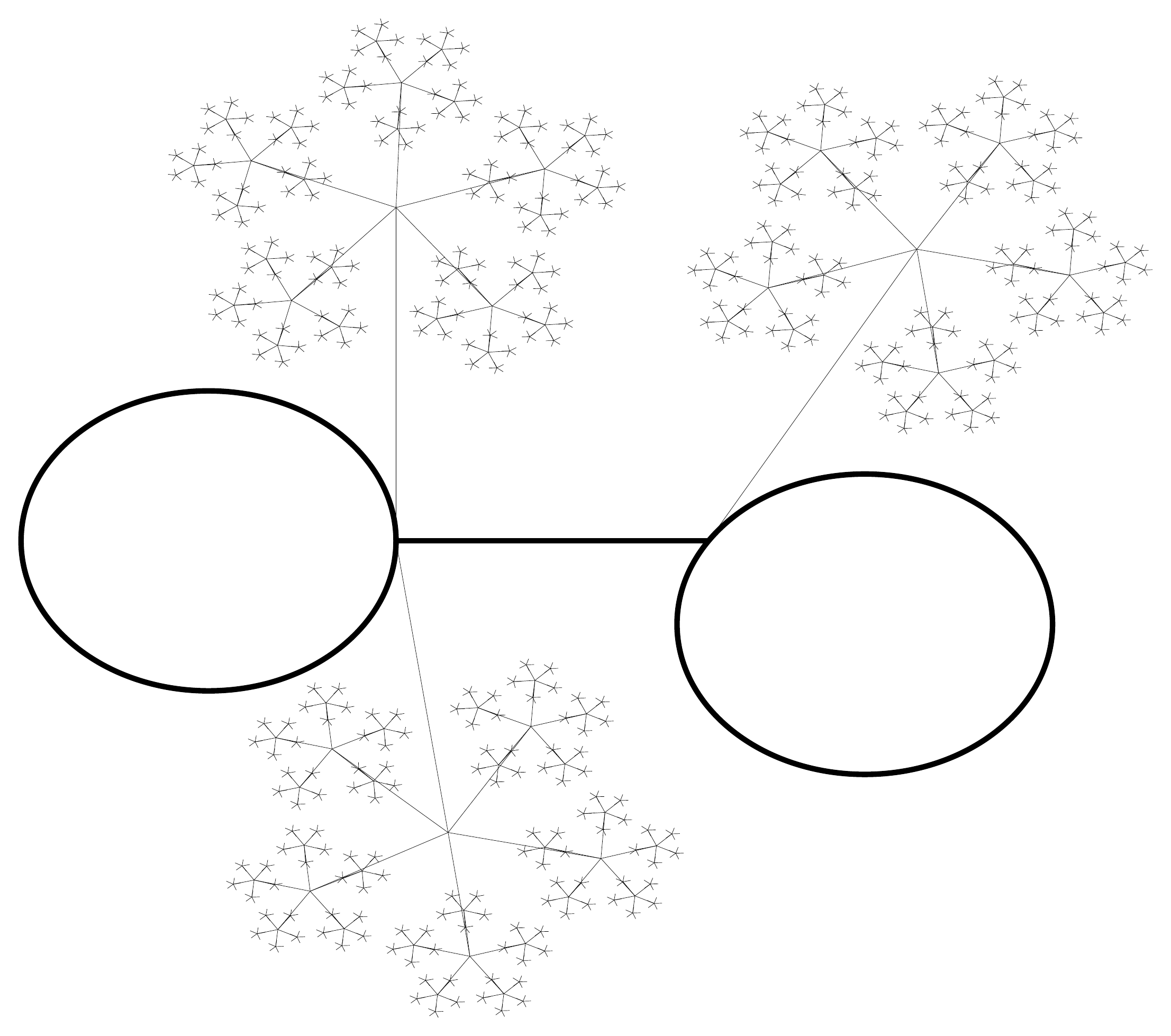}
    \end{subfigure}

 \caption{The domain $F^+$ of a Schottky figure for the Schottky group $\Gamma_x$ in the non-Archimedean case is depicted on the left. The group $\Gamma_x$ identifies the ends of the skeleton $\Sigma_{F^+}$, so that the corresponding Mumford curve (on the right) contains the finite graph $\Sigma_x$.}
 \label{fig:fundom}
\end{figure}

Note that the identification of $\calH(x)$ with a valued extension of $\calH(\mathrm{pr}_{\Z}(x))$ is not canonical, and to different immersions $\calH(\mathrm{pr}_{\Z}(x)) \hookrightarrow \calH(x)$ one associates different Schottky groups, yielding curves $\calC_x$ that might not be isomorphic.
To lift the ambiguity from this situation one has then to consider the points of the base-change $\calS_g \times_\Z \calH(x)$, as made more precise in the following remark.
\end{definition}

\begin{remark}\label{rem:SgA}
Let~$(A,\nm)$ be a Banach ring. Starting with $\calM(A)$ instead of $\calM(\Z)$, one can define a Schottky space~$\calS_{g,A}$ over~$A$. It is related to the Schottky space over~$\Z$ in the following way.
If we denote by $\pi_{A} \colon \AA^{3g-3,\an}_{A} \to \AA^{3g-3,\an}_{\Z}$ the projection map, it follows from Remark~\ref{rem:SchottkyKL} that we have $\calS_{g,A} = \pi_{A}^{-1}(\calS_{g})$. In other words, assuming that the suitable categories and fiber products are defined, we have $\calS_{g,A} = \calS_{g} \times_{\calM(\Z)} \calM(A)$. 
Moreover, the projection map~$\pi_{A}$ respects all the data: for each $x\in \calS_{g,A}$, the group~$\Gamma_{x}$ is the image of~$\Gamma_{\pi(x)}$ by the inclusion $\PGL_{2}(\calH(\pi(x)) \subseteq \PGL_{2}(\calH(x))$, the limit set of~$\Gamma_{\pi(x)}$ in~$\PP^{1,\an}_{\calH(\pi(x))}$ is the preimage of the limit set of~$\Gamma_{x}$ in~$\PP^{1,\an}_{\calH(x)}$, and so on.

In the special case where $A = \C$, we obtain a subset~$\calS_{g,\C}$ of~$\C^{3g-3}$. As one may expect, this is a classical object that has already been closely investigated. 
By \cite[Lemma 5.11]{Hejhal75}, one has a covering map $\calT_{g,\C} \to \calS_{g,\C}$ from the Teichm\"uller space to the complex Schottky space.
In particular, one deduces that $\calS_{g,\C}$ is a connected subset of~$\C^{3g-3}$. See also \cite[Proposition~2]{Bers75} for a direct proof.
\end{remark}

Recall that, for $\eps \in (0,1]$, we denote by~$a^\eps_{\infty}$ the point of~$\calM(\Z)$ associated with~$\va_{\infty}^\eps$ and that we have an isomorphism $\calH(a^\eps_{\infty}) \simeq \R$. We will use the canonical map $\rho_{\eps} \colon \C^{3g-3} = \AA^{3g-3,\an}_\C \to \AA^{3g-3,\an}_\R$, where~$\R$ and~$\C$ are endowed with~$\va_{\infty}^\eps$. See Section~\ref{sec:BerkovichZ} for details.

\begin{lemma}\label{lem:Sga}
For $\eps \in (0,1]$, we have $\calS_{g} \cap \pr_{\Z}^{-1}(a_{\infty}^\eps) = \rho_\eps(\calS_{g,\C})$. 
The set $\calS_{g} \cap \big(\AA^{3g-3,\an}_{\Z}\big)^\mathrm{a}$ is a connected open subset of $\AA^{3g-3,\an}_\Z$.
\end{lemma}
\begin{proof}
By Remarks~\ref{rem:Schottkyva} and~\ref{rem:SchottkyKL}, $\rho_{\eps}^{-1}(\calS_{g} \cap \pr_{\Z}^{-1}(a_{\infty}^\eps))$ coincides with the usual complex Schottky space $\calS_{g,\C}$. In other words, $\calS_{g} \cap \pr_{\Z}^{-1}(a_{\infty}^\eps)$ is the quotient of~$\calS_{g,\C}$ by the complex conjugation. In particular, it is a connected open subset of $\A^{3g-3,\an}_\R$.

Recall the homeomorphism
\[ \Phi \colon \AA^{3g-3,\an}_{\R} \times (0,1] \to \big(\AA^{3g-3,\an}_{\Z}\big)^\text{a}\]
from Section~\ref{sec:BerkovichZ}. It follows from Remark~\ref{rem:Schottkyva} that it induces a bijection between $\big(\calS_{g} \cap \pr_{\Z}^{-1}(a_{\infty}^\eps)\big) \times (0,1]$ and $\calS_{g}^\text{a}$. As a consequence, $\calS_{g}^\text{a}$ is connected and open, since $\big(\AA^{3g-3,\an}_{\Z}\big)^\text{a}$ is open in $\AA^{3g-3,\an}_{\Z}$.

\end{proof}

Let~$F_{g}$ be the free group of rank~$g$ with basis $e_{1},\dotsc,e_{g}$. For each complete valued field~$k$, we denote by $\Hom_{S}(F_{g},\PGL_{2}(k))$ the set of group morphisms $\varphi \colon F_{g} \to \PGL_{2}(k)$ that satisfy the following conditions:
\begin{enumerate}
\item $\varphi(e_{1})$ is loxodromic with attracting fixed point~0 and repelling fixed point~$\infty$;
\item $\varphi(e_{2})$ is loxodromic with attracting fixed point~1;
\item the image of~$\varphi$ is a Schottky group of rank~$g$.
\end{enumerate}
Each point~$x$ of~$\calS_{g}$ gives rise to an element $\varphi_{x}$ of~$\Hom_{S}(F_{g},\PGL_{2}(\calH(x)))$ that sends $e_{i}$ to $M_{i}(x)$. 

To state a converse result, we need to introduce an equivalence relation similar to that of \cite[Remark 1.2.2 (ii)]{Berkovich90}. We say that two elements $\varphi_{1} \in \Hom_{S}(F_{g},\PGL_{2}(k_{1}))$ and $\varphi_{2} \in \Hom_{S}(F_{g},\PGL_{2}(k_{2}))$ are equivalent if there exists an element $\varphi \in \Hom_{S}(F_{g},\PGL_{2}(k))$ and isometric embeddings $k \hookrightarrow k_{1}$ and $k \hookrightarrow k_{2}$ that make the following diagram commute:

\[\begin{tikzcd}
& & & \PGL_{2}(k_{1})\\ 
F_g \ar[rr, "\varphi"]  \ar[rrrd,"\varphi_2"', bend right=20] \ar[rrru,"\varphi_1", bend left=20] & &  \PGL_{2}(k) \ar[rd,hook] \ar[ru,hook]\\    
&  & & \PGL_{2}(k_{2})
\end{tikzcd}\]


We denote by $\Hom_{S}(F_{g},\PGL_{2})$ the set of classes of this equivalence relation.

\begin{lemma}\label{lem:kpoints}
The map $x \mapsto \varphi_{x}$ is a bijection between the underlying set of~$\calS_{g}$ and $\Hom_{S}(F_{g},\PGL_{2})$.
\end{lemma}

\begin{proof}
If $\varphi_x=\varphi_y$ then $\calH(x)=\calH(y)$ and $\varphi_x(F_g)=\varphi_y(F_g)$ coincide as marked Schottky groups in $\PGL_{2}(\calH(x))$. 
Hence $x=y$ in $\calS_{g}$ and $x \mapsto \varphi_x$ is injective.
Conversely, for a valued field $k$ and a map $\varphi\in \Hom_{S}(F_{g},\PGL_{2}(k))$, we can consider the point $y \in \AA^{3g-3,\an}_k$ given by the Koebe coordinates of the marked Schottky group $\varphi(F_{g})$.
The image of $y$ under the canonical projection $\AA^{3g-3,\an}_k \to \AA^{3g-3,\an}_\Z$ is a point $x \in \calS_g$ and there is an isometric embedding $\calH(x) \hookrightarrow k$ realizing $\varphi_x$ as canonical representative of the class of $\varphi$ in $\Hom_{S}(F_{g},\PGL_{2})$.
Hence $x \mapsto \varphi_x$ is surjective.
\end{proof}

\begin{remark}
Every marked Schottky group of rank $g$ over a complete valued field $k$ is conjugated in $\PGL_2(k)$ to a unique marked Schottky group with the property that its Koebe coordinates are of the form $\{ (0, \infty, \beta_1), (1, \alpha'_2, \beta_2), \dots, (\alpha_g, \alpha'_g, \beta_g) \}$.
The combination of this observation with Lemma~\ref{lem:kpoints} implies that every Schottky group over $k$ can be retrieved, up to conjugation, as $\Gamma_x$ for some suitable point $x\in\calS_{g, k}$.
\end{remark}

\subsection{Openness of $\calS_g$}
\begin{definition}
Let~$S$ be an analytic space. Consider the relative affine line~$\AA^1_{S}$ with coordinate~$Z$. For $\gamma = \begin{bmatrix} a&b\\ c&d\end{bmatrix}$ in $\PGL_{2}(\calO(S))$ and $\lambda\in \R_{>0}$, set 
\[D^+_{(\gamma,\lambda)} := \{x\in \AA^1_{S} : |(c Z +d)(x)|^2 \le \lambda |(ad-bc)(x)|\}\]
and
\[D^-_{(\gamma,\lambda)} := \{x\in \AA^1_{S} : |(c Z +d)(x)|^2 < \lambda |(ad-bc)(x)|\}.\]
We call such sets closed and open \emph{relative twisted Ford discs} respectively.
\end{definition}

We now generalize Gerritzen's theorem~\ref{Gerritzen} to the relative setting. 

\begin{theorem}\label{thm:Gerritzenvoisinageinfinity}
Let~$x$ be a non-archimedean point of~$\calS_{g}$ such that~$\infty$ is not a limit point of~$\Gamma_{x}$. There exist an open neighborhood~$W$ of~$x$ in~$U_{g}$, an automorphism~$\tau\in \Aut(F_{g})$ and positive real numbers $\lambda_{1},\dotsc,\lambda_{g} \in \R_{>0}$ such that, denoting
\[(N_{1},\dotsc,N_{g}) := \tau \cdot (M_{1},\dotsc,M_{g}) \in \GL_{2}(\calO(U_{g}))^g,\]
the family of relative twisted Ford discs over~$W$
\[\big(D^+_{N_{1},\lambda_{1}},\dotsc, D^+_{N_{g},\lambda_{g}}, D^+_{N_{1}^{-1},\lambda_{1}^{-1}},\dotsc,D^+_{N_{g}^{-1},\lambda_{g}^{-1}}\big)\] 
is a Schottky figure adapted to the family $(N_{1},\dotsc,N_{g})$ of $\PGL_{2}(\calO(W))$.

\end{theorem}
\begin{proof}
By Theorem~\ref{Gerritzen}, we can find a basis $(\delta_1, \dotsc, \delta_g)$ of~$\Gamma_{x}$ and positive real numbers $\lambda_1,\dotsc,\lambda_g$ such that the family of twisted Ford discs $\big (D_{\delta_1, \lambda_1}^+, \dots, D_{\delta_g, \lambda_g}^+, D_{\delta^{-1}_1, \lambda^{-1}_1}^+, \dots, D_{\delta^{-1}_g, \lambda^{-1}_g}^+\big )$ is a Schottky figure adapted to~$(\delta_1, \dotsc, \delta_g)$. 

Denote by~$\tau$ the automorphism of~$\Gamma_{x}$ sending~$M_{i}(x)$ to~$\delta_{i}$. Identifying~$F_{g}$ with~$\Gamma_{x}$ by sending~$e_{i}$ to~$M_{i}(x)$, we get an automorphism of~$F_{g}$ that we still denote by~$\tau$. Set 
\[(N_{1},\dotsc,N_{g}) := \tau \cdot (M_{1},\dotsc,M_{g}) \in \GL_{2}(\calO(U_{g}))^g\]
and write 
\[N_{i} = \begin{pmatrix} a_{i}&b_{i}\\c_{i}& d_{i}\end{pmatrix}\]
for $i\in\{1,\dotsc,g\}$. Note that the coefficients of the~$N_{i}$'s are rational functions in the $X_{j}$'s, $X'_{j}$'s and $Y_{j}$'s. Denote by~$U'$ the open subset of~$U_{g}$ where they are all defined.

Let~$V$ be the open subset of~$U'$ defined by
\[|\tr(N_{i})|^2 > \max(|4|,1)\, |\det(N_{i})| \text{ for } 1 \le i\le g.\]
By Lemma~\ref{lem:distinctabsval}, $V$ contains~$x$ and, for each $y\in V$ and each $i\in\{1,\dotsc,g\}$, the matrix~$N_{i}$ is loxodromic. 

\medbreak

Let $i\ne j \in \{1,\dotsc,g\}$. Since $\infty$ is not a limit point of~$\Gamma_{x}$, it cannot be a fixed point of~$N_{i}(x)$ or~$N_{j}(x)$, hence $c_{i}(x) c_{j}(x) \ne 0$. There exists a neighborhood~$W_{i,j}$ of~$x$ in~$V$ such that $c_{i}c_{j}$ does not vanish on~$W_{i,j}$. In this case, for each $y\in W_{i,j}$, $\infty$ is not a fixed point of~$N_{i}(y)$ or~$N_{j}(y)$ and, by Lemma~\ref{lem:explicitisometricdiscs}, we have
\[D^+_{N_{i}(y),\lambda_{i}} = D^+\left(-\frac{d_{i}}{c_{i}}(y),  \left|\frac{a_{i}d_{i} - b_{i}c_{i}}{c_{i}^2}(y)\right|^{1/2} \lambda_{i}^{1/2}\right)\]
and 
\[D^+_{N_{j}(y),\lambda_{j}} = D^+\left(-\frac{d_{j}}{c_{j}}(y),  \left|\frac{a_{j}d_{j} - b_{j}c_{j}}{c_{j}^2}(y)\right|^{1/2} \lambda_{j}^{1/2}\right).\]
By assumption, the discs at~$x$ are disjoint and Lemma~\ref{lem:2rho} ensures that we have
\[\left| \frac{d_{i}}{c_{i}}(x) -\frac{d_{j}}{c_{j}}(x)\right| > \max \left( \left|\frac{a_{i}d_{i} - b_{i}c_{i}}{c_{i}^2}(x)\right|^{1/2} \lambda_{i}^{1/2},  \left|\frac{a_{j}d_{j} - b_{j}c_{j}}{c_{j}^2}(x)\right|^{1/2} \lambda_{j}^{1/2}\right).\]
Since~$x$ is non-archimedean, we have $\max(|2(x)|,1) = 1$, hence, up to shrinking~$W_{i,j}$, we may assume that, for each $y\in W$, we have
\[\left| \frac{d_{i}}{c_{i}}(y) -\frac{d_{j}}{c_{j}}(y)\right| > \max(|2(y)|,1) \max \left( \left|\frac{a_{i}d_{i} - b_{i}c_{i}}{c_{i}^2}(y)\right|^{1/2} \lambda_{i}^{1/2},  \left|\frac{a_{j}d_{j} - b_{j}c_{j}}{c_{j}^2}(y)\right|^{1/2} \lambda_{j}^{1/2}\right),\]
which implies that $D^+_{N_{i}(y),\lambda_{i}}$ and $D^+_{N_{j}(y),\lambda_{j}}$ are disjoint, by Lemma~\ref{lem:2rho}. Similar arguments show that, up to shrinking~$W_{i,j}$, we may ensure that the discs $D^+_{N_{i}(y),\lambda_{i}}$, $D^+_{N^{-1}_{i}(y),\lambda^{-1}_{i}}$, $D^+_{N_{j}(y),\lambda_{j}}$ and $D^+_{N^{-1}_{j}(y),\lambda^{-1}_{j}}$ are all disjoint.

The result now holds with $W := \bigcap_{i\ne j} W_{i,j}$.
\end{proof}

\begin{corollary}\label{cor:Gerritzenvoisinage}
Let~$x$ be a non-archimedean point of~$\calS_{g}$. There exist an open neighborhood~$W$ of~$x$ in~$U_{g}$, an automorphism~$\tau$ of~$F_{g}$ and a family of closed subsets of~$\PP^1_{W}$ that is a Schottky figure adapted to $\tau \cdot (M_{1},\dotsc,M_{g})$.
\end{corollary}
\begin{proof}
We will distinguish two cases. 

\medbreak


$\bullet$ Assume that the extension $\calH(x)/\calH(\pr_{\Z}(x))$ is finite.

Then, there exists an algebraic integer that does not belong to~$\calH(x)$, in the sense that there exists a polynomial $P\in\Z[T]$ with no roots in~$\calH(x)$. If $\pr_{\Z}(x)=a_{0}$, so that $\calH(\pr_{\Z}(x))=\Q$, we may moreover assume that $P$ is totally real: all of its complex roots are real. Let~$K$ be a number field containing a root~$\omega$ of~$P$. Note that $\omega\in \calO_{K}$. 

We will work over the Schottky space $\calS_{g,\calO_{K}}$ over $\calM(\calO_{K})$ defined in Remark~\ref{rem:SgA}. Denote by $p_{K} \colon \calS_{g,\calO_{K}} \to \calS_{g}$ the projection morphism. Let $x_{K} \in p_{K}^{-1}(x)$ and set 
$U_{g}':=p^{-1}_{K}(U_{g})$. 

Let $A = \begin{pmatrix} 0 & 1\\ 1 & -\omega \end{pmatrix} \in \GL_{2}(\calO(U'_{g}))$. We have $A(\omega) = \infty$. If $\pr_{\Z}(x) = a_{0}$, then, for each $z\in (U'_{g})^\text{a}$, we have $A(z) \in GL_{2}(\R)$. For $i\in \{1,\dotsc,g\}$, set $M_{\infty,i} := A\, M_{i}\, A^{-1}$ in $\GL_{2}(\calO(U'_{g}))$. 
Denote by~$\Gamma_{\infty,x_{K}}$ the subgroup of~$\PGL_{2}(\calH(x_{K}))$ generated by $M_{\infty,1}(x_{K}),\dotsc,M_{\infty,g}(x_{K})$. By Corollary~\ref{cor:limittype1} and Remark~\ref{rem:SgA}, $\omega$ is not a limit point of~$\Gamma_{x_{K}}$, hence $\infty$ is not a limit point of~$\Gamma_{\infty,x_{K}}$. By Theorem~\ref{thm:Gerritzenvoisinageinfinity}, there exists an open neighborhood~$W'$ of~$x_{K}$ in~$U'_{g}$, an automorphism~$\tau$ of~$F_{g}$ and $\lambda_{1},\dotsc,\lambda_{g} \in \R_{>0}$ such that, denoting
\[(N_{\infty,1},\dotsc,N_{\infty,g}) := \tau \cdot (M_{\infty,1},\dotsc,M_{\infty,g}) \in \GL_{2}(\calO(U'_{g}))^g,\]
the family of twisted isometric discs 
\[\big(D^+_{N_{\infty,1},\lambda_{1}},\dotsc, D^+_{N_{\infty,g},\lambda_{g}}, D^+_{N_{\infty,1}^{-1},\lambda_{1}^{-1}},\dotsc,D^+_{N_{\infty,g}^{-1},\lambda_{g}^{-1}}\big)\] 
is a Schottky figure adapted to the family $(N_{\infty,1},\dotsc,N_{\infty,g})$ of $\PGL_{2}(\calO(W'))$. 
If $p(x) \ne a_{0}$, then $x$~belongs to the interior of the non-archimedean part of~$\calS_{g}$ and we may assume that $W' \subseteq \calS_{g}^\text{na}$.
  
For each $i\in\{1,\dotsc,g\}$, set $N_{i} := A^{-1} \, N_{\infty,i} \, A$. Note that we have 
\[(N_{1},\dotsc,N_{g}) := \tau \cdot (M_{1},\dotsc,M_{g}) \in \GL_{2}(\calO(U_{g}'))^g\]
and that the family 
$(A^{-1}(D^+_{N_{\infty,1},\lambda_{1}}),\dotsc, A^{-1}(D^+_{N_{\infty,g},\lambda_{g}}), A^{-1}(D^+_{N_{\infty,1}^{-1},\lambda_{1}^{-1}}),\dotsc, A^{-1}(D^+_{N_{\infty,g}^{-1},\lambda_{g}^{-1}}))$ 
is a Schottky figure adapted to the family $(N_{1},\dotsc,N_{g})$ of $\PGL_{2}(\calO(W'))$ (see Remark~\ref{rem:conjugatedSchottkyfigure}). 
Set $W := p_{K}(W')$. By \cite[Corollaire~5.6.4]{LemanissierPoineau20}, it is an open subset of~$U_{g}$.

Let us still denote by~$p_{K}$ the projection morphism $\PP^1_{\calS_{g,\calO_{K}}} \to \PP^1_{\calS_{g}}$. For $i\in\{1,\dotsc,g\}$ and $\eps\in\{-1,1\}$, set $B^+(N_{i}^\eps) := p_{K}(A^{-1}(D^+_{N_{\infty,i},\lambda_{i}^\eps}))$. It is a closed subset of $\PP^1_{W}$ and, by construction, for each $y'\in W'$, we have $B^+(N_{i}^\eps)\cap \pi^{-1}(y) = p_{K}\big(A^{-1}(D^+_{N_{\infty,i},\lambda_{i}^\eps}) \cap \pi^{-1}(y')\big)$. To prove that the family $(B^+(N_{i}^\eps), 1\le i\le g, \eps=\pm1)$ is a Schottky figure adapted to the family $(N_{1},\dotsc,N_{g})$ of $\PGL_{2}(\calO(W))$, it is enough to prove that for each $y\in W$, $i\in \{1,\dotsc,g\}$ and $\eps\in\{-1,1\}$, there exists a closed disc~$E_{y}$ in~$\pi^{-1}(y)$ such that, for each $y' \in p_{K}^{-1}(y)\cap W'$,  $A^{-1}(D^+_{N_{\infty,i},\lambda_{i}^\eps}) \cap \pi^{-1}(y') = p_{K}^{-1}(E_{y})$.



Let $y\in W$, $i\in \{1,\dotsc,g\}$ and $\eps\in\{-1,1\}$. Let $y'\in p^{-1}_{K}(y) \cap W'$. Assume that~$y$ is non-archimedean. The set $A^{-1}(D^+_{N_{\infty,i},\lambda_{i}^\eps}) \cap \pi^{-1}(y')$ is an open disc over~$\calH(y')$ that contains a fixed point of~$N_{i}(y')$. Since~$N_{i}(y')$ is defined over~$\calH(y)$, its fixed points come from $\calH(y)$-rational points by base change to~$\calH(y')$ and we deduce that $A^{-1}(D^+_{N_{\infty,i},\lambda_{i}^\eps}) \cap \pi^{-1}(y')$ is the preimage by~$p_{K}$ of a closed disc in~$\pi^{-1}(y)$. 
Assume that~$y$ is archimedean. Note that, in this case, we have $\pr_{\Z}(x)=a_{0}$. 
If $\calH(y') = \calH(y)$, then the result holds. Otherwise, we have $\calH(y') = \C$ and $\calH(y)=\R$ and the result follows from the fact that $A(y') \in \GL_{2}(\R)$. 

\medbreak

%

$\bullet$ Assume that the extension $\calH(x)/\calH(\pr_{\Z}(x))$ is infinite.

We deduce that the field~$\calH(x)$ is not locally compact, 
hence $\PP^1(\calH(x))$ is not compact. It follows from Corollary~\ref{cor:limittype1} that there exists a point $\omega \in \PP^1(\calH(x))$ that is not a limit point of~$\Gamma_{x}$. Moreover, by definition, the image~$\kappa(x)$ of $\Frac(\calO_{x})$ in~$\calH(x)$ is dense. By Corollary~\ref{cor:limittype1} again, the limit set of~$\Gamma_{x}$ is closed, hence we may assume that~$\omega$ belongs to $\PP^1(\kappa(x))$.

If $\omega=\infty$, then the result follows directly from Theorem~\ref{thm:Gerritzenvoisinageinfinity}, so let us assume that this is not the case. Then, there exists an open neighborhood~$V$ of~$x$ in~$U_{g}$ and an element~$\Omega \in \calO(V)$ whose image in~$\kappa(x)$ is~$\omega$. Let $A = \begin{pmatrix} 0 & 1\\ 1 & -\Omega \end{pmatrix} \in \GL_{2}(\calO(V))$. By construction, we have $A(x)(\omega) = \infty$ in $\PP^1(\calH(x))$. The same arguments as in the proof of the first case apply (without having to worry about a base-change), and the result follows.  
\end{proof}

We have now collected all the results necessary to prove the openness of~$\calS_{g}$.

\begin{theorem}\label{thm:open}
The Schottky space $\calS_g$ is an open subset of $\A^{3g-3, \mathrm{an}}_\Z$.
\end{theorem}
\begin{proof}
Let $x\in \calS_g$. We want to prove that there exists an open subset of $\A^{3g-3, \mathrm{an}}_\Z$ containing~$x$ that is contained in~$\calS_{g}$. This follows from Lemma~\ref{lem:Sga} when~$x$ is archimedean and from Corollary~\ref{cor:Gerritzenvoisinage} and Proposition~\ref{prop:geometrygroup} when~$x$ is non-archimedean.
\end{proof}

\subsection{The space of Schottky bases}

Let us assume $g\geq 2$ and fix a complete non-archimedean valued field $(k,\va)$.
We denote by $\mathcal{SB}_{g,k}$ the subspace of $\calS_{g,k}$ consisting of Schottky bases (see Remark~\ref{rem:SgA} and Definition~\ref{des:Schottkybasis} for these notions).

\begin{notation}
Let~$A$ be a finite subset of~$\PP^{1,\an}_{k}(k)$ with at least 2 elements. For each $\alpha\in A$, we denote by~$D^-(\alpha,A)$ the biggest open disc with center~$\alpha$ containing no other element of~$A$ and we denote by~$p_{\alpha,A}$ its boundary point in~$\PP^{1,\an}_{k}$.

Note that, if~$A$ contains at least 3 elements, then all the discs $D^-(\alpha,A)$ are disjoint.
\end{notation}

\begin{proposition}\label{prop:equationsSBg}
Let $\alpha_{1},\alpha'_{1},\dotsc,\alpha_{g},\alpha'_{g}$ be distinct elements of~$\pank(k)$. Set $A := \{\alpha_{1},\alpha'_{1},\dotsc,\alpha_{g},\alpha'_{g}\}$. Let $\beta_{1},\dotsc,\beta_{g}$ be elements of~$k$ with absolute values in~$(0,1)$. For each $i\in \{1,\dotsc,g\}$, set $\gamma_{i} := M(\alpha_{i},\alpha'_{i},\beta_{i}) \in \PGL_{2}(k)$. The following conditions are equivalent:
\begin{enumerate}
\item there exists a Schottky figure adapted to $(\gamma_{1},\dotsc,\gamma_{g})$;
\item for each $i\in \{1,\dotsc,g\}$, we have $\ell([p_{\alpha_{i},A} \, p_{\alpha'_{i},A}]) < |\beta_{i}|^{-1}$;
\item for each $i,j,k \in\{1,\dotsc,g\}$ with $j \ne i$, $k \ne i$ and $\sigma_{j},\sigma_{k} \in \{\emptyset,'\}$, we have
\[|\beta_i| \cdot |[\alpha_j^{\sigma_{j}}, \alpha_k^{\sigma_{k}}; \alpha_i, \alpha_i']| < 1.\]
\end{enumerate}
\end{proposition}
\begin{proof}
$(i) \implies (ii)$ Let $\calB = (B^+(\gamma_{i}^\eps), 1\le i\le g, \eps=\pm1)$ be a Schottky figure adapted to $(\gamma_{1},\dotsc,\gamma_{g})$. Let $i\in \{1,\dotsc,g\}$. Note that we have $\alpha_{i} \in B^+(\gamma_{i})$ and $\alpha_{i}' \in B^+(\gamma_{i}^{-1})$, hence $B^+(\gamma_{i}) \subset D^-(\alpha_{i},A)$ and $B^+(\gamma_{i}^{-1}) \subset D^-(\alpha'_{i},A)$. It follows that the segment between the boundary points of $D^-(\alpha_{i},A)$ and $D^-(\alpha'_{i},A)$ is strictly contained in the segment between the boundary points of $B^+(\gamma_{i})$ and $B^+(\gamma_{i}^{-1})$. Lemma~\ref{lem:betamodulus} then provides the desired inequality.

\medbreak

$(ii) \implies (iii)$ Let $i,j,k \in\{1,\dotsc,g\}$ with $j \ne i$, $k \ne i$ and $\sigma_{j},\sigma_{k} \in \{\emptyset,'\}$. If $[\alpha_{j}^{\sigma_{j}} \alpha_{k}^{\sigma_{k}}] \cap [\alpha_{i}\alpha'_{i}] = \emptyset$, then we have $|[\alpha_{j}^{\sigma_{j}} ,\alpha_{k}^{\sigma_{k}} ; \alpha_{i} , \alpha'_{i}]|=1$ and the inequality of the statement holds.

Assume that $I := [\alpha_{j}^{\sigma_{j}} \alpha_{k}^{\sigma_{k}}] \cap [\alpha_{i}\alpha'_{i}] \ne \emptyset$. Since $\alpha_{j}^{\sigma_{j}}$ and $\alpha_{k}^{\sigma_{k}}$ do not belong to the discs $D^-(\alpha_{i},A)$ and $D^-(\alpha'_{i},A)$, the segment~$I$ must be contained in the segment joining the boundary points of those two discs. It now follows from Lemma~\ref{lem:CrossRatio} that we have 
\[\max\big( |[\alpha_j^{\sigma_{j}}, \alpha_k^{\sigma_{k}}; \alpha_i, \alpha_i']|,  |[\alpha_j^{\sigma_{j}}, \alpha_k^{\sigma_{k}}; \alpha_i, \alpha_i']|^{-1}\big) = \ell(I) \le \ell([p_{\alpha_{i},A} \, p_{\alpha'_{i},A}]) <  |\beta_{i}|^{-1}.\]

\medbreak

$(iii) \implies (i)$ Let $i\in \{1,\dotsc,g\}$. We will construct discs $B^+(\gamma_{i})$ and $B^+(\gamma_{i}^{-1})$ that lie in~$D^-(\alpha_{i},A)$ and~$D^-(\alpha'_{i},A)$ respectively and such that $\gamma_{i}(\pank - B^+(\gamma_{i}^{-1}))$ is a maximal open disc inside $B^-(\gamma_{i})$ and $\gamma^{-1}_{i}(\pank - B^+(\gamma_{i}))$ is a maximal open disc inside $B^-(\gamma^{-1}_{i})$. To do so, we may choose coordinates on~$\pank$ such that $\alpha_{i} =0$ and $\alpha'_{i} = \infty$. The equalities of the statement then become 
\[|\beta_{i}| \cdot \frac{\alpha_{j}^{\sigma_{j}}}{\alpha_{k}^{\sigma_{k}}} < 1\] 
for  $j,k \in\{1,\dotsc,g\}$ with $j \ne i$, $k \ne i$ and $\sigma_{j},\sigma_{k} \in \{\emptyset,'\}$. It follows that there exists $r_{i} \in \R_{>0}$ such that 
\[|\beta_{i}| \, \max(|\alpha^{\sigma_{j}}_{j}|,\ j\ne i, \sigma_{j} \in \{\emptyset,'\}) < r_{i} < \min(|\alpha^{\sigma_{j}}_{j}|,\ j\ne i, \sigma_{j} \in \{\emptyset,'\}) .\]
The discs $B^+(\gamma_{i}) := D^+(0,r_{i})$ and $B^+(\gamma_{i}^{-1}) := \pank - D^-(0,|\beta_{i}|^{-1}r_{i})$ then satisfy the required conditions.

The family of discs $(B^+(\gamma_{i}^\eps), 1\le i\le g, \eps=\pm1)$ is a Schottky figure adapted to $(\gamma_{1},\dotsc,\gamma_{g})$.
\end{proof}

\begin{corollary}\label{cor:SBgkconnected}
The topological space $\mathcal{SB}_{g,k}$ is path-connected.
\end{corollary}
\begin{proof}
Let~$k'$ be a complete non-trivially valued extension of~$k$. Denote by $\pi_{k'/k} \colon \AA^{3g-3, \an}_{k'} \to \AA^{3g-3, \an}_k$ the projection map. By Remark~\ref{rem:SgA}, we have $\calS_{g,k'} = \pi_{k'/k}^{-1}(\calS_{g,k})$ and, by Proposition~\ref{prop:equationsSBg}, $\mathcal{SB}_{g,k'} = \pi_{k'/k}^{-1}(\mathcal{SB}_{g,k})$. Up to replacing~$k$ by~$k'$, we may assume that~$k$ is not trivially valued. 

We will consider the affine spaces $\AA^{2g-3, \an}_k$ with coordinates $X_{3},\dotsc,X_{g},X'_{2},\dotsc,X'_{g}$ and $\AA^{g, \an}_k$ with the coordinates $Y_{1},\dotsc,Y_{g}$. We denote by $\pi_{1} \colon \AA^{3g-3, \an}_k \to \AA^{2g-3, \an}_k$ and $\pi_{2} \colon \AA^{3g-3, \an}_k \to \AA^{g, \an}_k$ the corresponding projections.

Let $a,b \in \mathcal{SB}_{g,k}$. Let~$V$ be the open subset of~$\AA^{2g-3, \an}_k$ consisting of the points all of whose coordinates are distinct.
It is path-connected and contains $a_{1} := \pi_{1}(a)$ and $b_{1} := \pi_{1}(b)$. Let $\varphi \colon [0,1] \to V$ be a continuous map such that $\varphi(0)=a_{1}$ and $\varphi(1)=b_{1}$. The continuous maps $|[X_j^{\sigma_{j}}, X_k^{\sigma_{k}}; X_i, X_i']|$ for $i,j,k \in\{1,\dotsc,g\}$ with $j \ne i$, $k \ne i$ and $\sigma_{j},\sigma_{k} \in \{\emptyset,'\}$ are all bounded on $\varphi([0,1])$. Let $M \in \R_{>0}$ be a common upper bound. Since~$k$ is not trivially valued, there exists $\beta\in k^*$ such that
\[  |\beta|< \min\big(M^{-1}, |Y_{i}(a)|, |Y_{i}(b)|,\ 1\le i\le g\big).\]

Let us identify $\pi_{1}^{-1}(a_{1})$ and $\AA^{g,\an}_{\calH(a_{1})}$, so that~$a$ may be seen as a point in the latter space. The point $\beta := (\beta,\dotsc,\beta)$ of $\AA^{g,\an}_{k}$ canonically lifts to a point~$a_{\beta}$ of $\AA^{g,\an}_{\calH(a_{1})}$ and there exists a continuous path from~$a$ to~$a_{\beta}$ in $\AA^{g,\an}_{\calH(a_{1})}$ along which all the $|Y_{i}|$'s are non-increasing and remain in~$(0,1)$. By Proposition~\ref{prop:equationsSBg}, the corresponding path in~$\AA^{3g-3, \an}_k$ stays in~$\mathcal{SB}_{g,k}$. We similarly define a point~$b_{\beta}$ in $\pi^{-1}_{1}(b_{1})$ and a continuous path from~$b$ to~$b_{\beta}$ in~$\mathcal{SB}_{g,k}$.

To prove the result, it is now enough to construct a continuous path from~$a_{\beta}$ to~$b_{\beta}$ in $\mathcal{SB}_{g,k}$. Note that $\pi_{2}(a_{\beta}) = \pi_{2}(b_{\beta}) = \beta$, so that~$a_{\beta}$ and~$b_{\beta}$ identify to two points of the same fiber $\pi_{2}^{-1}(\beta) \simeq  \AA^{2g-3, \an}_{\calH(\beta)} =  \AA^{2g-3, \an}_k$. We may now use the path defined by~$\varphi$ to go from~$a_{\beta}$ to~$b_{\beta}$. By construction, it stays inside $\mathcal{SB}_{g,k}$.
\end{proof}

\begin{corollary}\label{cor:eqSgna}
The set $\mathcal{SB}_g^{\text{na}}$ is the subset of~$U_{g}^\text{na}$ described by the inequalities
\[|Y_i| \cdot |[X_j^{\sigma_{j}}, X_k^{\sigma_{k}}; X_i, X_i']| < 1\]
for all $i,j,k \in\{1,\dotsc,g\}$ with $j \ne i$, $k \ne i$ and $\sigma_{j},\sigma_{k} \in \{\emptyset,'\}$. It is a path-connected open subset of~$\calS_g^{\mathrm{na}}$.
\end{corollary}
\begin{proof}
The first part of the statement follows from Proposition~\ref{prop:equationsSBg}. The fact that $\mathcal{SB}_g^{\text{na}}$ is open in~$\calS_g^{\mathrm{na}}$ is an immediate consequence.

By Corollary~\ref{cor:SBgkconnected}, for each $z\in \calM(\Z)$, the fiber $\mathcal{SB}_{g} \cap \pr_{\Z}^{-1}(z)$ is connected and contains the point $P_{z}$ defined as the unique point in the Shilov boundary of the disc defined by the inequalities
\[\begin{cases}
|X_i|\le 1 \text{ for } 3\le i \le g;\\
|X'_i|\le 1 \text{ for } 2\le i \le g;\\
|Y_i|\le \frac12 \text{ for } 1\le i \le g.
\end{cases}\]
The result now follows from the continuity of the map $z\in \calM(\Z) \mapsto P_{z} \in \AA^{3g-3,\an}_{\Z}$.
\end{proof}

\section{Outer automorphisms and connectedness of $\calS_g$}
In this section, we study a natural action of the group $\Out(F_g)$ of outer automorphisms on the Schottky space $\calS_g$.
We show that this action respects the analytic structure of the Schottky space and use this to prove that $\calS_g$ is path-connected (see \ref{subsec:connect}).
Finally, we study the properness of this action. When $g=1$, $\Out(F_1)$ acts trivially on $\calS_1$, hence we assume for the rest of the section that $g\geq 2$.

\subsection{The action of $\Out(F_g)$ on the Schottky space}\label{sec:outer}
Recall from Lemma \ref{lem:kpoints} that a point of~$\calS_{g}$ corresponds to a homomorphism in $\Hom_S(F_g,\PGL_2)$.
This identification gives rise to a natural action of~$\Aut(F_{g})$ on~$\calS_{g}$ by letting an element of~$\Aut(F_{g})$ act on the source of homomorphisms in $\Hom_S(F_g,\PGL_2)$.\\
More precisely, let $\tau \in \Aut(F_{g})$, $x\in \calS_{g}$, and $\varphi_{x}$ be the associated homomorphism of $\Hom_S(F_g,\PGL_2)$.
Then, the map $\varphi_{x} \circ \tau$ is a group homomorphism from~$F_{g}$ to~$\PGL_{2}(\calH(x))$. 
Its image, being the same as that of~$\varphi_{x}$, is the Schottky group $\Gamma_x$.
Moreover, there exists a unique M\"obius transformation~$\eps$ that sends the attracting and repelling fixed points of $\varphi_{x}\circ \tau(e_{1})$ to~0 and~$\infty$ respectively and the attracting fixed point of $\varphi_{x}\circ \tau(e_{2})$ to~1. 
Then, $\eps^{-1} (\varphi_{x}\circ \tau)\, \eps$ belongs to $\Hom_{S}(F_{g},\PGL_{2}(\calH(x))$, hence gives rise to a point of~$\calS_{g}$. We denote it by~$\tau x$. 

\begin{definition}\label{def:outeraction}
The map $(\tau,x) \in \Aut(F_{g}) \times \calS_{g} \mapsto \tau x \in \calS_{g}$ defines an action of $\Aut(F_{g})$ on $\calS_{g}$ that factors through $\Out(F_{g})$. 
\end{definition}

We now describe the stabilizers of the points of~$\calS_{g}$ under the action of~$\Out(F_{g})$. 
The corresponding result for rigid Schottky spaces is known (see \cite[Satz~3]{Gerritzen81}).

\begin{lemma}\label{lem:conjugationfixedpoints}
Let $(k,\va)$ be a complete valued field and and let $\gamma$ be a loxodromic element of~$\PGL_{2}(k)$ with fixed points~$\alpha$ and~$\beta$. Let $\eps_{1},\eps_{2} \in \PGL_{2}(k)$ such that $\eps_{1}^{-1} \gamma \eps_{1} = \eps_{2}^{-1} \gamma \eps_{2}$. Then, we have $\eps_{1}^{-1}(\alpha) = \eps_{2}^{-1}(\alpha)$ and $\eps_{1}^{-1}(\beta) = \eps_{2}^{-1}(\beta)$.
\end{lemma}
\begin{proof}
We may assume that~$\alpha$ is the attracting point of~$\gamma$. Let $P \in \PP^{1,\an}_{k} - \{\eps_{1}^{-1}(\alpha),\eps_{1}^{-1}(\beta),\eps_{2}^{-1}(\alpha),\eps_{2}^{-1}(\beta)\}$. Then, for each $n\in \Z$, we have
\[\eps_{1}^{-1}(\gamma^n(\eps_{1}(P))) = \eps_{2}^{-1}(\gamma^n(\eps_{2}(P))) .\]
The left-hand side converges to~$\eps_{1}^{-1}(\alpha)$ (resp. $\eps_{1}^{-1}(\beta)$) when $n$~goes to~$+\infty$ (resp. $-\infty$). The right-hand side converges to~$\eps_{2}^{-1}(\alpha)$ (resp. $\eps_{2}^{-1}(\beta)$) when $n$~goes to~$+\infty$ (resp. $-\infty$). The result follows.
\end{proof}

\begin{proposition}\label{prop:stabilizer}
Let $x\in \calS_{g}$. The stabilizer of~$x$ under the action of $\Out(F_{g})$ is isomorphic to the quotient $\Gamma_x\backslash N(\Gamma_x)$, where $N(\Gamma_{x})$ denotes the normalizer of~$\Gamma_{x}$ in~$\PGL_{2}(\calH(x))$.

%
\end{proposition}
\begin{proof}

Let $\eps \in N(\Gamma_{x})$. The morphism~$\varphi_{x}$ induces an isomorphism $\psi_{x} \colon F_{g} \xrightarrow[]{\sim} \Gamma_{x}$. Since~$\eps$ belongs to the normalizer of~$\Gamma_{x}$ in~$\PGL_{2}(\calH(x))$, the conjugation by~$\eps$ in~$\PGL_{2}(\calH(x))$ induces an automorphism~$c_{\eps}$ of~$\Gamma_{x}$. It follows from the definitions that $\psi_{x}^{-1} \circ c_{\eps} \circ \psi_{x}$ is an automorphism of~$F_{g}$ stabilizing~$x$. We have just constructed a map $\nu \colon N(\Gamma_x) \to \Stab_{\Aut(F_{g})}(x)$. It is a morphism of groups.

Let $\eps \in N(\Gamma_{x})$. The automorphism~$\nu(\eps)$ is inner if, and only if, there exists $w\in F_{g}$ such that $\nu(\eps) = c_{w}$, where~$c_{w}$ denotes the automorphism defined by the conjugation by~$w$ in~$F_{g}$. Note that we have $c_{w} = \psi_{x}^{-1} \circ c_{\psi_{x}(w)} \circ \psi_{x}$. It follows that $\nu(\eps)$ is inner if, and only if, there exists~$\delta \in \Gamma_{x}$ such that $c_{\eps} = c_{\delta}$. If the latter condition holds, then, by Lemma~\ref{lem:conjugationfixedpoints}, $\eps^{-1}$ and $\delta^{-1}$ coincide on all the fixed points of the~$M_{i}(x)$'s. Since $g\ge 2$, there are more than two fixed points, hence $\eps=\delta$. The argument shows that~$\nu$ induces an injective morphism $\nu' \colon \Gamma_{x} \backslash N(\Gamma_x) \to \Stab_{\Out(F_{g})}(x)$. 

To conclude, it remains to prove that~$\nu'$ is surjective. It is enough to prove that~$\nu$ is surjective. Let $\sigma \in \Stab_{\Aut(F_{g})}(x)$. For each $i\in \{1,\dots,g\}$, $\sigma(e_{i})$ is an element of~$F_{g}$, that is to say a word $w_{i} = e_{j_{i,0}}^{n_{i,0}} \dotsb e_{j_{i,r_{i}}}^{n_{i,r_{i}}}$, for some $r_{i} \in \N$, $j_{i,0},\dotsc,j_{i,r_{i}} \in \{1,\dotsc,g\}$, $n_{i,0},\dotsc,n_{i,r_{i}} \in \Z$.
 Since~$\varphi_{x}$ is a morphism of groups, we have 
\[ N_{i}(x) := \varphi_{x} \circ \sigma(e_{i}) = M_{j_{i,0}}(x)^{n_{i,0}} \dotsb M_{j_{i,r_{i}}}(x)^{n_{i,r_{i}}}.\] 
In particular, $N_{i}(x) \in \Gamma_{x}$. By definition of the action, there exists $\eps \in \PGL_2(\calH(x))$ such that, for each $i\in \{1,\dotsc,g\}$, we have $M_{i}(x)=\eps N_i(x) \eps^{-1}$. By using the words associated to the morphism~$\sigma^{-1}$, one may express the~$M_{j}(x)$'s in terms of the $N_{i}(x)$'s. It follows that the $N_{i}(x)$'s generate the group~$\Gamma_{x}$, hence $\eps \in N(\Gamma_x)$. Moreover, we have $\nu(\eps) = \sigma$. The result follows.
\end{proof}

\begin{remark}\label{rem:stab<=autC}
Let $x\in \calS_{g}$. It is easy to check that each element of $N(\Gamma_x)$ preserves the limit set~$\calL_{x}$, hence inducing an automomorphism of $\pana{\calH(x)} - \calL_x$, and that the latter descends to an automorphism of $\calC_{x}$. This construction gives rise to a group homomorphism $\phi:N(\Gamma_x)\to\Aut(\calC_x)$ with kernel~$\Gamma_{x}$.

Let $x\in \calS_{g}^{\text{na}}$ be a non-archimedean point of the Schottky space. Then, the homomorphism $\phi$ is surjective (see \cite[Corollary 4.12]{Mumford72} for the first occurrence of this result and \cite[\S6.5.2]{VIASMII} for a discussion using Berkovich geometry), and as a result the quotient group $\Gamma_x\backslash N(\Gamma_x)$ considered above is isomorphic to the automorphism group $\Aut(\calC_{x})$ of the curve~$\calC_{x}$.
Every element of $\Aut(\calC_x)$ restricts to an isometry of the skeleton $\Sigma_x$ defined in \ref{not:GammaxCx}.
This restriction induces an injection of the automorphism group~$\Aut(\calC_{x})$ in the group~$\Aut(\Sigma_x)$ of isometric automorphisms of $\Sigma_x$ (cf. \cite[Proposition 6.5.9]{VIASMII}).

Let $x\in \calS_{g}^{\text{a}}$ be an archimedean point of the Schottky space. 
In this case the homomorphism $\phi$ is not surjective in general, and hence we only have an injection of $\Gamma_x\backslash N(\Gamma_x)$ into the group $\Aut(\calC_{x})$.
The question of lifting automorphisms of a Riemann surface to $\Gamma\backslash N(\Gamma)$ for some Schottky group $\Gamma$ uniformizing such surface has been thoroughly investigated by R. Hidalgo (see for example \cite{Hidalgo05}).
\end{remark}

\begin{proposition}\label{prop:outeraction}
The action of $\Out(F_g)$ on $\calS_{g}$ is analytic and has finite stabilizers.
Moreover:
\begin{enumerate}
\item If $g\geq3$, then this action is faithful;
\item If $g=2$, then the element $\iota\in \Out(F_2)$ defined by $\iota(e_i)=e_i^{-1}$ for $i=1,2$ stabilizes every point of $\calS_2$, and the action of the quotient $\Out(F_2)/\langle\iota\rangle$ on $\calS_{2}$ is faithful.
\end{enumerate}

\end{proposition}
\begin{proof}
It is a classical result of Nielsen 
that $\Out(F_g)$ is generated by the set of four elements $\{\sigma_1, \sigma_2, \sigma_3, \sigma_4\}$ defined by:

\[
\begin{cases}
\sigma_1(e_1)=e_g, \sigma_1(e_i)=e_{i-1} \; \forall i > 1 \\
\sigma_2(e_1)=e_2, \sigma_2(e_2)=e_1, \sigma_2(e_i)=e_i \; \forall i >2\\
\sigma_3(e_1)= e_1^{-1}, \sigma_3(e_i)=e_i \; \forall i > 1\\
\sigma_4(e_2)=e_1^{-1}e_2, \sigma_4(e_i)=e_i \; \forall i \neq 2
\end{cases}
\]
For $i=1,2,3$ a simple computation shows that every $\sigma_i$ acts on~$\calS_{g}$ by M\"obius transformations on the Koebe coordinates.
For example, in the case of $\sigma_1$, the point~$x$ of $\calS_g$ with Koebe coordinates $(\alpha_i, \alpha_i', \beta_i)$ is sent to a point with basis $\big(M(\alpha_g, \alpha_g', \beta_g), M(0, \infty, \beta_1), \dots, M(\alpha_{g-1}, \alpha_{g-1}', \beta_{g-1})\big)$.
To describe the action in terms of Koebe coordinates, we need to conjugate this basis by the unique M\"obius transformation $\gamma$ such that $\gamma(\alpha_g)=0$, $\gamma(\alpha_g')=\infty$, and $\gamma(0)=1$.
This conjugation sends the fixed points of a transformation to their images under $\gamma$, while leaving multipliers untouched.
Hence, the Koebe coordinates of~$\sigma_{1}(x)$ are
\[\sigma_1(\alpha_i, \alpha_i', \beta_i)=  
\big(\gamma(1),\dots,\gamma(\alpha_{g-1}),\gamma(\infty),\dots,\gamma(\alpha_{g-1}'),\beta_g,\beta_1,\dots,\beta_{g-1}\big).\]
The cases of $\sigma_2$ and $\sigma_3$ are completely analogous, and so in these three cases the action is analytic.

Let us show that this is the case for $\sigma_4$ as well.
Denote by $M:=M(\alpha^\star, {\alpha'}^\star, \beta^\star)$ 
the matrix representing the product 
$M(0, \infty, \beta_1)^{-1}M(1, \alpha'_2, \beta_2)$.
The Koebe coordinates $(\alpha^\star, {\alpha'}^\star, \beta^\star)$ are analytic functions in the coefficients of $M$ by virtue of Proposition \ref{prop:Koebe}.
Moreover, the coefficients of $M$ are rational functions without poles in the variables $\beta_1, \alpha'_2, \beta_2$ on~$\calS_{g}$, and then analytic as well.
Finally, if we want to get a basis with $1$ as attracting fixed point of the second generator we have to conjugate every element by multiplication by $\frac{1}{\alpha^\star}$.
Summarizing, we get the following expression of the action of $\sigma_4$ on $\calS_{g}$:

\[\sigma_4(\alpha_i, \alpha_i', \beta_i) = \big(\frac{\alpha_3}{\alpha^\star}, \dots, \frac{\alpha_g}{\alpha^\star}, \frac{{{\alpha'}^\star}}{\alpha^\star}, \frac{\alpha_3'}{\alpha^\star}, \dots,\frac{\alpha_g'}{\alpha^\star}, \beta_1, \beta^\star, \beta_3, \dots, \beta_g\big).\]
Since $\alpha^\star$, ${\alpha'}^\star$, and $\beta^\star$ are analytic functions of the Koebe coordinates, the action of $\sigma_4$, and hence of $\Out(F_g)$, is analytic on $\calS_{g}$.\\

The finiteness of the stabilizers follows from Proposition~\ref{prop:stabilizer} and Remark~\ref{rem:stab<=autC}.
To prove faithfulness for $g\geq 3$, it is enough to remark that for every valued field there exist Schottky uniformized curves with trivial automorphism groups.
For $g=2$, the outer automorphism $\iota$ of order~2 defined by $\iota(e_i)=e_i^{-1}$ for $i=1,2$ fixes every point $x\in \calS_2$.
In fact, writing in Koebe coordinates $x=(\alpha_2',\beta_1,\beta_2)$, the point $\iota(x)$ corresponds to the ordered basis $\big (M(\infty,0,\beta_1),M(\alpha_2',1,\beta_2)\big )$.
Then, the conjugation of this basis by the M\"obius transformation $i: z \mapsto \frac{\alpha_2'}{z}$ 
produces the ordered basis $\big (M(0,\infty,\beta_1),M(1,\alpha_2',\beta_2)\big )$, so that $\iota(x)=x$. Since a generic genus 2 curve that admits Schottky uniformization has an automorphism group of order 2, we can conclude that the action of the quotient $\Out(F_2)/\langle\iota\rangle$ on $\calS_{2}$ is faithful.
\end{proof}

\begin{remark}
The automorphism $\iota$ appearing in part (ii) of Proposition \ref{prop:outeraction} induces, via the isomorphism given in Proposition \ref{prop:stabilizer}, the element $J=\begin{bmatrix} 0 & \alpha_2' \\ 1 & 0\end{bmatrix} \in N(\Gamma_x)$.
In turn, the class $[J]$ in $\Gamma_x\backslash N(\Gamma_x) \cong \Aut(\calC_x)$ induces on $\calC_x$ the hyperelliptic involution.
If $x$ is a non-archimedean point, this follows from the fact that $[J]$ acts on $\Sigma_x$ in such a way that the quotient is a tree.
If $x$ is an archimedean point, one can argue as follows.
Writing $\gamma_1= M(0,\infty,\beta_1)$ and $\gamma_2= M(1,\alpha_2',\beta_2)$, the elements $J$, $\gamma_1 J$, and $\gamma_2 J$ restrict to the same automorphism of $\calC_x$ and give rise to 6 distinct fixed points in $\pana{k} - \calL_x$.
Up to conjugation by elements of $\Gamma_x$, one can assume that these points all lie in the same fundamental domain. 
Hence, the cover of curves $\calC_x \to \calC_x / \langle [J] \rangle$ is ramified at least at 6 points and applying Riemann-Hurwitz formula one finds that the genus of the target curve is 0.
\end{remark}

\subsection{Entr'acte: path connectedness of $\calS_g$}\label{subsec:connect}

We now apply the results of the previous section to show that $\calS_g$ is a connected topological space.

We already know that the archimedean part $\mathcal{S}_g^{a}$ of the Schottky space is path-connected, thanks to its relation to the complex Schottky space (see Lemma~\ref{lem:Sga}). This allows to use global arguments to show the connectedness of $\calS_g$.


\begin{theorem}\label{thm:pathconn}
The Schottky space $\calS_g$ is path-connected.
\end{theorem}
\begin{proof}
Let $x \in \calS_g^{\mathrm{na}}$ be a non-archimedean point of the Schottky space over $\Z$.
By Corollary~\ref{cor:Gerritzenvoisinage}, there is an automorphism $\tau \in \Out(F_g)$ such that $\tau(x) \in \mathcal{SB}_{g}^\text{na}$.


Let $\alpha_{3},\dotsc,\alpha_{g},\alpha'_{2},\dotsc,\alpha'_{g} \in \C$ such that the degree of transcendence of the extension of~$\Q$ they generate is maximal (equal to $2g-3$). Let $r_{1},\dotsc,r_{g} \in (0,1)$ whose images in the $\Q$-vector space~$\R_{>0}$ are linearly independent. For $\eps \in (0,1]$, set 
\[\sigma(a_{\infty}^{\eps}) := \rho_{\eps}(\alpha_3, \dotsc, \alpha_g, \alpha_2', \dotsc, \alpha_g', r_{1}^{1/\eps}, \dotsc, r_{g}^{1/\eps}) \in \pr_\Z^{-1}(a_{\eps})\]
(see Section~\ref{sec:BerkovichZ} for this notation). For each $a\in \calM(\Z)^\text{na}$, denote by~$\sigma(a)$ the unique point in the Shilov boundary of the disc inside $\pr_{\Z}^{-1}(a) \simeq \AA^{3g-3}_{\calH(a)}$ defined by the inequalities
\[\begin{cases}
|X_i|\le 1 \text{ for } 3\le i \le g;\\
|X'_i|\le 1 \text{ for } 2\le i \le g;\\
|Y_i|\le r_{i} \text{ for } 1\le i \le g.
\end{cases}\]
By comparing the limit of $\sigma(a_{\infty}^{\eps})$ for $\eps \to 0$ with the non-archimedean valuation $\sigma(a_0)$ over the central point (cf. Examples~\ref{ex:sectionGauss} and~\ref{ex:sectionetar}), one shows that the map
\[\sigma \colon a \in \calM(\Z) \mapsto \sigma(a) \in \AA^{3g-3, \an}_\Z\]
is a continuous section of the projection morphism $\pr_{\Z}\colon \AA^{3g-3, \an}_\Z \to \calM(\Z)$. 
By Corollary~\ref{cor:eqSgna}, $\sigma(a_{0})$ belongs to~$\mathcal{SB}_{g}$ and to the same path-connected component in~$\calS_{g}$ as~$\tau(x)$. 

Since~$\calS_{g}$ is open, by Theorem~\ref{thm:open}, we deduce that~$\sigma(a_{\infty}^{\eps})$ belongs to~$\calS_{g}$ for $\eps\in (0,1]$ small enough. In particular, $\tau(x)$ belongs to the same path-connected component of~$\calS_{g}$ as~$\mathcal{S}_g^{\mathrm{a}}$. By Proposition~\ref{prop:outeraction}, $\tau^{-1}$ acts continuously on~$\calS_{g}$, hence~$x$ belongs to the same path-connected component of~$\calS_{g}$ as~$\mathcal{S}_g^{\mathrm{a}}$. The result now follows from the path-connectedness of the latter (see Lemma~\ref{lem:Sga}).

\end{proof}

\subsection{Properness of the action of $\Out(F_g)$} Recall that we discussed proper actions at the beginning of Section~\ref{sec:limitsets}.

Let us first consider the archimedean case first. The following proposition is a consequence of well known facts about the complex Schottky space.

\begin{theorem}\label{thm:actionOutFga}
The action of $\Out(F_g)$ on $\calS_g^{\mathrm{a}}$ is proper. The quotient space $\Out(F_{g}) \backslash \calS_g^{\mathrm{a}}$ is Hausdorff and reduces locally to a quotient by a finite group.
\end{theorem}
\begin{proof}
Using the homeomorphism
\[ \Phi \colon \AA^{3g-3,\an}_{\R} \times (0,1] \to \big(\AA^{3g-3,\an}_{\Z}\big)^\text{a}\]
from Section~\ref{sec:BerkovichZ} as well as Remarks~\ref{rem:Schottkyva} and~\ref{rem:SchottkyKL} one reduces to the case of the complex Schottky space~$\calS_{g,\C}$. Recall from the discussion in Remark \ref{rem:SgA} the universal covering map 
$\Psi: \calT_{g,\C} \longrightarrow \calS_{g,\C}$ from the complex Teichm\"uller space. 
The mapping class group $\MCG_g$ acts properly on $\calT_{g,\C}$ (\cite[\S8, Theorem 6]{Gardiner87}) in such a way that $\Psi$ is a $G$-covering for a normal subgroup $G \subset \MCG_g$.
The action of $\Out(F_g)$ on $\calS_{g,\C}$ considered here comes from the realization of $\Out(F_g)$ as a subgroup of the quotient $\MCG_g / G$ (see \cite[Proposition 5.10]{HerrlichSchmithuesen07}). In particular, is is a proper action. 

The fact that the quotient space $\Out(F_{g}) \backslash \calS_g^{\mathrm{a}}$ is Hausdorff follows directly from the properness of the action. The fact that it reduces locally to a quotient by a finite group is a consequence of the finiteness of the stabilizers (see Proposition~\ref{prop:outeraction}).
\end{proof}

Let us now focus on the the non-archimedean case. We follow the strategy outlined in the proof of \cite[Proposition~7]{Gerritzen82} and take this opportunity to add details to said proof.

\begin{theorem}\label{thm:SBfinite}
The set
\[ SB=\{\tau \in \Out(F_g) : \tau(\calS\calB_g^{\mathrm{na}})\cap\calS\calB_g^{\mathrm{na}} \neq \emptyset \}\]
is finite.
\end{theorem}

\begin{proof}
For every point $x\in \calS\calB_g^{\mathrm{na}}$, let us denote by $L_x \subset \pana{\calH(x)}\big(\calH(x)\big)$ the limit set of $\Gamma_x$, and by $T_{\Gamma_x} \subset \pana{\calH(x)}$ the infinite tree defined as the skeleton of $\pana{\calH(x)} - L_x$.\footnote{Recall from \cite[4.1.3]{Berkovich90} that the skeleton of an analytic curve $C$ is defined as the subset of $C$ consisting of those points that do not have a neighborhood potentially isomorphic to a disc. If $C$ is the analytification of a smooth proper algebraic curve, its skeleton is a finite graph and this definition coincides with the one at the end of Notation \ref{not:GammaxCx}.}
The action of $\Gamma_x$ on the infinite tree $T_{\Gamma_x}$ is free and without inversions, and gives rise to a universal covering $p_x:T_{\Gamma_x}\to\Sigma_x$ of the skeleton of the Mumford curve uniformized by $\Gamma_x$.

Following Serre \cite[\S3.1]{Serre77}, we call \emph{representative tree of $T_{\Gamma_x}$} any subtree of $T_{\Gamma_x}$ that is a lifting of a spanning tree of $\Sigma_x$ via $p_x$.
Equivalently, a representative tree is a connected subtree of $T_{\Gamma_x}$ that has a unique vertex in any given $\Gamma_x$-orbit on the set of vertices of $T_{\Gamma_x}$.
With a representative tree $T \subset T_{\Gamma_x}$, we can associate a generating set of $\Gamma_x$ as follows.
Let us call $E_T$ the set of edges in~$T_{\Gamma_x}$ that have one endpoint in $T$ and the other in its complement $T_{\Gamma_x} - T$.
Note that the set $E_T$ consists of $2g$ elements.

\begin{lemma}\label{lem:Serre}
The set
\[G_T = \{ \gamma \in \Gamma_x - \{1\} : \exists \; e \in E_T \text{ with } \gamma(e) \in E_T\}\]
 is of the form $B \cup B^{-1}$ with $B=\{\gamma_1,\dots,\gamma_g\}$ a basis of $\Gamma_x$ and $B^{-1}=\{\gamma_1^{-1},\dots,\gamma_g^{-1}\}$.
\end{lemma}
\begin{proof}
Let us choose an orientation on the tree $T_{\Gamma_x}$ compatible with the action of $\Gamma_x$, and we consider the set $B\subset \Gamma_x$ consisting of the elements $\gamma\in \Gamma_x$ such that there exists an edge $e$ of $T_{\Gamma_x}$ starting in $T$ and ending in $\gamma(T)$.
By applying a theorem of Serre on free actions on oriented trees \cite[\S3.3 Th\'eor\`eme $4'$, a)]{Serre77}, we deduce that $B$ is a basis of $\Gamma_x$.
By construction, the set $B\cup B^{-1}$ is contained in $G_T$, so it suffices to show that $G_T\subset B\cup B^{-1}$ to conclude.
To show this, let us pick $\gamma \in G_T$ and $e_1, e_2\in E_T$, such that $\gamma(e_1)=e_2$.
If we call $v_i,w_i$ the endpoints of $e_i$ in such a way that $v_i \in T$, then $v_1$ and $v_2$ can not be in the same orbit, hence $\gamma(v_1)=w_2$ and $\gamma(w_1)=v_2$.
As a result, $w_1\in\gamma(T)$ and $w_2\in \gamma^{-1}(T)$.
Then, depending on the orientation chosen at the beginning, either $\gamma\in B$ or $\gamma\in B^{-1}$, as desired.
\end{proof}

We denote by $\frakF_x$ the set of representative trees of $T_{\Gamma_x}$, and by $\frakB_x$ the set of generating sets of $\Gamma_x$ of the form $B \cup B^{-1}$ as in the statement of Lemma \ref{lem:Serre}.

\begin{lemma}
The function 
\begin{align*}
G_x: &\frakF_x \to \frakB_x \\
& T \mapsto G_T
\end{align*}
is injective. Its image consists of those generating sets $B \cup B^{-1}$ such that $B$ is a Schottky basis.
\end{lemma}

\begin{proof}
Let $B$ be a Schottky basis, choose a Schottky figure adapted to $B$, and consider the associated analytic space $F^{+} \subset \pana{\calH(x)}$ as in Definition \ref{def:Schottkyfigure}.
We claim that the maximal subtree $T_B$ of $T_{\Gamma_x}$ contained in $F^{+}$ has a unique vertex in any $\Gamma_x$-orbit on the set of vertices of $T_{\Gamma_x}$, and therefore is a representative tree.
To prove this, recall from \ref{not:GammaxCx} that the skeleton $\Sigma_x$ is obtained by pairwise identifying the endpoints of the intersection $F^{+} \cap T_{\Gamma_x}$ according to the action of $\Gamma_x$.
In particular, $F^{+}$ contains a fundamental domain for the action of $\Gamma_x$ on $\pana{\calH(x)} - L_x$, so that there is a vertex of $T_B$ in the orbit of $v$, for every vertex $v$ of $T_{\Gamma_x}$.
Moreover, if $\xi$ is an endpoint of $F^{+} \cap T_{\Gamma_x}$, then the map $p_x$ identifies $\xi$ with only another endpoint of $F^{+} \cap T_{\Gamma_x}$, so that $p_x(\xi)$ is a point of degree 2 of $\Sigma_x$ and hence by definition it is not a vertex of $\Sigma_x$.
In particular, $\xi$ is not a vertex of $T_B$, so that all vertices of $T_B$ are contained in a fundamental domain for the action of $\Gamma_x$ on $\pana{\calH(x)} - L_x$.
This shows that every element in the orbit of a vertex $v$ of $T_B$ different from itself does not lie in $T_B$, concluding the proof that $T_B$ is a representative tree.

Note that the endpoints of $F^{+} \cap T_{\Gamma_x}$ lie precisely on those edges of $T_{\Gamma_x}$ that are in $E_{T_B}$, and by construction of $F^{+}$ the elements of $\Gamma_x$ acting on these endpoints are precisely those lying in $B \cup B^{-1}$.
Hence we have $G_{T_B}=B\cup B^{-1}$, showing that $B\cup B^{-1}$ is in the image of $G_x$.\\
Conversely, if $B=\{\gamma_1,\dots,\gamma_g\}$ is a basis of $\Gamma_x$ and $B\cup B^{-1}\in \frakB_x$ can be written as $G_x(T)$ for some representative tree $T$, one can build a Schottky figure adapted to $B$ as follows.
First one writes the set $E_T$ as $\{e_{-g},\dots, e_{-1}, e_1,\dots, e_g\}$ in such a way that $\gamma_i(e_{-i})=e_i$.
Then one chooses a set of $2g$ points $\{x_{-g},\dots, x_{-1}, x_1,\dots, x_g\}$ in $\pana{\calH(x)}$ in such a way that $x_{-i} \in e_{-i}$, $x_i \in e_i$ and $\gamma_i(x_{-i})=x_i$ for every $i=1,\dots,g$ (here we tacitly identify an element of $E_T$ with the corresponding subset of $\pana{\calH(x)}$).
Each $x_i \in \pana{\calH(x)}$ is the Shilov boundary of a unique closed disc of $\pana{\calH(x)}$ not containing $T$, that we denote by $B^+_i$.
The family $\calB = \big\{B^+_i ,i \in \{-g,\dots, -1,1,\dots,g \} \big\}$ is a Schottky figure adapted to $B$: In fact, if we denote by $B^-_i$ the unique connected component of $\pana{\calH(x)} - \{x_i\}$ such that $T\cap B^-_i = \emptyset$ and $L_x\cap B^-_i \neq \emptyset$, we have that $B^-_i$ is a maximal open disc inside $B^+_i$ and that \[B^-_i = \gamma_{i} (\pana{\calH(x)} - B_{-i}^+),\] where we adopted the convention $\gamma_{-i}=\gamma_i^{-1}$.
Then $B$ is a Schottky basis, constructed in such a way that the representative tree $T_B$ as above coincides with $T$.

The injectivity of $G_x$ is proved as follows: let $G_x(T_1)=G_x(T_2)=B \cup B^{-1}$.
By what we just proved, $B$ is a Schottky basis and the injectivity of $G_x$ is equivalent to the fact that every Schottky figure associated with $B$ gives rise to the same representative tree $T_B$.
We can show this by contradiction: suppose that two Schottky figures $\calB_1$ and $\calB_2$ adapted to $B$ give rise to different representative trees $T_1$ and $T_2$.
Then there is a vertex $v$ of $T_1$ that is not a vertex of $T_2$, and there are at least two edges $e_i, e_j$ of $E_{T_1}$ having $v$ as an endpoint that are not edges of $T_2$ nor they belong to $E_{T_2}$, for instance those edges departing from $v$ in a direction different from the one of $T_2$.
As a result, there is a unique connected component of $T_{\Gamma_x} - T_2$ that contains both $e_i$ and $e_j$.
Hence, the two closed discs $B_i^+, B_j^+ \in \calB_1$ corresponding to $e_i, e_j$ are contained in a single closed disc $B^+ \in \calB_2$.
This leads to a contradiction, since every disc in a Schottky figure adapted to $B$ contains a unique fixed point of a unique element of $B$, but $B^+$ contains at least two of them.
This shows that $T_1=T_2$.
\end{proof}

Thanks to the injectivity of $G_x$, one can associate with $x$ a unique representative tree $T_x\in\frakF_x$, defined as the preimage by $G_x$ of the generating set $\{M_1(x),\dots, M_g(x),M_1^{-1}(x),\dots, M_g^{-1}(x)\}$.
Furthermore, with the point $x$ one can also associate the set 
\[SB_x=\{\tau \in \Out(F_g) : \tau(x)\in\calS\calB_g \}.\]
Our proof of the theorem then relies on the following lemmas.

\begin{lemma}\label{lem:sbx}
Let $x \in \calS\calB_g^{\mathrm{na}}$.
Then $SB_x$ is a finite subset of $\Out(F_g)$.
\end{lemma}

\begin{proof}
Let us fix a vertex $v\in T_x$, and call $\frakF_{x,v}$ the subset of $\frakF_x$ consisting of those representative trees that contain $v$.
We first prove that every Schottky basis is $\Gamma_x$-conjugated to a unique Schottky basis in the image $G_x(\frakF_{x,v})$.
In fact, if $B$ is a Schottky basis, then $B \cup B^{-1}=G_x(T)$ for some $T\in \frakF_x$.
For every $\gamma\in \Gamma_x$, $\gamma G_x(T) \gamma^{-1}$ is the image by~$G_{x}$ of the representative tree $\gamma(T)$.
Since there exists a unique $\gamma\in \Gamma_{x}$ such that $v\in \gamma(T)$, there is a unique generating set in $G_x(\frakF_{x,v})$ conjugated to $B \cup B^{-1}$ by an element of $\Gamma_x$.
As a result, the function $G_x$ realizes a bijection between the set $\frakF_{x,v}$ and the set of $\Gamma_x$-conjugacy classes of generating sets of the form $B\cup B^{-1}$ with $B$ a Schottky basis of $\Gamma_x$.
As the former is a finite set, the latter is also finite and, in particular, the $\PGL_2$-conjugacy classes of Schottky bases of $\Gamma_x$ are finite.

For every $y\in \calS\calB_g^{\mathrm{na}}$, let us define the subset $SB_{x,y}=\{\tau \in \Out(F_g) : \tau(x)=y\}$ of $SB_x$.
If $\tau_1, \tau_2 \in SB_{x,y}$, then $\tau_1\tau_2^{-1}(x)=x$, so there is an element $\sigma\in \Out(F_g)$ in the stabilizer of $x$ such that $\tau_1=\sigma \tau_2$.
By Proposition \ref{prop:stabilizer}, this stabilizer is finite, so the set $SB_{x,y}$ is finite too.
We can write the set $SB_x$ as a union 
\[SB_x = \bigcup_{y\in \calS\calB_g^{\mathrm{na}}} SB_{x,y}.\]
Note that the set $SB_{x,y}$ is non-empty only if the Schottky basis $\big(M_1(y),\dots, M_g(y)\big)$ is $PGL_2$-conjugated to $\tau\big(M_1(x),\dots, M_g(x)\big)$.
Hence, by what precedes, $SB_x$ is a finite union of finite sets and then it is finite.

\end{proof} 

Given a tree $\calT$, recall that a vertex of degree one of $\calT$ is called a  \emph{leaf}.
If the set of leaves $L(\calT)$ of $\calT$ has cardinality $2g$, we call \emph{leaf labeling} of $\calT$ a bijection between $L(\calT)$ and the set $\{-g,\dots, -1,1,\dots, g\}$.
Let us denote by $\mathfrak{L}_g$ the set of pairs $(\calT, \ell)$ where $\calT$ is a finite tree with $2g$ leaves and no vertices of degree 2, and $\ell$ is a leaf labeling of $\calT$.
Since $g$ is fixed, $\mathfrak{L}_g$ is a finite set.
For a point $x~\in~\calS\calB_g^{\mathrm{na}}$, the subtree $T_x \cup E_{T_x}$ of $T_{\Gamma_x}$ has $2g$ leaves and no vertices of degree 2, and can naturally be endowed with a labeling induced by the writing $E_{T_x}=\{e_{-g},\dots, e_{-1}, e_1,\dots, e_g\}$ as in the first part of the proof.
This assignment defines a map $\lambda:\calS\calB_g^{\mathrm{na}} \to \mathfrak{L}_g$.

\begin{lemma}\label{lem:sbx=sby}
Let $x,y$ be two points of $\calS\calB_g^{\mathrm{na}}$ such that $\lambda(x)=\lambda(y)$. Then $SB_x=SB_y$. 
\end{lemma}
\begin{proof}
Since $\lambda(x)=\lambda(y)$, there exists an isomorphism of finite trees
\[\psi: T_x\cup E_{T_x} \to T_y \cup E_{T_y} \] that sends $T_x$ to $T_y$ and respects the leaf labelings.
Consider the group isomorphism $\phi:\Gamma_x \to \Gamma_y$ such that $\phi(M_i(x))=M_i(y)$.
Then there is a unique way to extend $\phi$-equivariantly the isomorphism $\psi$ to an isomorphism of infinite trees $\Psi:T_{\Gamma_x} \to T_{\Gamma_y}$.
Namely, for every vertex $v'\in T_{\Gamma_x}$, there is a unique pair $(\gamma,v)$ with  $\gamma \in \Gamma_x$ and $v$ a vertex of $T_x$ such that $v'=\gamma(v)$.
The assignment $\Psi(v')=\phi(\gamma)(\Psi(v))$ uniquely determines the isomorphism $\Psi$.

Let us fix vertices $v_x\in T_x$ and $v_y\in T_y$ such that $\Psi(v_x)=v_y$.
Note that we constructed $\Psi$ in such a way to be equivariant, so there is a commutative diagram of the form
\[\begin{tikzcd}
T_{\Gamma_x} \arrow[r, "\Psi"] \arrow[d,"p_x"] & T_{\Gamma_y}\arrow[d,"p_y"] \\
\Sigma_x \arrow[r, "\sim"] & \Sigma_y
\end{tikzcd},\]
where the arrow on the bottom is an isomorphism of graphs, and in particular sends spanning trees in $\Sigma_x$ to spanning trees in $\Sigma_y$.
As a result, $\Psi$ sends representative trees in $T_{\Gamma_x}$ to representative trees in $T_{\Gamma_y}$.
In particular, $\Psi$ restricts to a function 
\[\Psi_{|\frakF_{x,v_x}}:\frakF_{x,v_x}\to\frakF_{y,v_y}\]
satisfying $\Psi_{|\frakF_{x,v_x}}(T_x)=T_y$.

Now we consider an element $\tau\in SB_x$.
This has a unique representative $\sigma \in \Aut(F_g)$ such that $B_\sigma$, the Schottky basis resulting from applying $\sigma$ to $(M_1(x),\dots,M_g(x))$, satisfies $G_x^{-1}(B_\sigma)\in\frakF_{x,v_x}$.
If we consider the representative tree $T = \Psi_{|\frakF_{x,v_x}} \big( G_x^{-1}(B_\sigma) \big)$, we have that the generating set $G_y(T)$ of $\Gamma_y$ is of the form $B \cup B^{-1}$ for some Schottky basis $B$ of $\Gamma_y$.
By definition, $G_y(T)$ consists of those elements of $\Gamma_y$ that act on the set $E_T$.
Note that for every vertex $v\in T_{\Gamma_x}$, we have $\Psi\big(\big(M_i(x)\big)(v)\big)=\big(M_i(y)\big)(\Psi(v))$ thanks to the fact that $\Psi$ is $\phi$-equivariant.
This condition ensures that, if we set $B_\sigma = \big(M_{j_{1,0}}^{n_{1,0}}(x) \dotsb M_{j_{1,r_{1}}}^{n_{1,r_{1}}}(x), \dots, M_{j_{g,0}}^{n_{g,0}}(x) \dotsb M_{j_{g,r_{g}}}^{n_{g,r_{g}}}(x) \big)$, then $B$ can be taken to be $\big(M_{j_{1,0}}^{n_{1,0}}(y) \dotsb M_{j_{1,r_{1}}}^{n_{1,r_{1}}}(y), \dots, M_{j_{g,0}}^{n_{g,0}}(y) \dotsb M_{j_{g,r_{g}}}^{n_{g,r_{g}}}(y) \big)$, that is, $\tau(y)\in \calS\calB_g^{\mathrm{na}}$.
Hence $\tau$ is in $SB_y$ and so $SB_x \subset SB_y$.
The same construction applied to the isomorphism $\Psi^{-1}$ shows that $\tau\in SB_y$ implies $\tau\in SB_x$, so that $SB_x = SB_y$.
\end{proof}

The theorem then follows from the two lemmas above.
We write
\[SB=\bigcup_{x\in \calS\calB_g^{\mathrm{na}}} SB_x = \bigcup_{\lambda(x) \in \mathfrak{L}_g} SB_x,\]
where the second equality is given by Lemma \ref{lem:sbx=sby}.
The result of Lemma \ref{lem:sbx} ensures the finiteness of $SB_x$ for every $x$, and the finiteness of the set $\mathfrak{L}_g$ allows to conclude.
\end{proof}


\begin{corollary}\label{cor:actionOutFgna}
The action of $\Out(F_g)$ on $\calS_g^{\mathrm{na}}$ is proper. The quotient space $\Out(F_{g}) \backslash \calS_g^{\mathrm{na}}$ is Hausdorff and reduces locally to a quotient by a finite group.
\end{corollary}
\begin{proof}
Let $x,y \in \calS_{g}^{\mathrm{na}}$. By Corollary~\ref{cor:Gerritzenvoisinage}, there are $\sigma_{x}, \sigma_{y} \in \Out(F_g)$ such that $\sigma_{x}(x), \sigma_{y}(y) \in \mathcal{SB}_{g}^\mathrm{na}$. By Corollary~\ref{cor:eqSgna}, $\mathcal{SB}_{g}^\mathrm{na}$ is open in~$\calS_{g}^{\mathrm{na}}$. It follows that $U_{x} := \sigma_{x}^{-1}(\mathcal{SB}_{g}^\mathrm{na})$ and $U_{y} := \sigma_{y}^{-1}(\mathcal{SB}_{g}^\mathrm{na})$ are open neighborhoods of~$x$ and~$y$ respectively. By Theorem~\ref{thm:SBfinite}, the set 
\[ \{\tau \in \Out(F_g) : \tau(U_{x})\cap U_{y} \neq \emptyset \} = \{\tau \in \Out(F_g) : \sigma_{y}\tau \sigma_{x}^{-1}(\calS\calB_g^{\mathrm{na}})\cap\calS\calB_g^{\mathrm{na}} \neq \emptyset \} \] 
is finite. It follows that the action is proper, and that $\Out(F_{g}) \backslash \calS_g^{\mathrm{na}}$ is Hausdorff.

Let us now prove the last part of the statement. Let $x \in \calS_{g}^{\mathrm{na}}$. The previous result applied with $y=x$ ensures that there exists an open neighborhood~$U_{x}$ of~$x$ such that 
\[ T := \{\tau \in \Out(F_g) : \tau(U_{x})\cap U_{x} \neq \emptyset \}\]
is finite. Up to shrinking~$U_{x}$, we may assume that $T = \Stab(x)$ and that $U_{x}$ is stable under~$\Stab(x)$. We then have a canonical isomorphism $\Out(F_{g}) \backslash \calS_g^{\mathrm{na}} \simeq \Stab(x) \backslash \calS_g^{\mathrm{na}}$. The result follows.

\end{proof}

\begin{corollary}\label{cor:quotientOutFgfiber}
For each $a\in \calM(\Z)$, the quotient space $\Out(F_{g}) \backslash(\calS_{g}\cap\pr_{\Z}^{-1}(a))$ inherits a structure of $\calH(a)$-analytic space. 
\end{corollary}
\begin{proof}
Let $(k,\va)$ be a complete valued field. Recall that, if $X$ is a $k$-analytic space and $G$ a finite group acting on~$X$, then the quotient $G\backslash X$ inherits a structure of $k$-analytic space. The archimedean case reduces to the case of complex analytic spaces, which is handled in~\cite[Th\'eor\`eme~4]{CartanQuotient}. The non-archimedean case is a consequence of~\cite[Proposition~6.3.3/3]{BGR}.

With these results at hand, the statement follows from Theorem~\ref{thm:actionOutFga} and Corollary~\ref{cor:actionOutFgna}.
\end{proof}

\begin{remark}\label{rem:quotientOutFgglobal}
Even though the results of this section mostly apply to the local situation, we believe that they should hold globally, which is to say the action of $\Out(F_g)$ on $\calS_g$ is proper and the quotient space $\Out(F_{g}) \backslash \calS_{g}$ inherits a structure of analytic space over~$\Z$. 
\end{remark}

\section{Schottky uniformization for families of curves}\label{sec:modular}

\subsection{The universal Mumford curve over $\Z$}

\begin{definition}
We call \emph{universal Schottky group} the following subgroup of $\PGL_{2}(\calO(\calS_{g}))$:
\[\calG_{g} := \langle M(0,\infty,Y_{1}), M(1,X'_{2},Y_{2}),\dotsc, M(X_{g},X'_{g},Y_{g})\rangle.\]

For $x\in \calS_{g}$, we denote by $L_{x} \subseteq \pi^{-1}(x)\simeq \PP^{1,\an}_{\calH(x)}$ the limit set of~$\Gamma_{x}$. We call \emph{limit set} of~$\calG_{g}$ the set
\[\calL_{g} := \bigcup_{x\in\calS_{g}} L_{x} \subseteq \PP^1_{\calS_{g}}.\]
We set
\[\Omega_{g} := \PP^1_{\calS_{g}} - \calL_{g}.\]
\end{definition}

\begin{theorem}\label{thm:universalaction}
The limit set~$\calL_{g}$ of~$\calG_{g}$ is a closed subset of $\PP^1_{\calS_{g}}$, the action of~$\calG_{g}$ on its complement~$\Omega_{g}$ is free and proper and the quotient map $\calG_{g}\backslash\Omega_{g}\to \calS_{g}$ is proper.
\end{theorem}
\begin{proof}
It is enough to prove that every $x\in \calS_{g}$ admits a neighborhood~$U_{x}$  such that $\calL_{g} \cap \pi^{-1}(U_{x})$ is a closed subset of~$\PP^1_{U_{x}}$, the action of~$\calG_{g}$ on $\Omega_{g} \cap \pi^{-1}(U_{x})$ is free and proper and the quotient map $\calG_{g}\backslash(\Omega_{g}\cap \pi^{-1}(U_{x}))\to U_{x}$ is proper. Let $x\in \calS_{g}$. 

\medbreak

Assume that $x$ is archimedean. Arguing as in Remark~\ref{rem:SgA} and the proof of Lemma~\ref{lem:Sga}, we may replace~$\calS_{g}$ by the complex Schottky space $\calS_{g,\C}$. We will add subscripts~$\C$ to denote the various objects under consideration in this setting. 

By \cite[Proposition~3]{Bers75}, $\calL_{g,\C}$ is closed. The main ingredient in the proof is the existence of an open neighborhood $U_{0}$ of~0 in~$\C^{3g-3}$, a neighborhood~$U$ of~$x$ in~$\calS_{g,\C}$, an analytic isomorphism $\xi\colon U_{0}\to U$, and Beltrami coefficients $\mu_{1},\dotsc,\mu_{3g-3}\colon \PP^1(\C)\to \C$ such that, for each $s\in U_{0}$, we have
\[M_{i}(\xi(s)) = w^{s\mu}\, M_{i}(x) \,(w^{s\mu})^{-1} \textrm{ for } 1\le i\le 3g-3,\]
where $w^{s\mu}$ denotes the unique quasiconformal automorphism of~$\PP^1(\C)$ with Beltrami coefficient $s\mu = \sum_{i=1}^{3g-3} s_{i}\mu_{i}$ that fixes~0, 1 and~$\infty$. One may then relate the domain of discontinuity~$\Omega_{g,x}$ over~$x$ to that over~$\xi(s)$, for some  
$s\in U_{0}$, by $\Omega_{g,\xi(s)} = w^{s\mu}(\Omega_{g,x})$. It now follows from the continuity properties of~$w^{s\mu}$ (see \cite[Theorem~8]{AhlforsBers60}) that~$\Omega_{g}$ is open, or, equivalentely, that~$\calL_{g}$ is closed.

Similarly, let $F^+_{x}$ be a compact subset of $\Omega_{g,x}$ intersecting every orbit of the action of~$\calG_{g,x}$. Let~$V_{0}$ be a compact neighborhood of~0 in~$\C^{3g-3}$. Then, the set 
\[ F^+ := \bigcup_{s\in V_{0}} w^{s\mu}(V_{0})\]
is a compact subset of~$\Omega_{g}$ that intersects every orbit of the action of~$\calG_{g,\xi(s)}$ for $s\in V_{0}$. It follows that $\calG_{g}\backslash (\Omega_{g}\cap \pi^{-1}(V_{0}))$ is compact. 

Finally, the group $\calG_{g,\C}$ acts on~$\Omega_{g,\C}$ by preserving the fibers of the morphism $\Omega_{g,\C}\to \calS_{g,\C}$. Since its action is free and proper on each fiber, it is free and proper on the whole~$\Omega_{g,\C}$. 


\medbreak

Assume that $x$ is non-archimedean. The result then follows from Corollary~\ref{cor:Gerritzenvoisinage} and Proposition~\ref{prop:actionproper}.
\end{proof}

It follows from Theorem \ref{thm:universalaction} that the quotient
\[\calC_{g} := \Omega_{g}/\calG_{g}\]
makes sense as an analytic space over~$\Z$. We call it the \emph{universal Mumford curve} over~$\Z$.
This definition is motivated by the following corollary, summarizing the results of this section.
\begin{corollary}\label{cor:universalunif}
There is a commutative diagram in the category of analytic spaces over~$\Z$:
\[\begin{tikzcd}
\Omega_{g} \arrow[dd, "\pi"] \arrow[rd]& \\
& \calC_{g} \arrow[dl]\\
\calS_{g}&
\end{tikzcd}\]
where the quotient map $\Omega_{g} \to \calC_{g}$ is a local isomorphism and the natural map $\calC_{g}\to \calS_{g}$ is proper and smooth of relative dimension~1.\footnote{The property ``smooth of relative dimension~1'' has not yet been defined in the setting of Berkovich spaces over arbitrary Banach rings. However, it certainly holds here in any reasonable sense, since the fibration~$\psi$ is locally isomorphic to the relative line.}
Given a point $x\in \calS_{g}$, its preimage in $\calC_{g}$ is a curve over $\calH(x)$ which is isomorphic to the curve uniformized by the Schottky group $\Gamma_x$.
\end{corollary}

\subsection{Moduli spaces of Mumford curves}
The existence of the universal Mumford curve $\calC_g$ raises the question of the existence of a moduli space of Mumford curves and its connections with the moduli space of stable curves.
Over a non-archimedean fiber this space is obtained as the quotient of $\calS_g$ by the action of $\Out(F_g)$ described in section \ref{sec:outer}, which was previously known in the rigid analytic context from work of L. Gerritzen \cite{Gerritzen81} and F. Herrlich \cite{Herrlich84}, among others.
\begin{remark}
Over an archimedean fiber, the quotient of $\calS_g$ by the action of $\Out(F_g)$ captures more than isomorphism classes of Riemann surfaces, as there are different Schottky groups that uniformize the same complex curve.
In the case $g=1$, one gets a trivial action, and the space $\calS_{1,\C}$ is the pointed open unit disc over $\C$, image of the Poincar\'e open-half plane under the map $z \mapsto e^{2i\pi z}$.
As such, $\calS_{1,\C}$ is a cover of the modular curves $X_0(N)$ and $X_1(N)$ for every $N$. 
Hence a point of $\calS_{1,\C}$ not only determines an isomorphy class of an elliptic curve, but also takes into account its $N$-level structures for every $N$.
It would be interesting to investigate whether the space $\Out(F_g) \backslash \calS_g$ over the complex numbers can bear a similar interpretation for curves of higher genus.
\end{remark}

In what follows, we consider only the non-archimedean case: let $a\in\calM(\Z)$ be a non-archimedean point, and consider the fiber $\calS_{g,a}=\calS_g \cap \pr^{-1}_\Z(a)$ over $a$ of the Schottky space.
We denote by $\Mumf_{g,a}$ the quotient of $\calS_{g,a}$ by the continuous action of $\Out(F_g)$ defined in \ref{def:outeraction}.
This quotient is a $\calH(a)$-analytic space by Corollary \ref{cor:quotientOutFgfiber}. Since any element of $\Out(F_g)$ acts on the marking but does not affect the conjugacy class of a Schottky group, each point $x \in \Mumf_{g,a}$ corresponds to a unique conjugacy class of a Schottky group $\Gamma_x \subset PGL_2(\calH(x))$.
Moreover, by \cite[Corollary (4.11)]{Mumford72}, the isomorphism class of a Mumford curve over a non-archimedean field determines the conjugacy class in $\PGL_2$ of its Schottky group.
This shows that, for every valued extension $k$ of $\calH(a)$ the $k$-points of $\Mumf_{g,a} \times_{\calH(a)} k$ are in 1-to-1 correspondence with isomorphism classes of Mumford curves of genus $g$ defined over $k$.

\subsubsection{Relationship with geometric group theory and tropical moduli}\label{subsec:tropical}

The existence of a faithful action of $\Out(F_g)$ on $\calS_g$ with finite stabilizers is reminescent of Culler-Vogtmann definition of the \emph{outer space} in the context of geometric group theory, as introduced in their seminal paper \cite{CullerVogtmann86}.
This is not a coincidence, and in this section we show that we can indeed relate the topology of $\calS_g$ with that of the outer space.

Let us recall the definition of the Culler-Vogtmann outer space.
We fix $g\geq 2$ and an abstract graph $R_g$ with one vertex and $g$ edges, identifying its fundamental group $\pi_1(R_g)$ with $F_g$.
A finite connected graph $G$ is said to be \emph{stable} if all its vertices have degree $\geq$ 3.
A \emph{marking} on a stable graph $G$ of Betti number $g$ is a homotopy equivalence $m:R_g \to G$ or, equivalently, the choice of a group isomorphism between $F_g$ and the fundamental group $\pi_1(G)$.
Two pairs $(G,m)$ and $(G', m')$ each consisting of a stable metric graph and a marking are \emph{equivalent} if there is an isometry $s:G \to G'$ such that $s \circ m$ is homotopic to $m'$.
For a given marked graph $(G,m)$, the isomorphism $F_g \cong \pi_1(G)$ determines an action of $F_g$ on the universal cover $T$ of $G$, a tree naturally endowed with a metric, denoted by $d_T$.
The \emph{translation length function} of $(G,m)$ is the function $\ell_G:F_g \to \R$
associating to any $\sigma \in F_g$ the quantity $\ell_G(\sigma):=\min_{x\in T}\{ d_T(\sigma(x),x)\}$.
Let $CV_g$ denote the set of equivalence classes of stable marked graphs endowed with a metric such that the sum of edge lengths is unitary, and let $\calC$ denote the set of conjugacy classes in $F_g$.
The rule associating with a marked tree its translation length function defines an embedding $CV_g \hookrightarrow \R^{\calC}$ into the infinite dimensional real vector space $\R^{\calC}$.
Thanks to this fact, $CV_g$ inherits a topology from the product topology on $\R^{\calC}$. 
The topological space so obtained is called the \emph{Culler-Vogtmann outer space}, and it is also denoted by $CV_g$.

The original definition of the outer space can be found in \cite[\S0]{CullerVogtmann86}, where more details about the length functions and the topology of the outer space are given.
In what follows, it will be useful to drop the condition that the marked graphs have unitary sum of edge lengths. 
We will then denote by $CV_g'$ the \emph{unprojectivized outer space} $CV_g\times \R_{>0}$, which parametrizes marked graphs with arbitrary edge lengths.
There is a natural continuous action of $\Out(F_g)$ on $CV_g$, which extends to $CV_g'$ using the trivial action on the factor~$\R_{>0}$.
The quotient space $CV_g'/\Out(F_g)$ has a canonical injection in the moduli space of abstract weighted tropical curves $M_g^{\mathrm{trop}}$, whose image is given by those tropical curves that have weight zero at every vertex.
The induced map $CV_g'\to M_g^{\mathrm{trop}}$ is continuous and corresponds to forgetting the marking on a given metric graph.
For more details about equivalent definitions of~$M_g^{\mathrm{trop}}$ and its properties, we refer to \cite[\S 3]{BrannettiMeloEtAl11}, while a comparison between $CV_g$ and $M_g^{\mathrm{trop}}$ is discussed in \cite[\S 5.2]{Caporaso13}.

An isomorphism of Mumford curves induces an isometry between their skeletons.
This allows to define a continuous function $\Mumf_{g,a}\to M_g^{\mathrm{trop}}$ sending (the class of) a Mumford curve in (the class of) its skeleton.
For a point $x\in \Mumf_{g,a}$, recall that the Schottky uniformization $\big(\pana{\calH(x)} - \calL_x \big) \to \calC_x$ restricts to a universal cover of the skeleton $\Sigma_x$ (see \cite[Theorem 6.4.18]{VIASMII}).
Via this restriction, the Schottky group $\Gamma_x\cong \pi_1(\calC_x)$ can be identified with the topological fundamental group $\pi_1(\Sigma_x)$.

\begin{theorem}\label{thm:Schottky-outer-space}
There is a continuous surjective function
\[\phi:\calS_{g,a}\longrightarrow CV_g' \times_{M_g^{\mathrm{trop}}} \Mumf_{g,a}.\]
\end{theorem}
\begin{proof}
Let us consider the following:
\begin{itemize}
\item The continuous function $\phi_1: \calS_{g,a} \to \Mumf_{g,a}$ given by the quotient by the action of $\Out(F_g)$. Note that continuity descends from Proposition \ref{prop:outeraction}; 
\item The continuous function $\phi_2:\calS_{g,a} \to CV_{g}'$ given by assigning to each $y\in \calS_{g,a}$ the metric graph corresponding to the skeleton~$\Sigma_{y}$ of the Mumford curve $\calC_y$ and the marking as follows: recall from Lemma \ref{lem:kpoints} that the point $y$ can be identified with the conjugacy class of a morphism $\varphi_y:F_g \hookrightarrow \mathrm{PGL}_2(\calH(y))$, whose image is the fundamental group $\pi_1(\calC_y)$, and associate with $y$ the marking corresponding to the isomorphism $F_g \cong \Gamma_y$ induced by $\varphi_y$.
To prove continuity for $\phi_2$, we prove that the composite function $\calS_{g,a} \to \R^{\calC}$ is continuous.
This amounts to prove that the following:
if $\sigma\in \Aut(F_g)$ is defined by $\sigma(e_i) = e_{j_{i,0}}^{n_{i,0}} \dotsb e_{j_{i,r_{i}}}^{n_{i,r_{i}}}$, for some $r_{i} \in \N$, $j_{i,0},\dotsc,j_{i,r_{i}} \in \{1,\dotsc,g\}$, $n_{i,0},\dotsc,n_{i,r_{i}} \in \Z$, the assignment
\[ y \mapsto \ell_{\Sigma_{y}}(M_{j_{i,0}}(y)^{n_{i,0}} \dotsb M_{j_{i,r_{i}}}(y)^{n_{i,r_{i}}})\]
defines a continuous function $L:\calS_{g,a} \to \R$.
By Lemma~\ref{lem:betamodulus}, the length $\ell_{\Sigma_{y}}(M)$ for $M\in \Gamma_y$ coincides with $|\beta|^{-1}$, where $\beta$ is the multiplier of~$M$.
The result then follows from Proposition~$\ref{prop:Koebe}$, that ensures that the multiplier of the element $M_{j_{i,0}}(y)^{n_{i,0}} \dotsb M_{j_{i,r_{i}}}(y)^{n_{i,r_{i}}}$ is a continuous function in the Koebe coordinates of $y$.

\end{itemize}
The function $\phi_2$ is $\Out(F_g)$-equivariant, and then agrees with $\phi_1$ on $M_g^{\mathrm{trop}}$. 
By the universal property of the fiber product, the pair $(\phi_1,\phi_2)$ defines a continuous function 
\[\phi:\calS_{g,a}\longrightarrow CV'_g \times_{M_g^{\mathrm{trop}}} \Mumf_{g,a}.\]

We now prove that the function $\phi$ is surjective.
Let $\big([G,m],[C]\big) \in CV'_g \times_{M_g^{\mathrm{trop}}} \Mumf_{g,a}$ be a pair consisting of an equivalence class of a marked graph and an isomorphism class of a Mumford curve of genus $g$, such that the graph $G$ is isometric to the skeleton of $C$.
We fix an isometry between $G$ and the skeleton of $C$, inducing an isomorphism $j \colon \pi_{1}(G) \xrightarrow[]{\sim} \pi_{1}(C)$.
Let us denote by $y$ the point of $\calS_{g,a}$ whose underlying Schottky group is $\pi_{1}(C)$, with marking given by the image of the basis of $F_g$ under the composition of isomorphisms $F_g\xrightarrow[]{m} \pi_{1}(G) \xrightarrow[]{j} \pi_1(C)$.
Then $\phi_1(y)=[C]$ and $\phi_2(y)=[G,m]$. 
Hence $\phi$ is surjective.
%
%
%
%
\end{proof}

\begin{remark}\label{rem:marking}
In the proof of the surjectivity of the continuous function $\phi$ above, a different choice of isometry $j$ between $G$ and the fundamental group $\pi_1(C)$ might determine a different preimage in $\calS_{g,a}$ of the pair $\big([G,m],[C]\big)$.
For example, when $G$ is a rose with $g$ loops all of the same length, a permutation of the loops corresponds to a permutation of the basis $\{\gamma_1, \dots, \gamma_g\}$ of the Schottky group $\pi_1(C)$.
In most cases the element of $\Out(F_g)$ corresponding to such a permutation does not stabilize a point in $\calS_{g,a}$, for instance when two distinct elements $\gamma_i \neq \gamma_j$ have distinct multipliers $\beta_i\neq\beta_j$.
This shows in particular that the function $\phi$ is not injective.
\end{remark}

\begin{remark}\label{rmk:Martin}
In \cite{Ulirsch21}, Ulirsch constructs a non-archimedean analogue of Teichm\"uller space $\overline{\calT}_g$, using the tropical Teichm\"uller space that Chan, Melo, and Viviani introduced in \cite{ChanMeloEtAl13} and tools from logarithmic geometry.
The space $\overline{\calT}_g$ is a Deligne-Mumford analytic stack over a non-archimedean algebraically closed field $k$ whose points morally correspond to pairs $(C,\phi)$ consisting of a stable projective curve $C$ over a valued extension of $k$ and an isomorphism $\phi:\pi_1^{\mathrm{top}}(C^{\mathrm{an}}) \cong F_{b(C)}$, where $b(C)$ is the first Betti number of $C^{\mathrm{an}}$.
When restricting this construction on the locus of Mumford curve, one retrieves a space $\calT_g^{Mum}$, and a corollary of Ulirsch's construction is the realization of $CV_g$ as a strong deformation retract of $\calT_g^{Mum}$.
Moreover, the fibered product $CV_g' \times_{M_g^{\mathrm{trop}}} \Mumf_{g,a}$ is identified (after a suitable base-change to an algebraically closed field) with the locus of Mumford curves inside the coarse moduli space of $\overline{\calT}_g$.
As a result, Theorem \ref{thm:Schottky-outer-space} and Remark \ref{rem:marking} clarify the relationship between non-archimedean fibers of the Schottky space $\calS_g$ over $\Z$ and Ulirsch's $\calT_g^{Mum}$.
\end{remark}

\newpage

\bibliographystyle{alpha}
\bibliography{../MumfordBiblio}

\begin{thebibliography}{{Fav}20}

\bibitem[AB60]{AhlforsBers60}
Lars~V. {Ahlfors} and Lipman {Bers}.
\newblock {Riemann's mapping theorem for variable metrics}.
\newblock {\em {Ann. Math. (2)}}, 72:385--404, 1960.

\bibitem[{Ber}75]{Bers75}
Lipman {Bers}.
\newblock {Automorphic forms for Schottky groups}.
\newblock {\em {Adv. Math.}}, 16:332--361, 1975.

\bibitem[Ber90]{Berkovich90}
Vladimir~G. Berkovich.
\newblock {\em {Spectral theory and analytic geometry over non-{A}rchimedean
  fields}}, volume~33 of {\em {Mathematical Surveys and Monographs}}.
\newblock American Mathematical Society, Providence, RI, 1990.

\bibitem[Ber93]{Berkovich94}
Vladimir~G. Berkovich.
\newblock \'{E}tale cohomology for non-{A}rchimedean analytic spaces.
\newblock {\em Inst. Hautes \'Etudes Sci. Publ. Math.}, (78):5--161 (1994),
  1993.

\bibitem[Ber09]{BerkovichW0}
Vladimir~G. Berkovich.
\newblock A non-{A}rchimedean interpretation of the weight zero subspaces of
  limit mixed {H}odge structures.
\newblock In {\em Algebra, arithmetic, and geometry: in honor of {Y}u. {I}.
  {M}anin. {V}ol. {I}}, volume 269 of {\em Progr. Math.}, pages 49--67.
  Birkh\"auser Boston Inc., Boston, MA, 2009.

\bibitem[BGR84]{BGR}
Siegfried Bosch, Ulrich G{\"u}ntzer, and Reinhold Remmert.
\newblock {\em Non-{A}rchimedean analysis}, volume 261 of {\em Grundlehren der
  Mathematischen Wissenschaften}.
\newblock Springer-Verlag, Berlin, 1984.
\newblock A systematic approach to rigid analytic geometry.

\bibitem[BJ17]{BoucksomJonsson}
S\'{e}bastien Boucksom and Mattias Jonsson.
\newblock Tropical and non-{A}rchimedean limits of degenerating families of
  volume forms.
\newblock {\em J. \'{E}c. polytech. Math.}, 4:87--139, 2017.

\bibitem[BMV11]{BrannettiMeloEtAl11}
Silvia Brannetti, Margarida Melo, and Filippo Viviani.
\newblock On the tropical {T}orelli map.
\newblock {\em Advances in Mathematics}, 226(3):2546--2586, 2011.

\bibitem[Bou71]{BourbakiTG14}
Nicolas Bourbaki.
\newblock {\em \'{E}l\'{e}ments de math\'{e}matique. {T}opologie
  g\'{e}n\'{e}rale. {C}hapitres 1 \`a 4}.
\newblock Hermann, Paris, 1971.

\bibitem[Cap13]{Caporaso13}
Lucia Caporaso.
\newblock Algebraic and tropical curves: comparing their moduli spaces.
\newblock In G.~Farkas and I.~Morrison, editors, {\em Handbook of Moduli,
  Volume I.}, pages 119--160. Advanced Lectures in Mathematics, Volume XXIV,
  2013.

\bibitem[{Car}57]{CartanQuotient}
Henri {Cartan}.
\newblock {Quotient d'un espace analytique par un groupe d'automorphismes}.
\newblock {\em {Princeton Math. Ser.}}, 12:90--102, 1957.

\bibitem[CMV13]{ChanMeloEtAl13}
Melody Chan, Margarida Melo, and Filippo Viviani.
\newblock Tropical {T}eichm{\"u}ller and {S}iegel spaces.
\newblock In {\em Algebraic and combinatorial aspects of tropical geometry},
  volume 589 of {\em Contemp. Math.}, pages 45--85. Amer. Math. Soc.,
  Providence, RI, 2013.

\bibitem[CV86]{CullerVogtmann86}
Marc Culler and Karen Vogtmann.
\newblock Moduli of graphs and automorphisms of free groups.
\newblock {\em Invent. Math.}, 84(1):91--119, 1986.

\bibitem[DKY20]{DKY}
Laura {DeMarco}, Holly {Krieger}, and Hexi {Ye}.
\newblock {Uniform Manin-Mumford for a family of genus 2 curves}.
\newblock {\em {Ann. Math. (2)}}, 191(3):949--1001, 2020.

\bibitem[dSG10]{SaintGervaisBook}
Henri~Paul de~Saint-Gervais.
\newblock {\em Uniformisation des surfaces de {R}iemann. {R}etour sur un
  th{\'e}or{\`e}me centenaire}.
\newblock ENS {\'E}ditions, 2010.
\newblock 544 pp.

\bibitem[Duc]{DucrosRSS}
Antoine Ducros.
\newblock La structure des courbes analytiques.
\newblock Manuscript available at
  \url{http://www.math.jussieu.fr/~ducros/trirss.pdf}.

\bibitem[{Fav}20]{FavreEndomorphisms}
Charles {Favre}.
\newblock {Degeneration of endomorphisms of the complex projective space in the
  hybrid space}.
\newblock {\em {J. Inst. Math. Jussieu}}, 19(4):1141--1183, 2020.

\bibitem[Gar87]{Gardiner87}
Frederik~P. Gardiner.
\newblock {\em Teichm{\"u}ller theory and quadratic differentials}.
\newblock John {W}iley {\&} {S}ons, New York, 1987.

\bibitem[Ger74]{Gerritzen74}
Lothar Gerritzen.
\newblock Zur nichtarchimedischen {U}niformisierung von {K}urven.
\newblock {\em Math. Ann.}, 210:321--337, 1974.

\bibitem[Ger81]{Gerritzen81}
Lothar Gerritzen.
\newblock Zur analytischen {B}eschreibung des {R}aumes der
  {S}chottky-{M}umford-{K}urven.
\newblock {\em Math. Ann.}, 255(2):259--271, 1981.

\bibitem[Ger82]{Gerritzen82}
Lothar Gerritzen.
\newblock $p$-adic {T}eichm{\"u}ller space and {S}iegel halfspace.
\newblock {\em Groupe de travail d'analyse ultram\'etrique}, 9(2), 1981-1982.
\newblock talk:26.

\bibitem[GH88]{GerritzenHerrlich88}
L.~Gerritzen and F.~Herrlich.
\newblock The extended {S}chottky space.
\newblock {\em J. reine angew. Math.}, 389:190--208, 1988.

\bibitem[GvdP80]{GerritzenPut80}
Lothar Gerritzen and Marius van~der Put.
\newblock {\em Schottky groups and {M}umford curves}, volume 817 of {\em
  Lecture Notes in Mathematics}.
\newblock Springer, Berlin, 1980.

\bibitem[Har88]{HarbaterGaloisCovers}
David Harbater.
\newblock Galois covers of an arithmetic surface.
\newblock {\em Amer. J. Math.}, 110(5):849--885, 1988.

\bibitem[Hej75]{Hejhal75}
Dennis~A Hejhal.
\newblock On {S}chottky and {T}eichm\"uller spaces.
\newblock {\em Advances in Mathematics}, 15(2):133 -- 156, 1975.

\bibitem[Her84]{Herrlich84}
Frank Herrlich.
\newblock On the stratification of the moduli space of {M}umford curves.
\newblock {\em Groupe de travail d'analyse ultram\'etrique}, 11:1--10,
  1983-1984.

\bibitem[Her91]{Herrlich91}
Frank Herrlich.
\newblock Moduli for stable marked trees of projective lines.
\newblock {\em Mathematische Annalen}, 291(1):643--661, 1991.

\bibitem[{Hid}05]{Hidalgo05}
Rub\'en~A. {Hidalgo}.
\newblock {Automorphism groups of Schottky type}.
\newblock {\em {Ann. Acad. Sci. Fenn., Math.}}, 30(1):183--204, 2005.

\bibitem[HS07]{HerrlichSchmithuesen07}
Frank Herrlich and Gabriela Schmith{\"u}sen.
\newblock On the boundary of {T}eichm{\"u}ller disks in {T}eichm{\"u}ller and
  in {S}chottky space.
\newblock In {\em Handbook of {T}eichm{\"u}ller theory I}, pages 293--349.
  European Mathematical Society, 2007.

\bibitem[Ich00]{Ichikawa00}
Takashi Ichikawa.
\newblock Generalized {T}ate curve and integral {T}eichm\"uller modular forms.
\newblock {\em Amer. J. Math.}, 122(6):1139--1174, 2000.

\bibitem[Ich01]{Ichikawa01}
Takashi Ichikawa.
\newblock Universal periods of hyperelliptic curves and their applications.
\newblock {\em Journal of Pure and Applied Algebra}, 163(3):277--288, 2001.

\bibitem[Ich18]{Ichikawa18}
Takashi Ichikawa.
\newblock Klein's formulas and arithmetic of {T}eichm{\"u}ller modular forms.
\newblock {\em Proceedings of the American Mathematical Society},
  146(12):5105--5112, sep 2018.

\bibitem[Ich20a]{Ichikawa20a}
Takashi Ichikawa.
\newblock Periods of generalized {T}ate curves.
\newblock arXiv, 2020.
\newblock \url{https://arxiv.org/abs/1909.10149}.

\bibitem[Ich20b]{Ichikawa20}
Takashi Ichikawa.
\newblock The universal {M}umford curve and its periods in arithmetic formal
  geometry.
\newblock arXiv, 2020.
\newblock \url{https://arxiv.org/abs/2010.11517}.

\bibitem[JMM79]{JorgensenMardenEtAl79}
T.~J{\o}rgensen, A.~Marden, and B.~Maskit.
\newblock {The boundary of classical Schottky space}.
\newblock {\em Duke Mathematical Journal}, 46(2):441--446, 1979.

\bibitem[LP20]{LemanissierPoineau20}
Thibaud Lemanissier and J\'er\^ome Poineau.
\newblock Espaces de {B}erkovich sur $\mathbf{Z}$ : cat\'egorie, topologie,
  cohomologie.
\newblock arXiv, 2020.
\newblock \url{https://arxiv.org/abs/2010.08858}.

\bibitem[Mar74]{Marden74}
Albert Marden.
\newblock The geometry of finitely generated kleinian groups.
\newblock {\em Ann. of Math.}, 99:383--462, 1974.

\bibitem[MSW15]{IndrasPearls}
David Mumford, Caroline Series, and David Wright.
\newblock {\em Indra's pearls. The vision of Felix Klein.}
\newblock Cambridge University Press, Cambridge, 2015.

\bibitem[Mum72]{Mumford72}
David Mumford.
\newblock An analytic construction of degenerating curves over complete local
  rings.
\newblock {\em Compositio Math.}, 24:129--174, 1972.

\bibitem[Poi10a]{A1Z}
J\'er\^ome Poineau.
\newblock La droite de {B}erkovich sur~$\mathbf{Z}$.
\newblock {\em Ast\'erisque}, (334):xii+284, 2010.

\bibitem[Poi10b]{Raccord}
J{\'e}r{\^o}me Poineau.
\newblock Raccord sur les espaces de {B}erkovich.
\newblock {\em Algebra Number Theory}, 4(3):297--334, 2010.

\bibitem[Poi13]{Poineau13}
J\'{e}r\^{o}me Poineau.
\newblock Espaces de {B}erkovich sur $\mathbf{Z}$~: \'{e}tude locale.
\newblock {\em Invent. Math.}, 194(3):535--590, 2013.

\bibitem[PT21]{VIASMII}
J\'er\^ome {Poineau} and Daniele {Turchetti}.
\newblock {Berkovich curves and Schottky uniformization. II: Analytic
  uniformization of Mumford curves}.
\newblock In {\em {Arithmetic and geometry over local fields. VIASM 2018. Based
  on lectures given during the program ``Arithmetic and geometry of local and
  global fields'', summer 2018, Hanoi, Vietnam}}, pages 225--279. Cham:
  Springer, 2021.

\bibitem[Ser77]{Serre77}
Jean-Pierre Serre.
\newblock {\em Arbres, amalgames, {${\rm SL}_{2}$}}.
\newblock Ast\'{e}risque, No. 46. Soci\'{e}t\'{e} Math\'{e}matique de France,
  Paris, 1977.
\newblock Avec un sommaire anglais, R\'{e}dig\'{e} avec la collaboration de
  Hyman Bass.

\bibitem[Uli21]{Ulirsch21}
Martin Ulirsch.
\newblock A non-{A}rchimedean analogue of {T}eichm{\"u}ller space and its
  tropicalization.
\newblock {\em Selecta Mathematica}, 27(3):39, 2021.

\bibitem[Vog14]{Vogtmann14}
Karen Vogtmann.
\newblock On the geometry of outer space.
\newblock {\em Bulletin of the American Mathematical Society}, 52:27--46, 2014.

\bibitem[Wer16]{Werner16}
Annette Werner.
\newblock Analytification and tropicalization over non-{A}rchimedean fields.
\newblock In Matt Baker and Sam Payne, editors, {\em Nonarchimedean and
  Tropical Geometry}, pages 145--171. Springer, 2016.

\end{thebibliography}

\end{document}